\documentclass{amsart}
\usepackage{amsfonts,amscd, amssymb}

\newtheorem{theorem}{Theorem}[section]
\newtheorem{lemma}[theorem]{Lemma}
\newtheorem{corollary}[theorem]{Corollary}
\newtheorem{proposition}[theorem]{Proposition}
\theoremstyle{remark}

\theoremstyle{definition}

\newtheorem{example}[theorem]{Example}
\numberwithin{equation}{section}
\makeatother

\DeclareMathOperator{\Cdb}{{\mathbb C}}
\DeclareMathOperator{\Rdb}{{\mathbb R}}
\DeclareMathOperator{\Zdb}{{\mathbb Z}}

\DeclareMathOperator{\Ndb}{{\mathbb N}}

\begin{document}

\title[Positivity and approximate identities in Banach algebras]{Real 
positivity and approximate identities
in Banach algebras}
\thanks{The first author was supported by a grant from the NSF.   The second author was supported by JSPS KAKENHI Grant Number 26400114.  Some of this material was presented at the 7th Conference on Function Spaces, May 2014, and at the AMS National meeting in January 2015.} 
\author{David P. Blecher}
\address{Department of Mathematics, University of Houston, Houston, TX
77204-3008, USA}
\email[David P. Blecher]{dblecher@math.uh.edu}

\author{Narutaka Ozawa}
\address{Research Institute for Mathematical Sciences, Kyoto University, Kyoto 606-8502,
JAPAN}
\email[Narutaka Ozawa]{narutaka@kurims.kyoto-u.ac.jp}

\begin{abstract}  Blecher and Read have recently introduced and studied a new 
notion of positivity in operator algebras, with an eye to
extending certain $C^*$-algebraic results and theories to more  general algebras.     In the present paper we generalize some part 
of this, and some other facts,  to larger classes 
of Banach algebras.    
  \end{abstract}

\maketitle

\section{Introduction}

An {\em operator algebra} is a closed subalgebra of $B(H)$, for a
complex Hilbert space $H$.  Blecher and Read  recently introduced and studied a new notion of positivity in operator algebras \cite{BRI,BRII,Bord,Read} (see also \cite{BNI,BNII,BBS,BHN}), with an eye to
extending certain $C^*$-algebraic results and theories to more  general algebras.    Over the last several years we have mentioned in lectures on this work that most of  the results of those papers make sense for bigger classes of Banach algebras, and that many of the 
tools and techniques exist there.   In the present paper we initiate this direction.  Thus we 
generalize a number of the main results from the series of papers mentioned above, and some
other facts, to a larger class
of Banach algebras.
 In the process we give simplifications of several facts in these earlier papers.      We will also point out 
some of the main results from the series of papers mentioned above
which do not seem to generalize, or  are  less tidy if they do.   (We will not spend much time  discussing aspects
from that series concerning 
noncommutative peak interpolation, or generalizations of noncommutative topology such as the 
noncommutative Urysohn lemma; these seem unlikely to generalize much farther.)

  Before we proceed we make an editorial/historical note:  The preprint \cite{BRIII},
 which contains many of the basic ideas and facts which we use here, 
 has been split into several papers, which have each taken on 
a life of their own (e.g.\ \cite{Bord} which focuses on operator algebras,
and the present paper in the setting of Banach algebras).     

In this paper we are 
interested in Banach algebras $A$ (over the complex field) with a bounded  approximate
identity (bai).   In fact often there will be a contractive approximate
identity (cai), and in this case we call $A$ an  {\em approximately unital}
Banach algebra.   A Banach algebra with an identity of norm $1$ will be called {\em unital}.  Most of our results are 
stated for approximately unital algebras.   Frequently this is simply because algebras in this class 
have an especially nice `multiplier unitization' $A^1$,  defined below, and a large portion of our constructs are defined in terms of $A^1$.    Also   approximately unital algebras constitute a strong platform for the simultaneous
generalization of as much as possible from the series of papers alluded to above, e.g.\ \cite{BHN,BRI,BRII,Bord}.
However, as one might expect, for algebras without any kind of 
approximate identity it is easy to derive variants of a large portion of 
our results (namely,  almost all of Sections 3, 4, and 7), 
by viewing the  algebra as a subalgebra of a unital Banach algebra (any unitization for example).
We will discuss this point in more detail in the final Section 9 and in a
forthcoming conference proceedings survey article \cite{B2015}.  

 Indeed many of our results are stated for special classes of Banach algebras, for example for Banach algebras with a sequential 
cai, or which are Hahn-Banach smooth in a sense defined later.  Several of the results are sharper 
for {\em $M$-approximately unital Banach algebras}, which means that $A$ is an $M$-ideal in 
its multiplier unitization $A^1$ (see Section 2).
This is equivalent to saying that $A$ is approximately unital 
and for all $x \in A^{**}$ we have $\Vert 1 - x \Vert_{(A^1)^{**}} = \max 
\{ \Vert e - x \Vert_{A^{**}} , 1 \}$.   Here $e$ is the identity  for $A^{**}$ if it has one (otherwise 
it is a `mixed identity' of norm 1--see below for the definition of this).
However as will be seen from the proofs, some of
the results involving the $M$-approximately unital hypothesis will work under weaker assumptions, for example,  {\em strong 
proximinality} of $A$ in $A^1$ at $1$ (that is, given $\epsilon > 0$ there exists a $\delta > 0$
such that if  $y \in A$ with $\Vert 1 - y \Vert < 1 + \delta$ then there is a $z \in A$
with $\Vert 1 - z \Vert = 1$ and $\Vert y - z \Vert < \epsilon$).  

We now outline the structure of this paper, describing each section briefly.  Because our paper 
is rather diverse, to help the readers focus we will also mention at least 
one highlight from each section.  In Section 2 we discuss unitization and states, and also  introduce some  classes
of Banach algebras.   A key result in this section
ensures the existence of a `real positive' cai in Banach algebras with a countable cai satisfying a reasonable extra condition.   We also characterize this extra condition, and the related property that the quasi-state space be weak* closed and convex.  In the latter setting
by the bipolar theorem there exists a `Kaplansky density theorem'.  (Conversely, such a density result often 
immediately gives a real positive approximate identity by weak* approximating an identity in the bidual
 by real positive elements in $A$, and using e.g.\ Lemma \ref{netwee} below.)  
 Section 3 starts by generalizing many of the 
basic ideas from the Blecher-Read papers cited above involving cai's, roots, and positivity.
With these in place, we  give several applications of the kind found in those 
papers, for example we characterize when $xA$ is closed in terms of `generalized invertibility' of  the
`real positive'  element $x$; and show that 
these are the right ideals $qA$ for a `real positive' idempotent $q$ in $A$.
We also list several examples illustrating some of the
 things from the cited series of papers that will break down without further restrictions on  the
class of Banach algebras considered.   The main advance in Section 4 is the introduction 
of the concept of {\em hereditary subalgebra} (HSA), an important tool 
in $C^*$-algebra theory, to  Banach algebras, and  establishing the basics of their
theory.  In  particular we study 
the relationship between HSA's and 
  one-sided ideals
with one-sided approximate identities.   
Some aspects of this relationship 
is problematic for general Banach algebras, but 
it works much better in 
separable algebras, as we shall see.   We characterize the HSA's, and the associated 
class of one-sided ideals, as increasing unions of `principal' ones; and indeed in the separable case
they are exactly the `principal' ones.   Indeed it is obvious that in a Banach algebra 
 $A$ every closed right ideal with a `real positive' 
left bai is of the form $\overline{EA}$ for a set $E$ of real positive elements of $A$.   
Section 4 contains an Aarnes-Kadison type 
theorem for Banach algebras, and related results that use the Cohen's factorization
proof technique.  Some similar results and ideas have been found by Sinclair
(in [Sinclair, 1978] for example), but these are somewhat different, and were not
directly connected to `positivity'.   It is interesting though that Sinclair was
inspired by papers of Esterle based on the Cohen's factorization proof technique, and
one of these does have some connection to our notion of positivity \cite{Est}.

In Section 5 we consider the better behaved class
of  $M$-approximately unital  Banach algebras.  The main result here
is the  generalization of  Read's theorem from \cite{Read} 
to this class.  That is, such algebras have  cai's $(e_t)$ satisfying $\Vert 1 - 2 e_t \Vert \leq 1$.
This may be  the class  to which the most 
results from our previous operator algebra papers will generalize, as we shall see at points throughout our paper.
In Section 6 we show that basic
aspects and notions from the classical theory of ordered linear spaces correspond to interesting facts about
our `positivity', for our various classes of approximately unital  Banach algebras (for example,
for $M$-approximately unital algebras, or certain algebras with a sequential cai).   Indeed 
the  highlight of this section
is the revealing of interesting connections  between Banach algebras and this classical ordered linear theory (see also
\cite{Bord} for more, and clearer, such connections if the algebras are in addition operator algebras).
In the  process we 
generalize several basic facts about
$C^*$-algebras.   For example we give the aforementioned variant of Kaplansky's density theorem, and variants
of several well known order-theoretic properties 
of the unit ball of a $C^*$-algebra and its dual.

In Sections 7 and 8 we find variants for 
approximately unital  Banach algebras of  several other results about two-sided ideals from  \cite{BRI,BRII,Bord}. 
In Section 7 we assume that  $A$ is commutative, and in this case 
 we are able to establish the converse of the last result mentioned in our description of 
Section 4 above.    Thus closed ideals having a `real positive' bai, in a commutative Banach algebra $A$,
are precisely the spaces 
$\overline{EA}$ for sets $E$ of real positive elements of $A$.
In Section 8 we only consider ideals that are $M$-ideals in $A$
(this does generalize the operator algebra case at least for two-sided ideals, since the closed two-sided ideals with cai in an operator algebra
are exactly the $M$-ideals \cite{ER}).   The lattice theoretic properties
of such ideals behaves considerably more like the $C^*$-algebra case, and is related to faces in the quasi-state space.
Section 8 may be considered to be a continuation of the study of $M$-ideals in Banach algebras
initiated  in \cite{SW1,SW2,SW3} and e.g.\ \cite[Chapter V]{HWW}.    At the end of this section
we give  a `noncommutative peak interpolation' result reminiscent of Tietze's extension theorem, which is based on a remarkable result of Chui-Smith-Smith-Ward \cite{CSSW}.
This solves an open problem from \cite{BRIII} or earlier concerning real positive elements in a quotient.  
 Finally, in Section 9 we discuss which results from earlier sections
generalize to algebras without a cai; more details on this are given in \cite{B2015}.
The latter is a survey article which also contains a few additional
details on some of the material in the present paper, as well as some
small improvements found after this paper was in press.

We now list some of our notation and general facts: We write ${\rm Ball}(X)$ for the set $\{ x \in X : \Vert x \Vert \leq 1 \}$.  If $E, F$ are sets then 
$EF$ denotes the span of products $xy$ for $x \in E, y \in F$.  If $x \in A$
for a Banach algebra $A$, then ba$(x)$ denotes the closed subalgebra generated
by $x$.
  For two spaces $X, Y$ which are in duality,
for a subset $E$ of $X$ we use the polar  $E^\circ = \{ y \in Y : \langle x, y \rangle \geq -1 \;
\text{for all} \; x \in E \}$.

For us Banach algebras satisfy $\Vert xy \Vert \leq \Vert x \Vert \Vert y \Vert$.  We recall that a nonunital Banach algebra  $A$ is Arens regular iff its unitization is Arens regular (any unitization
will do here). 
   In the rest of this paragraph we consider an Arens regular 
approximately unital  Banach algebra $A$.   For such an algebra we 
will always write $e$ for the unique identity of $A^{**}$.  Indeed if $A$ is an Arens regular Banach algebra with cai $(e_t)$,
and $e_{t_\mu} \to \eta$ weak* in $A^{**}$, then $e_{t_\mu} a
\to \eta a$ weak* for all $a \in A$.  So $\eta a = a$, and similarly $a \eta = a$.
Therefore $\eta$ is the unique identity $e$ of $A^{**}$,
and $e_t  \to e$ weak*.    We will show at the end of this section that
the `multiplier unitization'  $A^1$ is  isometrically isomorphic to the subalgebra
$A + \Cdb e$ of $A^{**}$.
  
  If $A$ is a Banach algebra which is not Arens regular, then the multiplication we usually use
on $A^{**}$ is the  `second Arens product' ($\diamond$ in the notation of \cite{Dal}).
This is weak* continuous in the second variable.
 If $A$ is a nonunital, not necessarily Arens regular,  Banach algebra with a bai, 
then $A^{**}$ has a so-called `mixed identity'  \cite{Dal,Pal,DW}, which we will again write as $e$.  This is a right  identity for 
the first Arens product, and a left identity for 
the second Arens product.   A mixed identity need not
be unique, indeed  mixed identities are just the weak* limit points of bai's for $A$. 

 We will also use the theory of $M$-ideals.  These
were invented by Alfsen and Effros,
and  \cite{HWW} is the basic text for their theory.  We recall a subspace $E$ of a Banach space $X$ is an $M$-ideal in
$X$ if $E^{\perp \perp}$ is complemented in $X^{**}$ via a contractive projection $P$ so that 
$X^{**} = E^{\perp \perp} \oplus^{\infty} {\rm Ker}(P)$.   In this case there is a unique contractive projection  onto $E^{\perp \perp}$.
$M$-ideals have many beautiful properties, some of which will be mentioned below.

We will need the following result several  times:

\begin{lemma} \label{ban}  Let $X$ be a Banach space, and suppose that $(x_t)$ is a bounded net in $X$ with
$x_t \to \eta$ weak* in $X^{**}$.  Then $$\Vert \eta \Vert = \lim_t \, \inf \{ \Vert y \Vert : y \in {\rm conv} \{ x_j : j \geq t \} \} .$$
    \end{lemma}

\begin{proof}     It is easy to see that $\Vert \eta \Vert \leq \lim_t \, \inf \{ \Vert y \Vert : y \in {\rm conv} \{ x_j : j \geq t \} \} ,$
for example by using the weak*-semicontinuity of the norm, and noting that for every $t$ and any choice $y_t \in {\rm conv} \{ x_j : j \geq t \}$,
we have $y_t \to \eta$ weak*.   By way of contradiction suppose that $$\Vert \eta \Vert < C < \lim_t \, \inf \{ \Vert y \Vert : y \in {\rm conv} \{ x_j : j \geq t \} \} .$$
Then there exists $t_0$ such that the norm closure of ${\rm conv} \{ x_j : j \geq t \} \}$ is disjoint from $C {\rm Ball}(X)$, for all $t \geq t_0$.  
By the Hahn-Banach theorem there exists $\varphi \in X^*$ with $$C \Vert \varphi \Vert < K < {\rm Re} \, \varphi(x_j) , \qquad j \geq t,$$
so that $C \Vert \varphi \Vert < K  \leq {\rm Re} \, \varphi(\eta).$   This contradicts $\Vert \eta \Vert < C$.  
\end{proof}

Any nonunital operator algebra has a unique operator algebra  unitization (see \cite[Section 2.1]{BLM}), but of course 
this is not true for Banach algebras.    We will choose to use the unitization that typically has the 
smallest norm among all unitizations, and which we now describe.    If $A$ is an approximately unital  Banach algebra, then 
the left regular representation embeds $A$ isometrically in $B(A)$.  
We will always write $A^1$ for the {\em multiplier unitization} of $A$,
that is, we identify $A^1$ isometrically with $A + \Cdb I$ in
$B(A)$.   For $a \in A , \lambda \in \Cdb$ we have
$$\Vert a  + \lambda 1 \Vert = \sup \{ \Vert a c + \lambda c \Vert : c \in {\rm Ball}(A) \} = \sup_t  \, \Vert  a e_t+ \lambda e_t \Vert
= \lim_t \, \Vert  a e_t+ \lambda e_t \Vert ,$$     
by e.g.\ \cite[A.4.3]{BLM}.  If $A$ is actually nonunital then the map $\chi_0(a  + \lambda 1)  = \lambda $ on $A^1$ is contractive, as is any character
on a Banach algebra.  We call this the {\em trivial character}.
  Below $1$ will almost always denote the identity of $A^1$, if $A$ is not already unital.   Note that the 
multiplier unitization also makes sense for the so-called {\em self-induced} Banach algebras, namely those for
which the   left regular representation embeds $A$ isometrically in $B(A)$.

If $A$ is a nonunital, approximately unital  Banach algebra
then the multiplier  unitization $A^1$ may also be identified with a subalgebra of $A^{**}$.  
Indeed if $e$ is a `mixed identity' of norm $1$ for $A^{**}$ then  $A + \Cdb e$ is then a unitization of $A$ (by 
basic facts about the Arens product).   To see that this is isometric to
$A^1$ above note that  for any $c \in  {\rm Ball}(A), a \in A, \lambda \in \Cdb$ we have 
$$\Vert a c + \lambda c \Vert \leq \Vert a  + \lambda e \Vert_{A^{**}} =  \Vert e(a  + \lambda 1) \Vert_{(A^1)^{**}} 
 \leq \Vert a  + \lambda 1 \Vert_{A^1}.$$  Thus by the displayed equation in the last paragraph
$\Vert a  + \lambda e \Vert_{A^{**}}  = \Vert a  + \lambda 1 \Vert_{A^1}$
as desired.

\section{Unitization and states}  
  
If $A$ is an approximately unital  Banach algebra,
 then  we may view 
$A$ in its multiplier unitization $A^1$, and write $${\mathfrak F}_A = \{ a \in A : \Vert 1 - a \Vert
\leq 1 \} = \{ a \in A : \Vert e - a \Vert
\leq 1 \},$$  where $e$ is as in the last paragraph (or set $e=1$ if $A$ is unital).   So   $$\frac{1}{2} {\mathfrak F}_A = \{ a \in A : \Vert 1 - 2 a \Vert
\leq 1 \}.$$    If $x \in \frac{1}{2} {\mathfrak F}_A$ then $x, 1-x \in {\rm Ball}(A^1)$.
Also,  ${\mathfrak F}_A =  {\mathfrak F}_{A^1}  \cap A$, and ${\mathfrak F}_A$ is closed under the quasiproduct $a+b - ab$.
(It is interesting that cones containing ${\mathfrak F}_A$ were used to obtain nice results about `order' in unital Banach algebras and their
duals 
in Section 1 of the historically important paper \cite{KV}, based on a 1951 ICM talk.  Slightly earlier ${\mathfrak F}_A$ also appeared in a Memoir by Kadison.) 

If $\eta \in A^{**}$ then an expression such as $\lambda 1 + \eta$ will
usually need to be interpreted as an element of $(A^1)^{**}$,
with $1$ interpreted as the identity for $A^1$ and $(A^1)^{**}$.   Thus
$\Vert 1 - \eta \Vert$ denotes $\Vert 1 - \eta \Vert_{(A^1)^{**}}$.  
We define $${\mathfrak F}_{A^{**}} = \{ \eta \in A^{**} : \Vert 1 - \eta \Vert \leq 1 \} = A^{**} \cap {\mathfrak F}_{(A^1)^{**}}.$$  
We write ${\mathfrak r}_A$ for the set of $a \in A$ whose 
numerical range in $A^1$ is contained in the right half plane.
That is,  $${\mathfrak r}_A = \{ a \in A : {\rm Re} \, \varphi(a) \geq 0 \; \text{for all} \;
\varphi \in S(A^1) \} ,$$  where
$S(A^1)$ denotes the states on $A^1$.    
Note that ${\mathfrak r}_A$
is a closed cone in $A$, but it is not 
proper (hence is what is sometimes called a {\em wedge}).   We write $a \preceq b$ if $b-a \in 
{\mathfrak r}_A$.     It is easy to see that $\Rdb^+   {\mathfrak F}_A \subset {\mathfrak r}_A$.   Conversely,
if $A$ is a unital Banach algebra and $a \in   {\mathfrak r}_A$ then $a + \epsilon 1 \in \Rdb^+   {\mathfrak F}_A$
for every $\epsilon > 0$.   Indeed $a + \epsilon 1 \in C \, {\mathfrak F}_A$ where $C = \frac{\Vert a \Vert^2}{\epsilon} +  \epsilon$, as can be easily seen from the well known fact that the numerical range of $a$ is contained in the right half plane iff $\Vert 1 - t a \Vert \leq 1 + t^2 \Vert a \Vert^2$ for all $t > 0$ (see e.g.\ 
\cite[Lemma 2.1]{Mag}).

One main reason why we almost always assume that $A$ is  approximately unital in this paper is
that  ${\mathfrak F}_A$ and ${\mathfrak r}_A$ are well defined as above.   However as we said in the introduction,
 if $A$ is not  approximately unital it is easy to see how to proceed in a large number of our
results (namely in almost all of Sections 3, 4, and 7), and this is discussed
briefly in Section 9.    

The following is no doubt in the literature, but we do not know of a reference that proves all
that is claimed.   It follows from it that mixed identities
in $A^{**}$ are just the weak* limits of bai's for $A$, when these 
limits exist.  

\begin{lemma} \label{netwee}  If $A$ is a Banach algebra, and if a bounded net $x_t \in A$ converges weak* to a mixed identity
$e \in A^{**}$, then a bai for $A$ can be found with weak* limit $e$, and formed from convex combinations of the $x_t$.  \end{lemma} 

\begin{proof}   Given $\epsilon > 0$ and a finite set  $F \subset A^*$,
there exists $t_{F, \epsilon}$
such that $$|\varphi(x_t) - e(\varphi)| < \epsilon , \qquad t \geq t_{F, \epsilon}, \; \; \varphi \in F.$$   Given a finite set $E = \{ a_1, \cdots , a_n \}
\subset A$, we have that $x_t a_k \to a_k$ and $a_k x_t \to a_k$ weakly.  So there 
is a convex combination $y$ of the $x_t$ for $t \geq  t_{F, \epsilon}$, with 
$$\Vert y a_k - a_k \Vert + \Vert a_k y - a_k \Vert
\leq \epsilon.$$  We also have $|\varphi(y) - e(\varphi)| \leq \epsilon$  for 
$\varphi \in F$.   Write this $y$ as $y_\lambda$, where $\lambda = (E,F, \epsilon)$.  
Given $\epsilon_0 > 0$ and $a \in A$, if $\epsilon \leq \epsilon_0$ and $\{ a \} \subset 
E$, then $\Vert y_\lambda a - a \Vert + \Vert a y_\lambda - a \Vert \leq \epsilon  \leq \epsilon_0$
for $\lambda = (E,F, \epsilon)$, any $F$.      So $(y_\lambda)$ is a bai.    Also if 
$\varphi \in F$ then $|\varphi(y_\lambda) - e(\varphi)| < \epsilon$.  So $y_\lambda \to e$ weak*. 
\end{proof}

{\bf Remark.}  The `sequential version' of the last result is false.  For example, consider the usual 
cai $(n \, \chi_{[-\frac{1}{2n},\frac{1}{2n}]})$  of $L^1(\Rdb)$ with convolution product.  
A subnet of this converges weak* to a mixed identity
$e \in L^1(\Rdb)^{**}$.   However there can be no weak* convergent sequential bai for  $L^1(\Rdb)$,
since $L^1(\Rdb)$ is weakly sequentially complete.  

\medskip

For a general approximately unital nonunital Banach algebra $A$ with cai $(e_t)$, the definition of `state' is problematic.  There are 
many natural notions, for example: (i) \ a contractive functional
$\varphi$ on $A$ with $\varphi(e_t) \to 1$ for some fixed  cai $(e_t)$ for $A$,
(ii) \   a contractive functional
$\varphi$ on $A$ with $\varphi(e_t) \to 1$ for all cai $(e_t)$ for $A$,  and 
(iii) \ a norm $1$ functional on $A$ that extends to a state on $A^1$,
where $A^1$ is the `multiplier unitization' above.   If $A$ is not Arens regular 
then (i) and (ii) can differ, that is whether 
 $\varphi(e_t) \to 1$ depends on which  cai for $A$ we use.  And if $e$ is a `mixed identity' 
 then the statement $\varphi(e) = 1$ may depend on which mixed identity one considers.
 In this paper though 
for simplicity, and because of its connections with the usual theory of numerical range and accretive operators, 
we will take (iii) above as the definition of a {\em state} of $A$.  We shall also often consider states in the sense
of (i), and will usually ignore (ii) since in some sense it may be treated as a `special case' of (i) (that is, almost all
computations in the paper  involving the class  (i)  are easily tweaked to give the `(ii) version').
 We define  $S(A)$ to be  
the set of states
in the sense of (iii) above.  This is easily seen to be norm closed, but will not be weak* closed
if $A$ is nonunital.    We define $${\mathfrak c}_{A^*} = \{ \varphi \in A^* : 
{\rm Re} \,  \varphi(a) \geq 0 \; {\rm for} \, {\rm all} \; a \in {\mathfrak r}_A \},$$  and
note that this is a weak* closed cone containing $S(A)$.    These are called
the {\em real positive functionals} on $A$.  
 If ${\mathfrak e} = (e_t)$   is a fixed cai for $A$,
 define $$S_{\mathfrak e}(A) = \{ \varphi \in {\rm Ball}(A^*) :  
\lim_t \, \varphi(e_t) = 1 \}$$
(this corresponds to (i) above).   Note that $S_{\mathfrak e}(A)$ is convex but $S(A)$ may not be (as in
e.g.\  Example \ref{Ex4}).  An argument
in  the
next proof   shows that $S_{\mathfrak e}(A) \subset S(A)$.    Finally we remark that for 
any $y \in A$ of norm $1$, if $\varphi \in {\rm Ball}(A^*)$ satisfies 
$\varphi(y) = 1$, then $x \mapsto \varphi(yx)$ is in $S_{\mathfrak e}(A)$ 
for all cai's ${\mathfrak e}$ of $A$.  

We recall that a subspace $E$ of a Banach space $X$ is called `Hahn-Banach smooth' in $X$
if every functional on $E$ has a unique Hahn-Banach extension to $X$.   Any 
$M$-ideal in $X$ is Hahn-Banach smooth in $X$.  See e.g.\ \cite{HWW} and references therein
for more on this topic.

 \begin{lemma} \label{hbs}  For approximately unital Banach 
algebras $A$ which are Hahn-Banach smooth in $A^1$,
and therefore for $M$-approximately unital Banach algebras, and 
$\varphi \in A^*$ with norm $1$, the following are equivalent:
\begin{itemize} 
\item [(i)]   $\varphi$ is a state on $A$ (that is, extends to a state on $A^1$).
\item [(ii)]  $\varphi(e_t) \to 1$ for  every cai  $(e_t)$ for $A$.
\item [(iii)]  $\varphi(e_t) \to 1$ for some cai $(e_t)$ for $A$.
\item [(iv)] $\varphi(e) = 1$ whenever $e \in A^{**}$ is a weak* limit point  of a cai for $A$
(that is, whenever $e$ is a mixed identity of norm $1$ for $A^{**}$).
\end{itemize}
 \end{lemma}

   \begin{proof}  Clearly (ii) implies (iii).
If $\varphi \in {\rm Ball}(A^*)$
write $\tilde{\varphi}$  for its canonical weak* continuous
extension   to $A^{**}$.  If 
$(e_t)$ is a cai for $A$ with weak* limit point $e$ and  $\varphi(e_t) \to 1$, then  $\tilde{\varphi}(e) = 1$.
It follows  that $\tilde{\varphi}_{|A^1}$ is a state on $A^1$.   
So (iii) implies (i).  To see that (i) implies (iv), suppose that
$A$ is  Hahn-Banach smooth
in $A^1$, and that $\varphi$ is a  norm $1$ functional on $A$ that extends to a state $\psi$ on $A^1$.   If $(e_t)$ is a cai for $A$ 
with weak* limit point $e$, then also 
$\tilde{\varphi}_{|A + \Cdb e}$ is a norm $1$ functional extending $\varphi$, so that $\tilde{\varphi}_{|A + \Cdb e}  = \psi$, and 
for some subnet,
$$\varphi(e) = \lim_t \varphi(e_{t_\mu}) = \tilde{\varphi}(e) = \psi(1) = 1.$$  
We leave the remaining implication as an exercise.  \end{proof}

Under certain conditions on  
an approximately unital Banach algebra $A$ we shall see  in Corollary \ref{snr} that $S(A^1)$ is the convex hull of the 
trivial character $\chi_0$ and the set of states on $A^1$ extending states of $A$, and that the weak* closure of $S(A)$ equals
$\{ \varphi_{|A} : \varphi \in S(A^1) \}$.   

The  numerical range $W(a)$ (or $W_A(a)$) of $a \in A$, if $A$ is an approximately unital Banach algebra,  
will be defined to be  
 $\{ \varphi(a) : \varphi \in S(A) \}$.   If $A$ is Hahn-Banach smooth in $A^1$
then it follows from Lemma \ref{hbs} that $S(A)$ is convex, and hence so 
is $W(a)$.  We shall see in Corollary \ref{snr} 
that  under 
the condition mentioned in the last paragraph,
we have  $\overline{W_A(a)} = {\rm conv} \{ 0, W_{A}(a) \} =  W_{A^1}(a)$.   

The following is related to 
results from \cite{SW2} or \cite[Section V.3]{HWW} or \cite{AR,BAIC}.

\begin{lemma} \label{oz2}   If $A$ is an approximately unital Banach algebra, if $A^1$ is the unitization above, and if $e$ 
is a weak* limit of a cai (resp.\ bai in ${\mathfrak F}_A$)  for $A$ 
then $\Vert 1 - 2 e \Vert_{(A^1)^{**}} \leq 1$ iff there is a cai (resp.\ bai  in ${\mathfrak F}_A$) $(e_i)$  with weak* limit 
$e$ and $\limsup_i \, \Vert 1 - 2 e_i \Vert_{A^1} \leq 1.$
\end{lemma}

   \begin{proof}  The one direction follows from  Alaoglu's theorem.  Suppose that 
$\Vert 1 - 2 e \Vert_{(A^1)^{**}} \leq 1$ and there is a net $(x_t)$ which is a cai 
 (resp.\ bai in ${\mathfrak F}_A$)  for $A$ with  $x_t \to e$ weak*.  Then  $1 - 2x_t \to 1 - 2e$ weak* in $(A^1)^{**}$.  By Lemma \ref{ban},
for any $n \in \Ndb$ there exists  a $t_n$ such that
for every $t \geq t_n$, $$\inf \{ \Vert 1 -  2y \Vert :
y \in {\rm conv} \{x_j : j \geq t \} \} < 1 + \frac{1}{2n}.$$
For every $t \geq t_n$, choose such a $y^n_t \in {\rm conv} \{x_j : j \geq t \}$ with 
$\Vert 1 -  2y^n_t \Vert < 1 + \frac{1}{n}$.   If $t$ does not 
dominate $t_n$ define $y^n_t = y^n_{t_n}$.  So for all $t$
we have $\Vert 1 -  2y^n_t \Vert < 1 + \frac{1}{n}$.    
Writing $(n,t)$ as $i$, we may view $(y^n_t)$ as a 
net $(e_i)$ indexed by $i$, with 
$\Vert 1 - 2 y^n_t \Vert  \to 1$.  Given $\epsilon > 0$ and $a_1, \cdots , a_m \in A$,
there exists a $t_1$
such that 
$\Vert x_t a_k - a_k \Vert < \epsilon$ and  $\Vert a_k x_t  - a_k \Vert < \epsilon$ for all 
$t \geq t_1$ and all $k = 1, \cdots, m$.  Hence 
the same assertion is true with $x_t$ replaced by $y^n_t $.   Thus $(y^n_t) = (e_i)$ is a bai for $A$ with the desired
property.     \end{proof}

 We recall from the introduction that if $A$ is an approximately unital Banach algebra which is an $M$-ideal in the particular unitization $A^1$ above, then
$A$ is an $M$-approximately unital Banach algebra.   Any 
 unital Banach algebra is an $M$-approximately unital Banach algebra (here $A^1 = A$).
By \cite[Proposition I.1.17 (b)]{HWW}, examples of $M$-approximately unital Banach algebras include any Banach algebra that is an $M$-ideal in its bidual,
and which is approximately unital (or whose bidual has an identity).  Several examples of such are given in \cite{HWW};
for example the compact operators on $\ell^p$, for $1 < p < \infty$.  We also recall that the property of being an 
$M$-ideal in its bidual is inherited by subspaces, and hence by subalgebras.   Not every Banach algebra with cai is $M$-approximately unital.  By \cite[Proposition II.3.5]{HWW}, $L^1(\Rdb)$ with convolution multiplication cannot be an $M$-ideal in any proper
superspace.

We just said that any unital Banach algebra $A$ is $M$-approximately unital, hence 
any finite dimensional unital Banach algebra is   Arens regular and $M$-approximately unital (if one wishes to avoid 
the redundancy of $A = A^1$ in the discussion below take the direct sum of $A$ with any Arens regular $M$-approximately unital Banach algebra,
such as $c_0$).  
Thus any kind of bad behavior occurring in 
finite dimensional unital Banach algebras (resp.\ unital Banach algebras) will
appear in the class of Arens regular $M$-approximately unital Banach algebras
(resp.\ $M$-approximately unital Banach algebras).  This will have the consequence that 
several aspects of the Blecher-Read papers will not
generalize, for instance conclusions involving `near positivity'. 
This can also be seen in the examples scattered through our paper, for instance Examples \ref{Ex1}--\ref{Ex4} below.

Suppose that $(e_t)$ is a  cai 
for a Banach algebra $A$ with weak* limit point $e \in A^{**}$.  Then left multiplication by $e$ (in the second Arens product)
is a contractive projection from $(A^1)^{**}$ onto the ideal $A^{\perp \perp}$ of $(A^1)^{**}$ (note that $(A^1)^{**} = A^{\perp \perp} + \Cdb  1 = A^{\perp \perp} + \Cdb (1-e)$).      Thus by the theory of $M$-ideals \cite{HWW},  $A$ is an $M$-ideal in $A^1$  iff left multiplication by $e$ is an $M$-projection.

\begin{lemma} \label{hfin}  A nonunital approximately unital Banach algebra   $A$ is $M$-approximately unital  iff 
 for all $x \in A^{**}$ we have $\Vert 1 - x \Vert_{(A^1)^{**}} = \max 
\{ \Vert e - x \Vert_{A^{**}} , 1 \}$.    Here $e$ is a mixed identity for $A^{**}$ of norm $1$.
If these conditions hold then there is a unique mixed identity for $A^{**}$  of norm $1$, it belongs in
$\frac{1}{2} {\mathfrak F}_{A^{**}}$, and $$\Vert 1 - \eta \Vert = 1 \; \; \Leftrightarrow \; \; \Vert e - \eta \Vert \leq 1,  \qquad \eta \in A^{**} .$$
\end{lemma} 

\begin{proof}    
By the statement immediately above the Lemma, 
and by the theory of $M$-ideals \cite{HWW},  $A$ is an $M$-ideal in $A^1$  iff left multiplication by $e$ is an $M$-projection.
That is, iff
$$\Vert \eta + \lambda 1 \Vert_{(A^1)^{**}} = \max \{ \Vert \eta + \lambda e \Vert_{A^{**}}  , \vert \lambda \vert 
\Vert 1 - e \Vert  \}, \qquad \eta \in A^{**}, \lambda \in \Cdb.$$
If this holds then setting $\lambda = 1$ and $\eta = 0$ shows that  $\Vert 1 - e \Vert \leq 1$.   
However by the Neumann lemma  we cannot have $\Vert 1 - e \Vert < 1$.   Thus $\Vert 1 - e \Vert =  1$ if these hold.
The statement is tautological if $\lambda = 0$ so we may assume the contrary.
Dividing by $|\lambda|$ and setting $x = - \frac{\eta}{|\lambda|}$, one sees that 
$A$ is $M$-approximately unital iff 
$$\Vert 1 - x \Vert_{(A^1)^{**}} = \max \{ \Vert e - x \Vert_{A^{**}}  , 1 \}, \qquad x \in A^{**} .$$
In particular, $\Vert 1 - 2e \Vert_{(A^1)^{**}} = \max \{ \Vert  e \Vert , 1 \} = 1.$    The final assertion is now
clear too.    The uniqueness of the mixed identity follows from the next result.  \end{proof}

{\bf Remark.}   Indeed if $B$ is any unitization of a nonunital approximately unital Banach algebra $A$, and if 
$A$ is an $M$-ideal in $B$, then the first few lines of the last proof, with $A^1$ replaced by $B$,
 show that  $B = A^1$, the multiplier unitization of $A$.

\bigskip

Thus $A$ is  $M$-approximately unital  iff  $\Vert 1 - x \Vert_{(A^1)^{**}} = \Vert e - x \Vert_{A^{**}}$ for all $x \in A^{**}$,
unless the last quantity is $< 1$ in which case $\Vert 1 - x \Vert_{(A^1)^{**}} = 1$.

We will show later that for $M$-approximately unital Banach algebras
there is a cai $(e_t)$ for $A$ with $\Vert 1 - 2 e_t \Vert_{A^1} \leq 1$ for all $t$.  

\begin{lemma} \label{isun}  Let $A$ be a
closed ideal,
and also 
an $M$-ideal, in a 
unital Banach algebra $B$.   If $e$ and $f$ are two weak* limit points
in $A^{**}$ of two cai for $A$, then $e = f$.
Thus $A^{**}$ has a unique mixed identity  of norm $1$.   In particular if
$A$ is $M$-approximately unital  then $A^{**}$ has a unique mixed identity  of norm $1$.  
\end{lemma} \begin{proof}  
As in the discussion above Lemma \ref{hfin},
left multiplication by $e$ or $f$, in the second Arens
product, are contractive projections  onto the ideal $A^{\perp \perp}$ of $(A^1)^{**}$.
So these maps  equal the $M$-projection \cite{HWW}, hence are equal.  So $e = f$.
Thus every cai for $A$ converges weak* to $e$, 
so that $A^{**}$ has a unique mixed identity. \end{proof}

 If $A$ is an approximately unital  
Banach algebra, but $A^{**}$ has no identity
then  we define ${\mathfrak r}_{A^{**}} = A^{**} \cap {\mathfrak r}_{(A^1)^{**}}$.  
 If $A$ is an approximately unital Banach algebra then 
${\mathfrak F}_{A^{**}}$ and ${\mathfrak r}_{A^{**}}$ are weak* closed. 
Indeed the ${\mathfrak F}_{A^{**}}$ case of this  is obvious.   By \cite{Mag}, ${\mathfrak r}_{(A^1)^{**}}$ is weak* closed,
hence so is ${\mathfrak r}_{A^{**}} = A^{**} \cap {\mathfrak r}_{(A^1)^{**}}$.   

\bigskip

{\bf Remark.}   Note that if 
$A^{**}$ has a 
mixed identity  of norm $1$ then 
we can define states of $A^{**}$ to be norm $1$ functionals $\varphi$
with $\varphi(e) = 1$ for all  
mixed identities $e$ of $A^{**}$  of norm $1$.
Then one could define ${\mathfrak r}_{A^{**}}$ to be the elements 
$x \in A^{**}$ with Re $\varphi(x) \geq 0$ for all such states of $A^{**}$.  This coincides
with the definition of ${\mathfrak r}_{A^{**}}$ above the Remark
 if $A$ is $M$-approximately unital.
Indeed 
such states $\varphi$ on $A^{**}$ extend to states $\varphi(e \, \cdot)$ of $(A^1)^{**}$.   Conversely
if $A$ is an $M$-approximately unital  
Banach algebra, then 
given a  state $\varphi$ of $(A^1)^{**}$, we have 
$$1 = \Vert  \varphi \Vert = \Vert  \varphi \cdot e \Vert + \Vert  \varphi \cdot (1-e) \Vert \geq
|\varphi(e)| + |\varphi(1-e)| \geq \varphi(1) = 1 = \varphi(e) + \varphi(1-e).$$
It follows from this that $\Vert  \varphi e \Vert = |\varphi(e)| = \varphi(e)$.
 Hence if $\eta \in {\rm Ball}(A^{**})$ then
$$|\varphi(\eta)| =  |\varphi e (\eta)| \leq \Vert  \varphi e \Vert = \varphi(e), $$ so that the restriction of $\varphi$ to $A^{**}$ is either zero
or is a positive
multiple of a state on $A^{**}$.  Thus for $M$-approximately unital  
Banach algebras, the two notions of ${\mathfrak r}_{A^{**}}$ under discussion coincide.

\bigskip

 Let 
$Q(A)$ be the quasi-state space of $A$, namely 
$$Q(A) = \{ t \varphi : t \in [0,1], \varphi \in S(A) \}.$$.   Similarly,
$Q_{{\mathfrak e}}(A) = \{ t \varphi : t \in [0,1], \varphi \in S_{{\mathfrak e}}(A) \}$.   
We set $${\mathfrak r}^{\mathfrak e}_A = \{ x \in A : {\rm Re} \, \varphi (x) \geq 0 \; {\rm for \, all} \; \varphi 
\in S_{\mathfrak e}(A) \} ,$$
and
$${\mathfrak c}^{\mathfrak e}_{A^*} = \{ \varphi \in A^* : \, {\rm Re} \; \varphi(x) \geq 0 \; {\rm for} \, {\rm all} \;
x \in {\mathfrak r}^{\mathfrak e}_A \}. $$   Note that ${\mathfrak r}_A \subset {\mathfrak r}^{\mathfrak e}_A$ since
$S_{\mathfrak e}(A) \subset S(A)$.

\begin{lemma} \label{hascai}   Let $A$
be a nonunital Banach algebra with a cai ${\mathfrak e}$.
\begin{itemize}
\item [(1)]  Then $0$ is in the weak* closure of $S_{\mathfrak e}(A)$.  Hence 
$0$  is in the weak* closure of $S(A)$.   Thus $Q(A)$ is a subset of the 
 weak* closure of $S(A)$, and similarly $Q_{{\mathfrak e}}(A)  \subset \overline{S_{{\mathfrak e}}(A)}^{w*}$.  
\item [(2)] The
weak* closure of $S_{\mathfrak e}(A)$  is contained in
${\mathfrak c}^{\mathfrak e}_{A^*}  \cap {\rm Ball}(A^*).$   It is also contained
in $S(A^1)_{|A}$, and both of the latter two sets
are subsets of ${\mathfrak c}_{A^*} \cap {\rm Ball}(A^*)$.
\end{itemize}
\end{lemma}    \begin{proof}
(1) \ For every $t$, 
there exists $s(t) \geq t$ such that $\| e_{s(t)} - e_t \| \geq 1/2$  (or else taking the 
limit over $s > t$ we get the
contradiction $\Vert 1 - e_t \| < 1$, which is impossible by the Neumann lemma, or since the trivial character
$\chi_0$ is contractive).  
Take a norm one $\psi_t \in A^*$ such that
$\psi_t(e_{s(t)} - e_t) = \| e_{s(t)} - e_t \|.$
Let $\phi_t( x ) = \psi_t( (e_{s(t)} - e_t) x ) / \| e_{s(t)} - e_t \|.$
Then $\phi_t \in S_{\mathfrak e}(A)$  because it has  norm one
and $\lim_s \,  \phi_t( e_s) =1.$
One has $\lim_t \, \phi_t(x)= 0$ for all  $x \in A$.  To see this, given $\epsilon > 0$ choose $t_0$ such that 
$\Vert e_t x - x \Vert < \epsilon$ for all $t \geq t_0$.  
For such $t$ we have
$$| \psi_t( (e_{s(t)} - e_t) x )| / \| e_{s(t)} - e_t \| \leq 2 \Vert \psi_t \Vert \Vert  (e_{s(t)} - e_t) x \Vert
< 4 \epsilon .$$   Thus $\phi_t \to 0$ weak*.   The rest is obvious.

(2) \ The first assertion is clear by the definitions and since
${\mathfrak c}^{\mathfrak e}_{A^*}  \cap {\rm Ball}(A^*)$ is weak* closed.  Similarly,
that the weak* closure is contained in $S(A^1)_{|A}$ follows since $S_{\mathfrak e}(A)
\subset S(A)$ as we saw above, and because $S(A^1)$ and hence $S(A^1)_{|A}$,
are weak* closed.   We leave the rest as an exercise using
${\mathfrak r}_A \subset {\mathfrak r}^{\mathfrak e}_A$.  
\end{proof}

We will say that  an approximately unital Banach algebra $A$ is {\em scaled} 
(resp.\ ${\mathfrak e}$-{\em scaled}) if 
every $f$ in ${\mathfrak c}_{A^*}$ (resp.\ in ${\mathfrak c}^{\mathfrak e}_{A^*}$) is a nonnegative multiple of a state. 
That is, iff ${\mathfrak c}_{A^*} = \Rdb^+ \, S(A)$ (resp.\ ${\mathfrak c}^{\mathfrak e}_{A^*} = \Rdb^+ \, S_{{\mathfrak e}}(A)$).   
Equivalently, iff ${\mathfrak c}_{A^*}  \cap {\rm Ball}(A^*) = Q(A)$
(resp.\ ${\mathfrak c}^{\mathfrak e}_{A^*}  \cap {\rm Ball}(A^*)  = Q_{{\mathfrak e}}(A)$).     Examples of scaled
 Banach algebras include
$M$-approximately unital Banach algebras (see Proposition \ref{preq}) and $L^1(\Rdb)$ with convolution product. 
 One can show that $L^1(\Rdb)$ is not ${\mathfrak e}$-scaled 
if ${\mathfrak e}$ is the usual cai (see the Remark after Lemma \ref{netwee}
and Example \ref{Ex4}).

\begin{lemma} \label{escaled} Let $A$ be an approximately unital Banach algebra.
\begin{itemize}
\item [(1)]   Suppose that ${\mathfrak e} = (e_t)$ is a  cai 
for $A$.  Then  $Q_{{\mathfrak e}}(A)$ is weak* closed in $A^*$ iff 
$A$ is ${\mathfrak e}$-scaled.  If these hold then  $Q_{\mathfrak e}(A)$ is a weak* compact convex set in ${\rm Ball}(A^*)$, and 
 $S_{\mathfrak e}(A)$ is weak* dense in $Q_{\mathfrak e}(A)$. 
\item [(2)]
If $S(A)$ or $Q(A)$ is convex then $Q(A)$ is weak* closed in $A^*$ iff 
$A$ is scaled.  
\end{itemize}    \end{lemma} 

\begin{proof}   (1) \   By the bipolar theorem ${\mathfrak c}^{\mathfrak e}_{A^*}= \overline{\Rdb^+ S_{{\mathfrak e}}(A)}^{w^*}$. So $\Rdb^+ S_{{\mathfrak e}}(A)$
is weak* closed iff  ${\mathfrak c}^{\mathfrak e}_{A^*} = \Rdb^+ \, S_{{\mathfrak e}}(A)$; that is iff 
$A$ is ${\mathfrak e}$-scaled.  By the Krein-Smulian theorem this happens iff Ball$(\Rdb^+ S_{{\mathfrak e}}(A)) = Q_{{\mathfrak e}}(A)$ is weak* closed.
The weak* density assertion follows from Lemma \ref{hascai}.

(2) \ Follows by a similar argument to (1)  if $Q(A)$ is convex (and this is implied by  $S(A)$ being convex).
  \end{proof}

\begin{corollary}  \label{snr}   If $A$ is a 
nonunital  approximately unital  
 Banach algebra, then the following are equivalent:
\begin{itemize}
\item [(i)]    $A$ is scaled.
\item [(ii)]  
 $S(A^1)$ is the convex hull of the 
trivial character $\chi_0$ and the set of states on $A^1$ extending states of $A$.  
\item [(iii)] 
$Q(A) = \{ \varphi_{|A} : \varphi \in S(A^1) \}$.  
\item [(iv)]   $Q(A)$ is convex and weak* compact.
\end{itemize}   If these hold then  $Q(A) = \overline{S(A)}^{{\rm w*}}$, and the numerical range satisfies
$$\overline{W_{A}(a)} =  {\rm conv} \{ 0, W_{A}(a) \} =  W_{A^1}(a), \qquad a \in A.$$
\end{corollary} \begin{proof}   (i) $\Rightarrow$ (ii) \  
Clearly  the convex hull in (ii) is a subset of $S(A^1)$.
Conversely, if $\varphi \in S(A^1)$  then $\varphi_{|A}$ is real positive, so that
by (i) we have 
 $\varphi_{|A} = t \, \psi$ for $t \in (0,1]$ and
$\psi \in S(A)$.   Then  
$\varphi = t \hat{\psi}+ (1-t) \chi_0$, where $\hat{\psi}$ is the state extending $\psi$.

(ii) $\Rightarrow$ (iii) \ We leave this as an exercise.

(iii) $\Rightarrow$ (iv) \   Suppose that $(\varphi_t)$ is a net in $S(A^1)$ whose restrictions to $A$
converge weak*
to $\psi \in A^*$.
A subnet $(\varphi_{t_\lambda})$ converges weak* to $\varphi \in S(A^1)$, and $\psi = \varphi_{|A}$
clearly.  This gives the weak* compactness in (iv), and the convexity is easier.

(iv) $\Rightarrow$ (i) \ This follows from (2) of the previous lemma.

Assume that these hold.    Since $S(A) \subset Q(A)$, that $Q(A) = \overline{S(A)}^{{\rm w*}}$  is now clear from
the fact from Lemma \ref{hascai}  that $Q(A) \subset
\overline{S(A)}^{w*}$.
Since  $A$ is
nonunital we have  $0 \in W_{A^1}(a)$.  Clearly $W_A(a) \subset W_{A^1}(a)$, so that 
 ${\rm conv} \{ 0, W_{A}(a) \}  \subset W_{A^1}(a)$.  
The converse inclusion follows easily from the above, 
so  ${\rm conv} \{ 0, W_{A}(a) \} =  W_{A^1}(a)$.     Also, 
clearly $\overline{W_{A}(a)} \subset W_{A^1}(a)$, and the converse
inclusion follows since $S(A^1)_{|A} = Q(A) = \overline{S(A)}^{w*}$.
  \end{proof}

{\bf Remarks.}  1) \   Thus if $S(A) = S_{{\mathfrak e}}(A)$ for some cai ${\mathfrak e}$ of $A$, 
then $A$ is scaled iff $Q(A)$ is weak* closed.

\smallskip

2) \ In particular, if $A$ is unital then conditions (i) and (iv) in the previous result  are automatically  true. 
Indeed  $S(A)$ is weak* closed, and hence $Q(A)$ is too, and the rest follows from  Lemma \ref{escaled}.  Item (i) also follows from the proof of \cite[Theorem 2.2]{Mag}.

\begin{theorem}  \label{con}  Let ${\mathfrak e} = (e_n)$  be a sequential cai
for a Banach algebra $A$.
 If   $Q_{\mathfrak e}(A)$ is weak* closed, 
 then $A$ possesses a sequential  cai in ${\mathfrak r}^{\mathfrak e}_A$.  Moreover
for every $a \in A$ with
$\inf \{ {\rm Re} \, \varphi(a) : \varphi \in S_{\mathfrak e}(A) \} > -1$,
there is a  sequential  cai $(f_n)$  in ${\mathfrak r}^{\mathfrak e}_A$  such that $f_n +a \in {\mathfrak r}^{\mathfrak e}_A$  for all $n$.  \end{theorem} 

\begin{proof} 
  We first state a general fact about compact spaces $K$.
If $(f_n )$ is a bounded sequence
in $C(K,\Rdb)$, such that $\lim_n f_n(x)$ exists
for every $x \in K$ and is non-negative, then
for every $\epsilon>0$, there is a function $f \in {\rm conv} \{ f_n \}$
such that $f \geq -\epsilon$ on $K$.   Indeed if this were not true, then $\overline{{\rm 
conv}} \{ f_n \}$ and $C(K)_+$
would be disjoint.   By a Hahn-Banach separation argument and the Riesz--Markov theorem there is
a probability measure $m$ such that
$\sup_n \int_K  f_n \, dm < 0$.  This is a contradiction since $\lim_n \, \int_K \, f_n \, dm \geq 0$ by 
Lebesgue's dominated  convergence theorem. 

Set $K$ to be the weak* closure of $S_{\mathfrak e}(A)$ in $A^*$
(so that
$K = Q_{\mathfrak e}(A)$ by Lemma \ref{hascai}), and let $f_n(\varphi) = {\rm Re} \, \varphi(e_n)$
for $\varphi \in K$.  Since $\lim_n \, {\rm Re} \, \varphi(e_n) \geq 0$ for all $\varphi \in Q_{\mathfrak e}(A)$, we can
apply the previous paragraph to find an
element $x \in {\rm conv} \{ e_n \}$ such that
$\inf_{\varphi \in K} \, \varphi( x) > -\epsilon.$
Similarly, choose  $y_1  \in {\rm conv} \{ e_n \}$ such that
$\inf_{\varphi \in K} \, \varphi( x + \epsilon y_1 ) > -\epsilon/2.$   Continue in this way, 
choosing  $y_n \in {\rm conv} \{ e_n \}$ such that
$$\inf_{\varphi \in K} \, \varphi( x + \epsilon \sum_{k=1}^n 2^{1 - k} y_k ) > -\epsilon/2^n.$$
Set $u =\sum_{k=1}^{\infty} \, 2^{-k} y_k  \in \overline{{\rm 
conv}}  \{ e_n \}$, and $z = x + 2 \epsilon u$.  This 
is in ${\mathfrak r}^{\mathfrak e}_A,$ and $|| z - x || < 2 \epsilon.$  
 
Choose a subsequence  $(e_{k_n})$ of $(e_n)$ such that 
$$\Vert e_{k_n} e_n -   e_n \Vert + \Vert  e_n e_{k_n} -  e_n \Vert < 2^{-n}.$$
  For each $m \in \Ndb$ apply the last paragraph to $(e_{k_n})_{n \geq m}$, and with $\epsilon$ replaced by $2^{-m}$,
 to find $x_m, u_m \in \overline{{\rm conv}} \{ e_{k_n} : n \geq m
\}$ with $z_m = x_m + 2^{1-m}  u_m \in  {\mathfrak r}^{\mathfrak e}_A$.     Then 
$\Vert x_m  e_m -  e_m \Vert + \Vert  e_m x_m -  e_m \Vert < 2^{-m}$.  
From this it is easy to see that  $(x_m)$ is a cai for $A$.  It is also easy to see now that 
$e'_m = \frac{1}{\Vert z_m \Vert} \, z_m$ is a bai (hence also a cai) for $A$ in ${\mathfrak r}^{\mathfrak e}_A$.

The case for the ``moreover" is similar.   Suppose that $\inf \{ {\rm Re} \, \varphi(a) : \varphi \in S_{\mathfrak e}(A) \} > -1$.  We may assume the infimum is negative, and choose $t > 1$ so that 
the infimum  is still
$>-1$ with $a$ replaced by  $ta$.  We now begin to follow the argument in previous paragraphs,
with the same $K$, but starting from a cai $(e'_n)$ in ${\mathfrak r}^{\mathfrak e}_A$.
Since $\lim_n \, {\rm Re} \, \varphi(ta + e'_n) \geq 0$ for all $\varphi \in Q_{\mathfrak e}(A)$, we can apply the above to find an
element $x \in {\rm conv} \{ e'_n \} \subset {\mathfrak r}^{\mathfrak e}_A$ such that
$\inf_{\varphi \in K} \, \varphi(ta +  x ) > -\epsilon.$   Continue as above to find $u  \in \overline{{\rm 
conv}}  \{ e'_n \} \subset {\mathfrak r}^{\mathfrak e}_A$ so that 
 $z = ta + x + 2 \epsilon u$    
is  in ${\mathfrak r}^{\mathfrak e}_A,$ with  $|| z - x - ta || < 2 \epsilon.$      For each $m \in \Ndb$
there exists such $x_m , u_m \in {\mathfrak r}^{\mathfrak e}_A$ so that 
 $z_m = ta + x_m + 2^{1-m} u_m$ is  in ${\mathfrak r}^{\mathfrak e}_A,$ with  $|| z_m - x_m - ta
\Vert \leq 2^{1-m}$, and such that $(x_m)$ is a cai for $A$.     Note that 
$z_m - ta \in {\mathfrak r}^{\mathfrak e}_A$, and hence  $f_m = \frac{1}{\Vert z_m - ta \Vert} (z_m - ta) \in {\mathfrak r}^{\mathfrak e}_A$.   Also $(f_m)$ is a bai (hence a cai) for $A$ in ${\mathfrak r}^{\mathfrak e}_A$.   There exists an $N$ such that  $\frac{t}{\Vert z_m - ta \Vert} > 1$ 
for $m \geq N$.  Thus  $f_m + a \in {\mathfrak r}^{\mathfrak e}_A$ 
for $m \geq N$, since this is a convex combination
of $f_m$ and $f_m + \frac{ta}{\Vert z_m - ta \Vert} = \frac{z_m}{\Vert z_m - ta \Vert}$.   \end{proof}

\begin{corollary} \label{scase}   Let ${\mathfrak e} = (e_n)$  be a sequential cai
for a Banach algebra $A$.  Assume that 
$S(A) = S_{\mathfrak e}(A)$ (which is the case for example if $A$ is Hahn-Banach smooth).
 If   $Q(A)$ is weak* closed, 
 then $A$ possesses a sequential  cai in ${\mathfrak r}_A$.  Moreover
for every $a \in A$ with
$\inf \{ {\rm Re} \, \varphi(a) : \varphi \in S(A) \} > -1$,
there is a   sequential  cai $(f_n)$  in  ${\mathfrak r}_A$ such that $f_n \succeq -a$  for all $n$.  
If, in addition, $A$ has a sequential  cai 
in ${\mathfrak F}_A$ then the sequential  cai $(f_n)$ in the last line can also
be chosen to be in ${\mathfrak F}_A$. \end{corollary}

\begin{proof}  By the last result $A$ has a sequential   cai in  ${\mathfrak r}_A$
satisfying the first two assertions.  Suppose that
   $A$ has a sequential  cai,  $(e_n')$ say, in  ${\mathfrak F}_A$.    One then follows the last paragraph of 
the last proof.  Now $x_m, u_m \in {\mathfrak F}_A$.  Define  $f_m$ as before,
but the desired cai is $\frac{\Vert x_m +
2^{1-m} u_m \Vert}{1 + 2^{1-m}} \, f_m$, which 
is easy to see is a convex combination
of $x_m$ and $u_m$,
and hence is in ${\mathfrak F}_A$.   Moreover a tiny modification of the 
argument above shows that the sum of this cai and $a$ is in ${\mathfrak r}_A$
for $m$ large enough.  
\end{proof}

{\bf Remark.}\ 
Under the 
conditions of Corollary \ref{scase}, and if $A$ has a sequential 
approximate identity in $\frac{1}{2} {\mathfrak F}_A$ (resp. ${\mathfrak F}_A$), then a slight variant of the last proof shows that 
for any $a \in A$ with $\inf \{ {\rm Re} \, \varphi(a) : \varphi \in S(A) \} > -1$,
there is a   sequential  bai $(f_n)$  in  $\frac{1}{2}  {\mathfrak F}_A$ 
(resp. ${\mathfrak F}_A$) such that $f_n \succeq -a$  for all $n$.  
By Corollary \ref{hasbi} (and the 
remark after it) below, if $A$ has a sequential bai in ${\mathfrak r}_A$ 
then $A$ does have a sequential bai in ${\mathfrak F}_A$.

\bigskip

We also remark that Corollary 3.4 of \cite{B2015} generalizes the 
first assertion of Corollary \ref{scase} above to non-sequential cais.

\begin{proposition} \label{pgold}  If $A$ is a scaled approximately unital Banach algebra then the weak* closure
of ${\mathfrak r}_A$ is ${\mathfrak r}_{A^{**}}$.
\end{proposition} 

\begin{proof}     
It is easy to see from the definitions 
 that ${\mathfrak r}_A \subset {\mathfrak r}_{A^{**}}$.    
Clearly   ${\mathfrak r}_A^\circ = {\mathfrak c}_{A^*}$, so the result will  follows from the bipolar theorem if we can 
show that  $$({\mathfrak c}_{A^*})^\circ = 
{\mathfrak r}_{A^{**}} = {\mathfrak r}_{(A^1)^{**}} \cap A^{**}.$$    
Since ${\mathfrak r}_A \subset {\mathfrak r}_{A^{**}}$ it is clear that 
$({\mathfrak r}_{A^{**}})_\circ \subset {\mathfrak c}_{A^*}$.   If $\varphi
\in  {\mathfrak c}_{A^*}$ then $\varphi = t \psi$ for $t > 0, \psi \in S(A)$.
Then $\psi$ extends to a state $\hat{\psi}$ on $A^1$, and to a weak* continuous 
state $\rho$ on $(A^1)^{**}$.   If  $\eta \in 
{\mathfrak r}_{A^{**}}$ we have
$${\rm Re} \, \eta(\psi) =  {\rm Re} \, \eta(\hat{\psi})  = {\rm Re} \, \rho(\eta) \geq 0.$$   
That is, 
$\varphi
\in  ({\mathfrak r}_{A^{**}})_\circ$.   So $({\mathfrak r}_{A^{**}})_\circ = {\mathfrak c}_{A^*}$,
and hence by the bipolar theorem $({\mathfrak c}_{A^*})^\circ = {\mathfrak r}_{A^{**}}$.
\end{proof}

We 
 remark that if an approximately unital  
 Banach algebra $A$ is scaled then any mixed identity $e$ for $A^{**}$  of norm $1$ is lower semicontinuous on $Q(A)$.
For if $\varphi_t \to \varphi$ weak* in $Q(A)$,
and $\varphi_t(e) = \Vert \varphi_t \Vert \leq r$ for all $t$,
then $\Vert \varphi \Vert  = \varphi(e)  \leq r$.   A similar assertion holds in the ${\mathfrak e}$-scaled case.

\section{Positivity and roots in Banach algebras}

\begin{proposition} \label{aufr}  If $B$ is a closed subalgebra of a nonunital
Banach algebra $A$, and if $A$ and $B$ have a 
common cai, then $B^1 \subset A^1$ isometrically and unitally, 
 $S(B^1) = \{ f_{|B^{1}} : f \in S(A^1) \},$ and ${\mathfrak F}_B=B \cap {\mathfrak F}_A$
and ${\mathfrak r}_B = B \cap {\mathfrak r}_A$.   Moreover in this case if $A$ is $M$-approximately unital then so is $B$.
\end{proposition} 

\begin{proof}  We leave the first part of this as an exercise.   The last assertion follows using 
\cite[Proposition I.1.16]{HWW}, since in this case multiplying by $e$ leaves $(B^1)^\perp$ invariant inside
$(A^1)^{**}$.  
\end{proof}

{\bf Remark.}  Similarly, in the situation of Proposition \ref{aufr} we have ${\mathfrak r}^{\mathfrak e}_B = B \cap {\mathfrak r}^{\mathfrak e}_A$ if ${\mathfrak e}$ is the common cai.

\begin{proposition} \label{auf2}  Suppose that  $J$ is a closed  
approximately unital ideal in an
approximately unital  Banach algebra $A$, and that  $J$ is
also an $M$-ideal in $A$.   \begin{itemize} \item [{\rm (1)}] 
 ${\mathfrak F}_J= J \cap {\mathfrak F}_A$ and 
${\mathfrak r}_J = J \cap {\mathfrak r}_A$, and  states on $J$ extend to states on $A$.  
\item [{\rm (2)}] If $J$ is nonunital then  $J^1 \subset A^1$ isometrically and unitally, and 
$S(J^1) = \{ f_{|J^{1}} : f \in S(A^1) \}.$   
\item [{\rm (3)}]  If  $A$  is $M$-approximately unital, then so is $J$. 
 \item [{\rm (4)}]  If ${\mathfrak e} = (e_i)$ is a cai of $A$, then there is a cai
${\mathfrak h} = (h_j)$ of $J$ such that $\varphi_{\vert J} \in Q_{{\mathfrak h}}(J)$ whenever
$\varphi \in S_{{\mathfrak e}}(A)$.
\end{itemize} 
\end{proposition}

\begin{proof}  (2) \  For $a \in J$ and $\lambda \in \Cdb$ we have 
$$\| a + \lambda 1 \|_{A^1}
 = \sup\{ \| ax + \lambda x \|_A : x \in {\rm Ball}(A) \}
 \geq \| a + \lambda 1 \|_{J^1}.$$
Let $f$ be a mixed identity of $J^{**}$ of norm one, which is the limit of a cai $(f_i)$.
For every $x \in {\rm Ball}(A)$, one has
$$\| ax +  \lambda x \|_A
 = \max\{ \| fax + \lambda fx \|, \| \lambda (1-f)x \| \} .$$
Setting $a = 0$ temporarily we see that $\| \lambda (1-f)x \| \leq |\lambda|  \le \| a + \lambda 1 \|_{J^1}$.   
For any $a \in J$ we have $fax = ax$ and
$ax + \lambda fx = w^*\lim_i \,  af_i x + \lambda f_i x$, 
so that
$$\| fax + \lambda fx \|
 \le \liminf_i \,  \| af_ix + \lambda f_ix \|
 \le \| a + \lambda 1 \|_{J^1}.$$
Thus $\| a + \lambda 1 \|_{A^1} = \| a + \lambda 1 \|_{J^1}$.  

(1) \   If $J$ is nonunital   then by (2) and the Hahn-Banach theorem we have $S(J^1) = \{ f_{|J^{1}} : f \in S(A^{1}) \}$,
and so  states on $J$ extend to states on $A$.    If $J$ is unital an extension of states is given by $\varphi
\mapsto \varphi(1_J \,\cdot)$.  
It also is clear from (1) that ${\mathfrak F}_J=J \cap {\mathfrak F}_A$ in the nonunital case; and we leave the  unital case as an exercise 
(using the fact that multiplication by the identity of $J$ is an $M$-projection).
Similarly for  ${\mathfrak r}_J = J \cap {\mathfrak r}_A$.    Indeed clearly 
$J \cap {\mathfrak r}_A \subset {\mathfrak r}_J$ since states on $J$ extend to states on $A^1$.   We leave the 
converse inclusion as an exercise (for example it follows from ${\mathfrak F}_J=J \cap {\mathfrak F}_A \subset J \cap {\mathfrak r}_A$,
and Proposition \ref{whau} below).  

 (3) \  We can assume $J$ nonunital.
It follows from \cite[Proposition 1.17b]{HWW} that if $J$ is an $M$-ideal in $A$,
and $A$ is an $M$-ideal in $A^1$, then
$J$ is an $M$-ideal in $A^1$.   By \cite[Proposition 1.17b]{HWW},
$J$ is an $M$-ideal in $J^1$.   

(4) \ 
 Let $e$ denote a weak* limit point in $A^{**}$ of $(e_i)$.
Let $(g_k)$ be any cai for $J$, with weak* limit point
$g$ in $J^{\perp \perp}$.
Then $(h_j) = (g_k \,  e_i)$ (indexed first by $i$ and then $j$)
 is a cai for $J$.  Then $h = ge$ is a weak* limit point of $(h_j)$.
We have $(1-g) e = e - h$. Since left multiplication
 by $g$ is the $M$-projection of $A^{**}$ onto $J^{\perp \perp}$,
 as we have seen several times above,
 one has $\| e - h \| \le 1.$
Let $\varphi \in S_{{\mathfrak e}}(A)$  be given. 
We claim that if $\varphi(h) = 0$ then $\varphi_{\vert J} = 0$;
 and if $\varphi(h) \neq 0$ then  
$\varphi(h \, \cdot ) / \varphi(h)$ is
a state on $J^1$. 
Note that if $\varphi(h) \neq 0$
then 
$$1 = \varphi(e) = \varphi(h) + \varphi((1-g)e)
\leq |\varphi(h)| + |\varphi((1-g)e)|
\leq \Vert \varphi(g \, \cdot) \Vert + 
\Vert \varphi((1-g) \, \cdot) \Vert,$$ 
which equals $1$ due to the $L$-decomposition in $A^*$.
Thus  we must have $\varphi(h) \geq 0$.
Let $a + \lambda 1 \in {\rm Ball}(J^1)$ be given.
Then for any unimodular scalar $\gamma$ one has
$$\|\gamma (ha + \lambda h) + e - h \|_{A^{**}}
 = \max\{ \| ha + \lambda h \|, \| e - h \| \} \le 1.$$ 
Therefore 
$$|\varphi(\gamma (ha + \lambda h) + e - h)| = 
|\gamma \varphi(ha + \lambda h) + 1 - \varphi(h)| \le 1$$
for all such $\gamma$. 
So for some such $\gamma$, 
$$|\varphi(ha + \lambda h)| + 1 - \varphi(h)
= \varphi(\gamma (ha + \lambda h) + e - h)
\le 1,$$
so that $|\varphi(ha + \lambda h)| \leq \varphi(h).$
\end{proof}
 
\begin{proposition} \label{clnth}  {\rm (Esterle)} \ If $A$ 
is a unital Banach algebra then 
${\mathfrak F}_A$ is closed under (principal) $t$'th powers  
for any $t \in [0,1]$.   Thus if $A$ is an approximately unital Banach algebra then
${\mathfrak F}_A$  and $\Rdb^+ {\mathfrak F}_A$ are closed under $t$'th powers
for any $t \in (0,1]$.
\end{proposition} 

\begin{proof}  This is in  \cite[Proposition 2.4]{Est} (see also \cite[Proposition 2.3]{BRI}), but for convenience we repeat the construction.  If $\Vert 1 - x \Vert \leq 1$, define 
$$x^{t} = \sum_{k = 0}^\infty \, {t \choose k} (-1)^k (1-x)^k \; , \qquad 
t > 0.$$
For $k \geq 1$ the sign of ${t \choose k} (-1)^k$ is
always negative, and $\sum_{k = 1}^\infty \,
{t \choose k} (-1)^k = -1$.
It follows that the series for $x^t$ above 
 is a norm limit of polynomials in $x$ with no constant term.
Also,   
$1 - x^{t} = \sum_{k=1}^\infty\,
{t \choose k} (-1)^k \, (1-x)^k$, which is a convex combination
in Ball$(A^1)$.   So $x^{t} \in {\mathfrak F}_A$.
 
Using
 the Cauchy product formula in Banach algebras in a standard way,
one deduces that  
$(x^{\frac{1}{n}})^n = x$ for any positive integer $n$.   
\end{proof}

From  \cite[Proposition 2.4]{Est} if $x \in {\mathfrak F}_A$ then we also have $(x^t)^r = x^{tr}$ for $t \in [0,1]$ and 
any real $r$; and that if $a x_n \to a$ where $a \in A$ and $(x_n)$ is a sequence with $\Vert x_n - 1 \Vert < 1$, then  $a x^t_n \to a$ with $n$ for all real $t$.

If $A$ is  a unital 
Banach algebra then we define the ${\mathfrak F}$-transform 
to be  ${\mathfrak F}(x) = x (1 + x)^{-1} = 1 - (1 + x)^{-1}$ for $x \in {\mathfrak r}_A$.  Then ${\mathfrak F}(x) 
\in {\rm ba}(x)$.  The inverse transform takes $y$ to 
$y(1 -y)^{-1}$.

\begin{lemma} \label{stam}  If $A$ is an approximately unital 
Banach algebra then ${\mathfrak F}({\mathfrak r}_A) \subset {\mathfrak F}_A$.  \end{lemma}

 \begin{proof} This is because by a result of Stampfli 
and Williams \cite[Lemma 1]{SW},
$$\Vert 1 - x (1 + x)^{-1} \Vert  =  \Vert (1 + x)^{-1}
\Vert \leq d^{-1} \leq 1$$ where $d$ is the distance from $-1$ 
to the numerical range of $x$.  \end{proof}

  If $A$ is also an operator algebra then we have shown 
elsewhere \cite[Lemma 2.5]{Bord} that the range of the  ${\mathfrak F}$-transform is exactly the set of strict contractions
in $\frac{1}{2}{\mathfrak F}_A$.

\begin{proposition} \label{whau}  If $A$ is an approximately unital Banach algebra then
$\overline{\Rdb^+ {\mathfrak F}_A} = {\mathfrak r}_A$.  
\end{proposition}  \begin{proof}  
As in \cite[Theorem 3.3]{BRII}, it follows that if
 $x \in {\mathfrak r}_A$
then $x = \lim_{t \to 0^+} \, \frac{1}{t} \, tx(1 + tx)^{-1}$.  
By Lemma \ref{stam}, $tx(1 + tx)^{-1} \in {\mathfrak F}_A$. So
$\Rdb^+ {\mathfrak F}_A$ is dense in ${\mathfrak r}_A$.  \end{proof} 

In the following results we will use the fact that 
if $A$ is an approximately unital Banach algebra,
then the `regular representation' $A \to B(A)$ is
isometric.  Thus we can view an accretive $x \in A$ and its (principal) roots
as operators in $B(A)$.   These are sectorial of angle $\frac{\pi}{2}$, and so we can use the theory of roots (fractional powers) from 
e.g.\ \cite[Section 3.1]{Haase},
or \cite{LRS,NF}.  
Basic properties of such  (principal)  powers include: $x^s x^t = x^{s+t}$, $(cx)^t = c^t x^t$ for 
positive scalars $c,s,t$, and $t \to x^t$ is continuous. See also  e.g.\
Section 11 in IX in Yosida's
classic Functional Analysis text,  \cite[Lemma 1.1 (1)]{Bord}, or \cite[p.\ 64]{Est}.  
Also $x^t = \lim_{t \to 0^+} \, (x+\epsilon 1)^t$ for $t > 0$, and the latter can be taken to be with respect to the 
usual Riesz functional calculus (see \cite[Proposition 3.1.9]{Haase}).  Principal nth roots of accretive elements 
are unique, for any positive integer $n$ (see \cite{LRS}).

\bigskip

{\bf Remark.}    It is easy to see from the last fact that the definitions of $x^t$ given in 
\cite{Haase} and \cite[Theorem 1.2]{LRS} coincide.    A similar argument shows that if 
$x \in {\mathfrak F}_A$ then  the definitions of $x^t$ given in 
\cite{Haase} and  Proposition \ref{clnth} coincide,  if $t  > 0$. 
Indeed for the latter we may assume that $0 < t \leq 1$ and work in $B(A)$ as above (and we may assume $A$ unital).  Then the two definitions of $y^t$ coincide
if  $y = \frac{1}{1+\epsilon} (x + \epsilon I)$, since 
both equal the  $t$th power of $y$ as given by the Riesz functional calculus.   However 
$\sum_{k = 0}^\infty \,
{t \choose k} (-1)^k (1-y)^k$ converges uniformly to  $\sum_{k = 0}^\infty \,
{t \choose k} (-1)^k (1-x)^k$, as $\epsilon \to 0^+$, since the norm of the difference of these two series is dominated by
$$\sum_{k = 1}^\infty \,
{t \choose k} (-1)^k \, (\frac{1}{1+\epsilon} - 1) \, \Vert (1-x)^k \Vert \; \leq \; \frac{\epsilon}{1 + \epsilon} \, \to \, 0 .$$ 

See \cite{B2015} for more details concerning
the last remark, and also for a better estimate in the next result in 
the operator algebra case.

\begin{lemma} \label{Bal} Let  $A$ be an approximately unital Banach algebra.   If $||x|| \leq 1$ and $x \in {\mathfrak r}_A$,
 then $||x^{1/m}||
\leq \frac{2m^2}{(m-1) \pi} \sin(\frac{\pi}{m}) \leq \frac{2m}{m-1}$ for $m \geq 2$.
More generally, $||x^{\alpha}|| \leq  \frac{2 \sin(\alpha \pi)}{\pi \alpha (1 - \alpha)} \, \Vert x \Vert^\alpha$ if $0 < \alpha < 1$ and $x \in {\mathfrak r}_A$.   If $A$ is also an operator algebra then one may 
remove the $2$'s in these estimates.
\end{lemma}

\begin{proof}
This follows from the well known A.\ V.\ Balakrishnan representation of powers,
 $x^\alpha = \frac{\sin(\alpha \pi)}{\pi} \int_0^\infty \, t^{\alpha - 1} \, (t + x)^{-1} x \, dt$
(see e.g.\ \cite{Haase}).
We use the simple fact that $\Vert (t + x)^{-1} \Vert \leq \frac{1}{t}$ for accretive $x$ and $t > 0$,
and so 
$$\Vert (t + x)^{-1} x \Vert = \Vert  (1 + \frac{x}{t})^{-1}  \frac{x}{t} \Vert
= \Vert {\mathfrak F}(\frac{x}{t}) \Vert \leq 2,$$
and is even $\leq 1$ in the operator algebra case by the observation
after  Lemma \ref{stam}.
Then
the norm of $x^\alpha$ is dominated by
$$\frac{2 \sin(\alpha \pi)}{\pi} (\int_0^1 t^{\alpha - 1} \, \cdot 1  dt + \int_1^\infty \, t^{\alpha - 1}
\frac{1}{t}  \, dt ) = \frac{2 \sin(\alpha \pi)}{\pi \alpha (1 - \alpha)}.$$
The rest is clear from this.
\end{proof}

We will sometimes use the fact from  \cite[Corollary 1.3]{LRS} that the $n$th root function is continuous on 
${\mathfrak r}_{A}$.  

\begin{lemma} \label{sin}
 There is a nonnegative sequence $(c_n)$ in $c_0$  such that for any unital Banach algebra $A$,
and $x \in {\mathfrak F}_A$ or $x \in  {\rm Ball}(A) \cap {\mathfrak r}_A$, we have
$\Vert x^{\frac{1}{n}} x - x \Vert \leq c_n$ for all $n \in \Ndb$.
    \end{lemma}

\begin{proof}   We follow the
proof of \cite[Theorem 3.1]{BRII}, taking $R = 3$ there.   This is based on the Banach algebra construction from 
\cite{LRS}, so will be valid in the present generality.
There an estimate $\Vert x^{\frac{1}{n}} x - x \Vert \leq D c_n$ is given,
 for a   nonnegative sequence $(c_n)$ in $c_0$.  We need to know that $D$ does not 
depend on $A$ or $x$.  
This follows if $\Vert \lambda \, (\lambda 1 - x)^{-1} \Vert$ is
bounded independently of $A$ or $x$ on the curve $\Gamma$ there.  
On the piece of the curve $\Gamma_2$, this follows by the  result of Stampfli 
and Williams \cite[Lemma 1]{SW} that $\Vert (\lambda 1 - x)^{-1}
\Vert \leq d^{-1}$ where $d$ is the distance from $\lambda$ to $W(x)$.
On the other part of $\Gamma$ we have $\lambda = t e^{i \theta}$
for $0 \leq t \leq R$, and for a fixed $\theta$ with 
$\frac{\pi}{2} < |\theta| < \pi$.  However by the same result of Stampfli 
and Williams $\Vert (\lambda 1 - x)^{-1}
\Vert \leq d^{-1}$ if $\lambda \neq 0$, where $d$ is the distance from $\lambda$ 
to the $y$-axis.  Thus the quantity will be bounded
since $|\lambda|/d = \csc(\theta - \frac{\pi}{2})$.
\end{proof}

 The following (essentially from \cite{MP})  is a related result:

\begin{lemma} \label{strsq}
Let  $A$ be an unital Banach algebra.
If $\alpha \in (0,1)$ then there
exists a constant $K$ such that if $a, b \in  {\mathfrak r}_{A}$, and $ab = ba$,  then
 $\Vert (a^{\alpha}  - b^\alpha) c \Vert
\leq K \Vert (a-b) c \Vert^\alpha$, for any $c \in {\rm Ball}(A)$.
    \end{lemma}

\begin{proof}    By the Balakrishnan representation in the proof of Lemma \ref{Bal},
 if $c \in {\rm Ball}(A)$ we have
$$(a^{\alpha}  - b^\alpha) c =
\frac{\sin(\alpha \pi)}{\pi} \int_0^\infty \, t^{\alpha - 1} \, [(t + a)^{-1} a - (t + b)^{-1} b]  c \, dt.$$
By the inequality $\Vert (t + x)^{-1} \Vert \leq \frac{1}{t}$ for accretive  $x$, we have
$$\Vert [(t + a)^{-1} a - (t + b)^{-1} b] c \Vert =
\Vert (t + a)^{-1} (t + b)^{-1} (a-b) t c \Vert
\leq \frac{1}{t}  \Vert (a-b) c \Vert ,$$
   and so as in the proof of Lemma \ref{Bal},  $\Vert \int_0^\infty \, t^{\alpha - 1} \, [(t + a)^{-1} a - (t + b)^{-1} b] c \, dt \Vert$ is dominated
by $$4  \int_0^\delta \, t^{\alpha - 1}  \, dt  + \int_{\delta}^\infty \,  t^{\alpha - 2} \, dt \, \Vert (a-b) c \Vert
= \frac{4}{\alpha} \delta^{\alpha} + \frac{\delta^{\alpha -1}}{1-\alpha} \, \Vert (a-b) c \Vert$$
for any $\delta > 0$.
We may now set $\delta = \Vert (a-b) c \Vert$ to obtain our inequality.
   \end{proof}

\begin{corollary} \label{hasbi}  An approximately unital  Banach algebra with a left bai 
(resp.\ right bai, bai) in 
${\mathfrak r}_A$ has a left bai 
(resp.\ right bai, bai)  in ${\mathfrak F}_A$.  
\end{corollary}

 \begin{proof}  If $(e_t)$ is a left bai in
${\mathfrak r}_A$, let $b_t = {\mathfrak F}(e_t) 
\in {\mathfrak F}_A$.  If $a \in A$ then 
$$b_t^{\frac{1}{n}} a = 
b_t^{\frac{1}{n}}(a  - e_t a) +
(b_t^{\frac{1}{n}} e_t - e_t) a + e_t a .$$
The first term here converges to $0$ with $t$ since 
$(b_t^{\frac{1}{n}})$ is in ${\mathfrak F}_A$,
hence is bounded. 
Similarly, the middle term can be seen to converge to $0$ with
$n$ by rewriting it  as
$(b_t^{\frac{1}{n}} b_t - b_t) \, (1+e_t) \, a$.
Working in $A^1$ and applying  Lemma \ref{sin}
we have $$\Vert (b_t^{\frac{1}{n}} b_t - b_t) \, (1+e_t) \, a
\Vert \leq c_n \Vert 1+e_t \Vert \Vert a \Vert \leq
K c_n \to 0,$$
for a constant $K$ independent of $t$.  
The third term converges to $a$ with $t$.
  So $(b_t^{\frac{1}{n}})$ is a left bai.
Similarly in the 
right and two-sided cases.        \end{proof}

{\bf Remark.}   If the bai in the last result is sequential, then so is
the one constructed in ${\mathfrak F}_A$.

\begin{corollary} \label{hasb}  If $A$ is an approximately unital Banach algebra then ${\mathfrak r}_A$ is closed under $n$th roots
for any positive integer $n$. 
\end{corollary} \begin{proof}   We saw in the 
proof of Proposition \ref{whau} 
that if $x \in {\mathfrak r}_A$ then
 $x = \lim_{t \to 0^+} \, \frac{1}{t} \, tx(1 + tx)^{-1}$, and $tx(1 + tx)^{-1} \in 
 {\mathfrak F}_A$.  Thus by \cite[Corollary 1.3]{LRS} we have that 
$x^r = \lim_{t \to 0^+} \, \frac{1}{t^r} \, (tx(1 + tx)^{-1})^r$ for $0 < r < 1$.  By Proposition \ref{clnth},
the latter powers are in  $\Rdb^+ {\mathfrak F}_A$, so that $x^r \in \overline{\Rdb^+ {\mathfrak F}_A} = {\mathfrak r}_A$.  \end{proof} 

\begin{proposition} \label{whba}   If $A$ is an approximately unital Banach algebra and $x \in {\mathfrak r}_A$
then $ {\rm ba}(x) = {\rm ba}({\mathfrak F}(x))$, and so
$\overline{xA} =\overline{{\mathfrak F}(x) A}$.  \end{proposition}  \begin{proof}   This follows from the 
elementary spectral theory of unital Banach algebras, applied in $A^1$.  Below we compute the spectrum in ba$(x)^1$.  
Since $0 \notin {\rm Sp}(1+x)$  we have $(1+x)^{-1} 
\in  {\rm ba}(1,x)$, so that ${\mathfrak F}(x) \in {\rm ba}(x)$.  Any character of ba$(x)^1$ applied to
${\mathfrak F}(x)$ gives a number of form
$z = w (1+w)^{-1}$ in the open unit disk, and in fact also  inside the 
circle 
$|z - \frac{1}{2}| \leq  \frac{1}{2}$ if Re$(w) \geq 0$.  Since $1 \notin
 {\rm Sp}({\mathfrak F}(x))$ we have 
 $(1- {\mathfrak F}(x))^{-1} \in {\rm ba}(1,{\mathfrak F}(x))$, so that
$x = - {\mathfrak F}(x) \, (1- {\mathfrak F}(x))^{-1} \in {\rm ba}({\mathfrak F}(x))$.   
The rest is clear.   
 \end{proof}

\begin{lemma} \label{lump}  If $p$ is an idempotent in a unital Banach algebra 
$A$ then $p  \in {\mathfrak F}_A$
iff $p  \in {\mathfrak r}_A$.   If $p$  is an idempotent in $A^{**}$ for an approximately  unital Banach algebra  $A$ then 
$p  \in {\mathfrak F}_{A^{**}}$
iff $p  \in {\mathfrak r}_{A^{**}}$.   \end{lemma}  \begin{proof}   The first follows from the well-known 
Lumer-Phillips
characterization of accretiveness in terms of $\Vert {\rm exp}(-tp) \Vert \leq 1$
for all $t > 0$ (see e.g.\
\cite[Theorem 6, p.\ 30]{BoNR1}).   If $p$ is idempotent then ${\rm exp}(-tp) = 1- (1-e^{-t})p$, and if this is contractive for all $t > 0$ then
$\Vert 1 - p \Vert \leq 1$.
For the second, work in $(A^1)^{**}$ and use facts above.   \end{proof} 

 However one cannot say that the idempotents in the last result are also in 
$\frac{1}{2} {\mathfrak F}_A$ as is the case for operator algebras.  The following examples illustrate
this, and other `bad behavior' not seen in the class of operator algebras.

\begin{example} \label{Ex1}  Let $\ell^1_4$ be identified with the 
$l^1$-semigroup algebra of the abelian semigroup  $\{ 1, a, b, c \}$ with
relations making $a, b, c$ idempotent, and $ab = ac = bc = c$. 
Then $p = 1-a, q = 1-b
 \in {\mathfrak F}_A \setminus \frac{1}{2} {\mathfrak F}_A \subset {\mathfrak r}_A$. 
   For such $p$ set $x = \frac{1}{2} p \in
\frac{1}{2} {\mathfrak F}_A$, and notice that $x^{\frac{1}{n}} = \frac{1}{2^{\frac{1}{n}}} p$ which is
not always in $\frac{1}{2} {\mathfrak F}_A$ (if it were, then we get the contradiction that its limit
$p$ is  in $\frac{1}{2} {\mathfrak F}_A$).   So we see that   $\frac{1}{2} {\mathfrak F}_A$ is not closed
under $n$th roots. 
  We also see that if $x \in \frac{1}{2} {\mathfrak F}_A$ then 
$\overline{xA}$ need not have a left cai (even if $A$ is commutative).  It does have a left bai 
of norm $\leq 2$, indeed a left bai in ${\mathfrak F}_A$ by Corollary \ref{lba}. 

In this example  $p q = p^{\frac{1}{2}} q^{\frac{1}{2}} = 1 - a - b + c  \notin  {\mathfrak r}_A$ (as can be seen by considering states
$f(\alpha a + \beta b + \gamma c + \lambda 1) =  \gamma z + \lambda + \alpha + \beta$ for $|z| \leq 1$).   So
 $x^{\frac{1}{2}}  y^{\frac{1}{2}}$ need not be
in  ${\mathfrak r}_A$ even if $x, y \in \frac{1}{2} {\mathfrak F}_A$.   This shows that
the main results about roots in \cite{BBS} fail in more general $M$-approximately
unital  Arens regular Banach algebras.    Note too that if $J_1 = pA$ and $J_2 = qA$, 
then $J_1 \cap J_2 = \Cdb d = d A$, where $d = pq$, but $dA$ has no identity or bai in ${\mathfrak r}_A$.
This shows that, unlike in the operator algebra case,
 finite intersections of extremely nice closed ideals need not be `nice' in
the sense of the theory developed in this paper.   See however Section 8 for a context 
in which finite intersections  will behave well.
 \end{example} 

\begin{example} \label{Ex2}   In the Banach algebra $A = l^1(\Zdb_2)$ with convolution multiplication, $p = (\frac{1}{2}, \frac{1}{2})$  is
a contractive idempotent in $\frac{1}{2} {\mathfrak F}_A$ with numerical range $\overline{B(\frac{1}{2}, \frac{1}{2})}$.
The states in this example are the functionals $(a,b) \mapsto a + bz$, for $|z| \leq 1$.
All of the  principal $n$th roots of $p$ obviously have the same  
numerical range.  So the numerical range  of $p^{\frac{1}{n}}$ does not `converge' to the $x$-axis.  Thus we cannot 
expect statements in the Blecher-Read papers involving `near positivity' to generalize (unless $A$ is a Hermitian Banach *-algebra satisfying the conditions in the latter part of \cite{LRS}, in which case 
 the numerical ranges  of  $x^{\frac{1}{n}}$ do `converge' to the $x$-axis if $x$ is accretive).     Note also in this example
that $p$ is not an $M$-projection in $A$.   Thus we cannot 
expect  support projections to be associated with  $M$-projections in general.   In this example it is easy to see that $x = (a,b) \in {\mathfrak r}_A$ iff $|b| \leq {\rm Re} \,
 a$, whereas $x \in \frac{1}{2} {\mathfrak F}_A$ iff $|b|^2 - |b| \leq {\rm Re} \,
a - |a|^2$.   
In this example the Cayley transform does not take ${\mathfrak r}_A$ into the set of contractions, so that 
$x(1+x)^{-1}$ need not be in $\frac{1}{2} {\mathfrak F}_A$.  

This example also serves to show that if $B$ is an approximately unital  closed ideal in a commutative
finite dimensional approximately unital  Banach algebra, then ${\mathfrak r}_B$ and ${\mathfrak F}_B$ need not
be related to ${\mathfrak r}_A$ and ${\mathfrak F}_A$, unlike the setting of operator algebras (where
there is a very strong relationship between these, even in the case $B$ is a subalgebra). 
 Indeed let $B = \Cdb (1,1)$
inside the last example.   We have $1_B = (\frac{1}{2}, \frac{1}{2})$, and ${\mathfrak r}_B=  \{ (a,a) : 
{\rm Re} \, a \geq 0 \}$ and ${\mathfrak F}_B
= \{ (a,a) : a \in \overline{B(\frac{1}{2}, \frac{1}{2})} \}$.

For a state $\varphi$ on an operator algebra $A$ and $x \in {\mathfrak F}_A$
 it is the case that $\varphi(s(x)) = 0$ iff
$\varphi(x) = 0$ iff $\varphi \in {\rm ba}(x)^\perp$.  Here $s(x)$ is the support
projection of $x$ from \cite{BRI}.  In Example \ref{Ex2},
if $x = (\frac{1}{2}, \frac{i}{2})$ and $\varphi((a,b)) = a + ib$
then $x \in {\rm Ker} \, \varphi$ but $x^2$ and $s(x) = 1$
are not in ${\rm Ker} \, \varphi$.   Thus much of the theory 
of `strictly real positive' elements from \cite{BRI} and its sequels breaks down.

A slight variant of this example is the same algebra, but with norm 
$|||(a,b)||| = |a| + 2 |b|$.  Here $J = \Cdb (\frac{1}{2}, \frac{1}{2})$ 
is an ideal equal to $xA$ for $x \in {\mathfrak F}_A$, but 
this ideal has no cai.   \end{example} 

\begin{example} \label{Ex3}   The unital Banach algebra $l^1(\Ndb)$,
with convolution product,  is easily seen to be equal to ba$(x)$ 
where $x = 1 + \frac{1}{2} \vec e_2 \in {\mathfrak F}_A$.   However
$l^1(\Ndb)$ is not Arens regular; thus its second dual 
is not commutative in either one of the Arens products \cite[1.4.9]{Pal}.
Thus ba$(x)^{**}$ need not be  commutative if
$x \in {\mathfrak F}_A$.    In this example it is easy to compute ${\mathfrak F}_A$ and ${\mathfrak r}_A$.   C. A. Bearden has verified
that in this example, unlike the operator algebra case \cite{BBS},
$(x^{\frac{1}{n}})$ need not increase in the 
`real positive ordering' with $n$, for $x \in
\frac{1}{2} {\mathfrak F}_A$. 
  \end{example}

\begin{example} \label{Ex4}  The approximately unital Banach algebra $A = 
L^1(\Rdb)$ 
with convolution product has `multiplier unitization' $A^1 = A \oplus^1 \Cdb$.  
 This can be seen from Wendel's result
that the measure algebra $M(\Rdb)$ embeds canonically in $B(L^1(\Rdb))$ isometrically \cite{Dal},
 so that
$L^1(\Rdb)^1$ can be identified with $L^1(\Rdb) + \Cdb \delta_0$,
where $\delta_0$ is the point mass at $0$.   Thus $S(A)$ corresponds to 
the set of $f \in L^{\infty}(\Rdb)$ of norm $1$.  
It follows immediately 
that ${\mathfrak F}_A = {\mathfrak r}_A = (0)$ in this case.
This algebra is not Arens regular.  Note that any 
norm one functional on $L^1(\Rdb)$ extends to a state on $L^1(\Rdb)^1$ 
clearly.   However there are many
norm one functions $g \in L^\infty(\Rdb)$ with
$1 \neq \lim_{t \to 0^+} \, \int_{\Rdb} g e_t$, for the usual positive
cai ${\mathfrak e} = (e_t)$ of $L^1(\Rdb)$ (the one in the Remark after Lemma \ref{netwee}).  
For example if $g$ takes only negative values.    This shows that Lemma \ref{hbs}  fails for more general Banach algebras.   For this same cai ${\mathfrak e}$ we remark that $S_{\mathfrak e}(A)$ corresponds to 
the set of $f \in {\rm Ball}(L^{\infty}(\Rdb))$ for which the mean value of $f$ at $0$
(this mean value is the limit
with $n$  of the (integral) average of $f$ over the interval
of width $1/n$ centered at 0) exists
and equals $1$.   
\end{example}

Because of the above examples, and the considerations mentioned 
after Lemma \ref{oz2} 
above,
the following result cannot be improved, even for $M$-approximately
unital  Arens regular Banach algebras: 

\begin{proposition} \label{wh}  If $x \in {\mathfrak r}_A$
then ba$(x)$ has a bai in ${\mathfrak F}_A$, and hence any 
 weak* limit point of this bai is a mixed identity 
residing in ${\mathfrak F}_{A^{**}}$.   Indeed 
$(x^{\frac{1}{n}})$ is a bai for ${\rm ba}(x)$ in ${\mathfrak r}_A$,
and $({\mathfrak F}(x)^{\frac{1}{n}})$ is a bai for ${\rm ba}(x)$ in ${\mathfrak F}_A$.
  \end{proposition}

\begin{proof}     Note that  $x^{\frac{1}{n}} x \to x$ by Lemma \ref{sin}. 
That $(x^{\frac{1}{n}})$ is bounded follows from Lemma \ref{Bal}.  Thus 
$(x^{\frac{1}{n}})$ is a bai for ${\rm ba}(x)$ in ${\mathfrak r}_A$.

In the case that $x \in {\mathfrak F}_A$, then $(x^{\frac{1}{n}})$
is in  ${\mathfrak F}_A$ (using Proposition \ref{clnth}). 
We remark that the proof of \cite[Lemma 2.1]{BRI} (see also 
\cite{BHN}) displays 
a different, and often useful, bai
in ${\mathfrak F}_A$.
In the general case note that if $x \in {\mathfrak r}_A$ then 
ba$(x) = {\rm ba}({\mathfrak F}(x))$ by Proposition \ref{whba}, and so 
$({\mathfrak F}(x)^{\frac{1}{n}})$ is a bai for ba$(x)$.   
\end{proof}

For an approximately
unital Banach algebra $A$ and
$x \in {\mathfrak r}_A$, 
by Proposition \ref{whba}  we have ba$(x) = 
{\rm ba}({\mathfrak F}(x))$ and $\overline{xA} = 
\overline{{\mathfrak F}(x) A}$.   If $A$ is not Arens regular then 
Example \ref{Ex3} shows that ba$(x)$ need not be Arens regular if
$x \in {\mathfrak F}_A$.  (However it is Arens semiregular 
as is any commutative Banach algebra \cite{Pal}.)   Thus
ba$(x)^{**}$  need not be commutative.   We    
write $s(x)$ for the weak* Banach limit of $(x^{\frac{1}{n}})$
in $A^{**}$.    That is $s(x)(f) = {\rm LIM}_n \, f(x^{\frac{1}{n}})$ for 
$f \in A^*$, where LIM is a Banach limit.    It is easy to see that $x s(x) = s(x) x = x$, by applying these to
$f \in A^*$.  Hence $s(x)$  is a mixed  identity of ${\rm ba}(x)^{**}$, and 
is idempotent.
By the Hahn-Banach theorem it is easy to see that 
$s(x) \in \overline{{\rm conv}( \{ x^{\frac{1}{n}} : n \in \Ndb \})}^{w*}$.
By Corollary \ref{hasb} and  \ref{lump},  and the fact below Lemma \ref{isun}
that ${\mathfrak F}_{A^{**}}$ is weak* closed, we see that $s(x)$ resides in ${\mathfrak F}_{A^{**}}$.
If ${\rm ba}(x)$ is Arens regular then $s(x)$ will be the identity of ${\rm ba}(x)^{**}$.
 Therefore in this 
case, or more generally if   ${\rm ba}(x)^{**}$ has a unique left identity in the 
second Arens product, then $s(x)$ is also
the weak* limit of $({\mathfrak F}(x)^{\frac{1}{n}})$.  Indeed  in this 
case
we can set $s(x)$ to be the weak* limit of any bai
for ${\rm ba}(x)$.      This is the case for example, if ${\rm ba}(x)$ is $M$-approximately
unital (that is, if it is an $M$-ideal in ${\rm ba}(x)^{1}$), by Lemma \ref{isun}.

\bigskip

{\bf Remark.}   Note that if  $x \in {\mathfrak r}_A$ then ${\rm ba}(x)$
is $M$-approximately
unital if $A$ is $M$-approximately unital
and ${\rm ba}(x)^1 \subset A^1$ isometrically
(by the argument in Proposition \ref{aufr}).      It is claimed in  \cite{SW3} that the `support projection'
of an $M$-ideal in a commutative Banach algebra is central.   We did not follow this proof (and its author confirmed that
at present there seemed to him to be a gap), but this would imply that 
if ${\rm ba}(x)$ is $M$-approximately unital then
$s(x)$  is central in ${\rm ba}(x)^{**}$, and thus is actually a (unique) two-sided
identity for ${\rm ba}(x)^{**}$.  

\bigskip

We call $s(x)$ above a {\em support} idempotent of $x$,
or a (left) support idempotent of $\overline{xA}$
(or a (right) support idempotent of $\overline{Ax}$).   
The reason for this name is the following result.
   
\begin{corollary}  \label{lba}  If  $A$ is an approximately unital 
Banach algebra, and  $x \in {\mathfrak r}_A$
then $\overline{xA}$ has a left bai in ${\mathfrak F}_A$ and 
$x \in \overline{xA} = s(x) A^{**} \cap A$ and $(xA)^{\perp \perp} = s(x) A^{**}$. 
(These products are with respect to the second Arens product.)
  \end{corollary}

\begin{proof}    Indeed if $J = \overline{xA}$ then $J = \overline{{\mathfrak F}(x)A}$ by Proposition \ref{whau}.
So we may assume that $x \in {\mathfrak F}_A$.   Since $\overline{xA}$ contains
$\overline{x {\rm ba}(x)}$, which in turn contains
(actually, is equal to) ${\rm ba}(x)$, it contains $x$ and $x^{\frac{1}{n}}$.  So 
$(x^{\frac{1}{n}})$ is a left bai in ${\mathfrak F}_A$ for $\overline{xA}$.
    We have  
$s(x) \in J^{\perp \perp}$, and $J^{\perp \perp} \subset s(x) A^{**} \subset J^{\perp \perp},$
since $J^{\perp \perp}$ is a right ideal in $A^{**}$.
Hence $J^{\perp \perp} = s(x) A^{**}$, so that $J = s(x) A^{**} \cap A$.
\end{proof}

As in \cite[Lemma 2.10]{BRI} we have:

\begin{corollary}  \label{supp3}  If $A$ is an
approximately unital  Banach
 algebra, and $x, y \in {\mathfrak r}_A$, 
then $\overline{xA} \subset \overline{yA}$
iff $s(y) s(x) = s(x)$.  In this case $\overline{xA} = A$
iff   $s(x)$ is a left identity for $A^{**}$.   (These products are with respect to the second Arens product.)
\end{corollary}

\begin{proof}   This is essentially just as in the proof of Theorem 2.10 (and 2.6) of \cite{BRI}.   For example
if $\overline{xA} \subset \overline{yA}$ then since $x \in \overline{xA}$ we have $s(y) x = x$.
Hence $s(y) z = z$ for all $z \in {\rm ba}(x)$, and so $s(y) s(x) = s(x)$, since as we said earlier 
$s(x) \in \overline{{\rm ba}(x)}^{w*}$.  \end{proof}

As in \cite[Corollary 2.7]{BRI} we have:

\begin{corollary}  \label{perm}  Suppose that $A$ is a closed
approximately unital subalgebra of an
approximately unital Banach
 algebra $B$, and that ${\mathfrak r}_A
\subset {\mathfrak r}_B$.  If  $x \in {\mathfrak r}_A$, then
the support projection of $x$ computed in $A^{**}$ is
the same, via the canonical embedding $A^{**} \cong A^{\perp \perp}
\subset B^{**}$, as the support projection of $x$ computed in $B^{**}$.
\end{corollary}

We recall that  $x$ is pseudo-invertible in $A$ if 
there exists $y \in A$  with $xy x = x$.   The following result (and several of its corollaries below) 
 should be compared with the $C^*$-algebraic version of the result due to  Harte and Mbekhta \cite{HM,HM2},
and to the earlier version of the result in the operator algebra case (see particularly \cite[Section 3]{BRI}, and 
\cite[Subsection 2.4]{Bord}).
 
\begin{theorem}   \label{ws}  Let $A$ be
 an approximately unital Banach algebra $A$, and
 $x \in {\mathfrak r}_A$.  The following
are equivalent:
  \begin{itemize}
\item [(i)] $s(x) \in A$,
\item [(ii)]  $xA$ is closed,
\item [(iii)] $Ax$  is closed,
\item [(iv)]  $x$ is pseudo-invertible in $A$,
 \item [(v)]  $x$ is invertible in ${\rm ba}(x)$.
\end{itemize}
Moreover, these conditions imply
 \begin{itemize}
\item [(vi)] $0$ is isolated in, or absent from, ${\rm Sp}_A(x)$.
\end{itemize}
Finally, if  ${\rm ba}(x)$ is semisimple
then {\rm (i)--(vi)} are equivalent.
\end{theorem}

 \begin{proof}
We recall that $(x^{\frac{1}{m}})_{m \in \Ndb}$ is a bai for ba$(x)$, by Proposition \ref{wh},
and it has weak* limit point $s(x) \in {\rm ba}(x)^{\perp \perp} \subset A^{**}$.

(ii)  $\Rightarrow$ (i) \ Suppose $xA$ is closed.  Then   $$x^{\frac{1}{2}} \in {\rm ba}(x) \subset \overline{x  {\rm ba}(x)} \subset
\overline{xA} = xA,$$
so $x^{\frac{1}{2}} = xy$ for some $y \in A$.
Thus if $z = x^{\frac{1}{2}} y \in A$  then $x =  x^{\frac{1}{2}} x y = x z$,
and so $a = a z$ for every $a \in {\rm ba}(x)$.
Now  $s(x) z = z$ since $x^{\frac{1}{2}} \in {\rm ba}(x)$ for example.
On the other hand $s(x) z = s(x)$ since $x^{\frac{1}{n}} z = x^{\frac{1}{n}}$ so that 
$$(s(x) z)(f) = fs(x) (z) =
 {\rm LIM}_n f(x^{\frac{1}{n}} z) 
=   {\rm LIM}_n f(x^{\frac{1}{n}}) 
= s(x)(f) , \qquad f \in A^*.$$    Thus $s(x) = z \in A$.  (Of course in this case 
$x^{\frac{1}{n}} \to s(x)$ in norm.)

(i) $\Rightarrow$ (iv) \  Recall $s(x)$ is a left identity of ${\rm ba}(x)^{**}$ in the second Arens product, and if 
(i) holds it is an identity, and ${\rm ba}(x)$ is unital.  This implies by the Neumann
lemma that $x$ is invertible in ${\rm ba}(x)$, hence that $x$ is pseudo-invertible in $A$.

(iv) $\Rightarrow$ (ii) \ Item (iv) implies
that $xA = xy A$ is closed since $xy$ is idempotent.

That (iii) is equivalent to the others follows from (ii) and the
symmetry in (i) or (iv).   That (v) is equivalent to (i) is now
obvious from the above.

For the equivalences with (vi), by definition of spectrum, and because of the form of  (v), we may assume $A$ is unital.
That (iv) implies (vi) may be proved similarly to the analogous argument in  \cite[Theorem {\rm 3.2}]{BRI}, but replacing $B(H)$ and $B(K)$ with $B(A)$ and $B(xA)$.   We can assume that 
$0 \in Sp_A(x)$, so that $x$ is not invertible.   Then $xA \neq A$, for if $xA = A$
then $s(x)$ is a left identity for $A$.  It is also a right identity since if $(e_t)$ is a cai for $A$
then  $s(x) e_t = e_t \to s(x)$.   Then the inverse of $x$ in ba$(x)$ is an inverse in $A$,
contradicting the  fact that  $x$ is not invertible in $A^1$.   It may be simpler to 
prove the equivalent fact that $0$ is isolated in the spectrum of $x^{\frac{1}{2}}$.
By the argument in \cite[Theorem {\rm 3.2}]{BRI} it is enough to prove that 
$0$ is isolated in the spectrum of $L$ in $B(A)$, where $L$ is left multiplication by
$x^{\frac{1}{2}}$.   We note that $$x^{\frac{1}{2}} A \subset xA \subset eA \subset x^{\frac{1}{2}}A ,$$ where $e = x^{\frac{1}{2}} y = s(x)$ and $y$ is the pseudoinverse of $x$.   So these subspaces 
coincide; call this space $K$.  It follows that $K$ is an invariant subspace for $L$,
indeed $R = L_{|K}$ is continuous, surjective and one-to-one (since $x^{\frac{1}{2}} x^{\frac{1}{2}} a
= 0$ implies that $x^{\frac{1}{2}} a = 0$, since $x^{\frac{1}{2}}$ is a limit of polynomials
in $x$ with no constant term).  Thus $0 \notin {\rm Sp}_{B(K)}(R)$; hence 
$R + z I_K$ is invertible for $z$ in a small disk centered at $0$.  Since $A = eA \oplus (1-e)A$, it is then easy to argue 
that $L + z I_A = (L + z I)e  \oplus z (1-e)$ is invertible in $B(A)$ for such $z$, if $z \neq 0$.  So $0$ is isolated in  the spectrum of $L$ in $B(A)$. 

The last assertion follows just as in \cite[Theorem {\rm 3.2}]{BRI}.
\end{proof}  

{\bf Remark.}  We have been informed by Matthias Neufang that he and M. Mbehkta have also generalized 
the analogous result from \cite{BRI,BRIII},  
or a variant of it, to the class of  Banach algebras that are ideals in their bidual.  

\bigskip

The next result
 is an analogue of \cite[Theorem 2.12]{BRI}:

\begin{proposition} \label{oka}
If $A$ is an approximately unital Banach algebra, a 
subalgebra of a unital Banach algebra $B$ with
${\mathfrak r}_A \subset {\mathfrak r}_B$,
and $x \in {\mathfrak r}_A$,  then $x$ is invertible in $B$ 
iff  $1_B \in A$ and $x$
is invertible in $A$, and iff  ${\rm ba}(x)$ contains $1_B$; and in this case $s(x) = 1_B$.
\end{proposition} \begin{proof}  It is clear by the Neumann
lemma that if ${\rm ba}(x)$ contains $1_B$ then $x$ is invertible in
${\rm ba}(x)$, and hence in $A$.
Conversely, if $x$ is invertible in $B$ (or in $A$)
then by the equivalences 
(i)--(iv)
proved in the last theorem,
we have $s(x) \in B$, and this is the identity
of ${\rm ba}(x)$.  If $x y = 1_B$,
then $1_B = xy = s(x) x y = s(x) \in {\rm ba}(x) \subset
A$.
\end{proof}

\begin{corollary} \label{Aha3}   Let $A$ be an approximately unital Banach algebra.  A closed right ideal $J$ of $A$
is of the form $x A$ for some $x \in {\mathfrak r}_A$  iff
$J  = qA$ for an idempotent 
 $q \in {\mathfrak F}_A$.
\end{corollary}   \begin{proof}   If $xA$ is closed for a nonzero $x \in {\mathfrak r}_A$
then by the theorem $q = s(x) \in {\mathfrak F}_A$.   Hence it 
is easy to see that
 $xA = qA$. The other direction is trivial.   \end{proof}

\begin{corollary} \label{Aha2}   If a  nonunital 
approximately unital Banach  algebra
$A$ contains a nonzero $x \in {\mathfrak r}_A$ with $xA$ closed,
 then $A$ contains  a nontrivial 
idempotent in ${\mathfrak F}_A$.  
 \end{corollary}   \begin{proof}
By the above $xA = qA$ for a nontrivial 
idempotent $q$ in ${\mathfrak F}_A$. 
\end{proof}

\begin{corollary} \label{Aha}   If an 
approximately unital Banach  algebra $A$ has  no 
left identity, then $x A  \neq A$ for all $x \in {\mathfrak r}_A$.
\end{corollary}  

{\bf Remark.}    If $A$ is a Banach algebra such that $\frac{1}{2} {\mathfrak F}_A$ closed under $n$th roots then
one may also generalize other parts of the theory in \cite{BRI}.  For example in this case, if $x \in   {\mathfrak F}_A$
then the support projection $s(x)$ is a bicontractive projection, and ba$(x)$ has a cai in $\frac{1}{2} {\mathfrak F}_A$.

\section{One-sided ideals and hereditary subalgebras}

At the outset 
it should be said there seems to be no completely satisfactory  theory 
of hereditary subalgebras.  This can already be seen in finite dimensional unital examples
where one may have  $p A = qA$ for projections $p,q \in 
{\mathfrak F}_A$, but  no  good relation between 
$pAp$ and $qAq$.   For example one could take the opposite algebra to the one in Example \ref{Ex7}.
Another example arises when one considers
 various mixed identities in the second dual  $A^{**}$, with the second 
Arens product, inside
$(A^1)^{**}$.   In this section we will investigate what initial  parts of the  theory do work.
We shall see that things work considerably better if $A$ is separable.

We define an {\em inner ideal} in $A$ to be a closed subalgebra $D$ with $D A D \subset D$.
To see what kinds of results one might hope for, note that in the unital example in the last paragraph,
given an idempotent $p \in A$, 
the right ideal $J = p A$ contains a unital inner ideal $D = pAp$ of $A$.  
Conversely if $D = pAp$ then $J = DA = pA$
is a right ideal with a left identity.  

In nonunital examples things become more complicated.  One may define a 
hereditary subalgebra to be an  inner ideal $D$ of $A$ which has a bai.  This then induces 
a right ideal $J = DA$ with a left bai, and a left ideal $K = AD$ with a right bai.   We shall call 
these the {\em  induced} one-sided ideals.  We have $JK = J \cap K = D$ just as in \cite[Corollary 2.6]{BHN}.
However unlike the previous paragraph, without further conditions one cannot in general 
obtain a hereditary subalgebra from a right 
ideal  with a left bai.  The following example illustrates some of what can go wrong.   

\begin{example} \label{Ex6}   
One of the main results in \cite{BHN} is that if $J$ is
a closed right ideal with a left cai in an operator algebra $A$, 
then there exists an associated  hereditary subalgebra $D$ of $A$,
in particular a closed approximately unital subalgebra $D \subset J$
with $J = DA$.   This is false without further conditions
in more general Banach algebras.
Indeed, suppose that $J = A$ is a  separable Banach algebra with a sequential 
left cai, but no commuting bounded left approximate identity. See 
\cite{Dix} for such an example.   By way of contradiction, suppose that there is a closed subalgebra $D \subset J$
with a bai, such that 
$J = DA$.  By \cite{Sinc}, $D$ has a commuting bounded approximate identity,
and this will be a commuting bounded left approximate identity for $J$,
a contradiction.   

This example also shows that if $J$ is
a closed right ideal with a left cai, we cannot rechoose another left cai
$(e_t)$ with $e_s e_t \to e_s$ with $t$, for all $s$.    This is critical 
in the operator algebra theory in e.g.\ \cite[Section 2]{BHN}. 
 \end{example}

In order to obtain a working theory, we now impose the condition that the bai's considered are 
in ${\mathfrak r}_A$.
Thus we define a {\em right ${\mathfrak F}$-ideal} (resp.\ left ${\mathfrak F}$-ideal) in an approximately unital
Banach algebra $A$  to be a closed right (resp.\ left) ideal with a left (resp.\ right) bai in ${\mathfrak F}_A$
(or equivalently, by Corollary \ref{hasbi}, in ${\mathfrak r}_A$).     
Henceforth in this section, by a {\em hereditary subalgebra} (HSA) of $A$ we will mean 
 an inner ideal  $D$ with a two-sided  bai in ${\mathfrak F}_A$
(or equivalently, by Corollary \ref{hasbi}, 
in ${\mathfrak r}_A$).
Perhaps these should be called ${\mathfrak F}$-HSA's to avoid confusion with the notation  in 
\cite{BHN,BRI} where one uses cai's instead of bai's, but for brevity we shall use the 
shorter term.  Also it is shown in \cite{B2015} that in an operator 
algebra $A$  these two notions coincide, and that right
${\mathfrak F}$-ideals in A are just the r-ideals of \cite{BHN}
(and similarly in the left case).

Note that a HSA $D$ induces a pair of right and left ${\mathfrak F}$-ideals $J = DA$ and $K = AD$.   As we pointed
out a few paragraphs back, it is
not clear that the converse holds, namely that every right  ${\mathfrak F}$-ideal comes from a
HSA in this way.  In fact the main results of this section 
are, firstly, that if $A$ is  separable then this is true, and indeed  all HSA's and ${\mathfrak F}$-ideals
are of the form in the next lemma.  Secondly,
we shall prove (see Corollaries  \ref{pred} and \ref{pred2}) that if $A$ is not necessarily separable then 
the HSA's and ${\mathfrak F}$-ideals in $A$ are just the closures of increasing unions of ones of the form in this lemma:

\begin{lemma}  \label{zH}  If $A$ is an approximately unital Banach algebra, and $z \in  {\mathfrak F}_A$,
set  $J = \overline{zA}$, 
 $D = \overline{zAz}$, and $K = \overline{Az}$. \ Then $D$ is a HSA in $A$ 
and $J$ and $K$ are the induced right and left ${\mathfrak F}$-ideals mentioned above.
\end{lemma}  

\begin{proof}     By Cohen factorization
$D = D^4 \subset JK \subset J \cap K$, and if $x \in J \cap K$ then 
$x = \lim_n \, z^{\frac{1}{n}} x z^{\frac{1}{n}} \in D$.   So $z \in D = JK = J \cap K$.  
Also $J = p A^{**}  \cap A$ by Corollary  \ref{lba}, and    
$D = p A^{**} p \cap A$ is a HSA in $A$, and $K = A^{**} p \cap A$, where $p= s(z)$.
 To see this, note that $p z = z = zp$, so that $K \subset A^{**} p \cap A$.  If $a \in A^{**} p \cap A$, then $az^{\frac{1}{n}}$
has weak* limit point $a p = a$.   Hence a convex combination converges in norm, so that $a \in K$,
so that $K = A^{**} p \cap A$.   A similar argument works for $D$.   Finally, $DA = J$, since $zA \subset DA \subset J$,
and similarly $AD = K$. \end{proof}

{\bf Remarks.}   1) \  In general
$D$ and $K$ are  determined by  the particular $z$ used above, and not by $J$ alone.  

\smallskip

2) \  We note that if $z \in  {\mathfrak F}_A$ then with the notation in the last proof,
$K^{\perp \perp} =  \overline{A^{**} p}^{w*}$ and $D^{\perp \perp} =
\overline{p A^{**} p}^{w*}$.   (The weak* closure here is not necessary if $A$ is Arens regular.)  
Indeed $K^{\perp \perp} \subset  \overline{A^{**} p}^{w*} $.
Also $p \in {\rm ba}(z)^{\perp \perp}
\subset D^{\perp \perp} \subset K^{\perp \perp}$, so that 
$A^{**} p \subset K^{\perp \perp}$.    Thus $K^{\perp \perp} =  \overline{A^{**} p}^{w*}$.
 It is well known that 
$J + K$ is closed, which implies as in the proof
of e.g.\ \cite[Lemma 5.29]{BZ} that $(J \cap K)^\perp = \overline{J^{\perp} + K^{\perp}}$, 
so that $D^{\perp \perp} = J^{\perp \perp} \cap K^{\perp \perp} = \overline{p A^{**} p}^{w*}$.

\begin{example} \label{Ex7}  The following example illustrates some other issues that arise
for left ideals in general Banach algebras, which obstruct following the r-ideal and hereditary subalgebra
theory of operator algebras \cite{BHN,BRI}.
First, for $E \subset  {\mathfrak F}_A$ it may be that $\overline{EA}$ has no  left cai.
Even if $E$ has two elements this may fail, and in this case
$\overline{EA}$ may not even 
equal $\overline{aA}$ for any $a \in A$.   Thus in general the class of right ${\mathfrak F}$-ideals
in noncommutative  algebras is not closed under either finite sums or finite intersections (see 
Example \ref{Ex1}).
Also, it need not be the case that $EAE$ has a bai if $E \subset {\mathfrak F}_A$.  
A simple three-dimensional example illustrating all of these points is the lower 
triangular $2 \times 2$ matrices with its norm  as an operator on $\ell^1_2$
(see \cite[Example 4.1]{SW1}), and $E = \{ E_{11} \pm E_{21} \}$.  
\end{example}  

\begin{theorem} \label{oleftc}  Suppose that $J$ is a right ${\mathfrak F}$-ideal in an approximately unital
Banach algebra $A$.   For every compact subset $K \subset J$, there exists 
$z \in J \cap {\mathfrak F}_A$ with $K \subset z J \subset zA$.
\end{theorem}  

\begin{proof}  We may assume that $A$ is unital, and follow
the idea in the proof of Cohen's factorization theorem
(see e.g.\ \cite[Theorem 4.1]{PedE}, or \cite{Dal}).
For any $f_1, f_2, \cdots \in J \cap {\mathfrak F}_A$  define $z_n = 
\sum_{k=1}^n \, 2^{-k} f_k + 2^{-n} \in  J + \Cdb 1$.  We have $$\Vert 1 - z_n \Vert
= \Vert \sum_{k=1}^n \, 2^{-k} (1-f_k) \Vert 
\leq \sum_{k=1}^n \, 2^{-k}  = 1 - 2^{-n},$$
and so by the Neumann lemma $z_n^{-1} \in J + \Cdb 1$ and $\Vert z_n^{-1} \Vert
\leq 2^{n}$.  

Let $(e_t)$ be a  left cai for $J$ in ${\mathfrak F}_A$, set $z_0 = 1$, and choose $\epsilon > 0$.
For each $x \in K$ we have $\lim_t \,  \Vert (1 - e_t) z_n^{-1} x \Vert = 0$.
Thus by the Arzela-Ascoli theorem, and passing repeatedly
to subnets, we can inductively  choose a
subsequence $(f_n)$ of $(e_t)$, and
 use these to inductively define $z_n$ by the formula above,
so that
$$\max_{x \in K} \,
 \Vert (1 - f_{n+1}) z_n^{-1} x \Vert \leq 2^{-n} \epsilon, \qquad n \geq 0.$$ 
 Set $z = \sum_{k=1}^\infty \, 2^{-k} f_k \in \overline{{\rm conv}}(e_n) \subset J \cap {\mathfrak F}_A$.   If $x \in K$ set 
$x_n = z_n^{-1} x$.   Then 
$$\Vert x_{n+1}- x_n \Vert = \Vert z_{n+1}^{-1} (z_n - z_{n+1}) z_n^{-1} x
\Vert = \Vert 2^{-n-1} \, z_{n+1}^{-1} (1- f_{n+1}) z_n^{-1} x
\Vert \leq 
2^{-n} \epsilon.$$
Hence $w = \lim_n \, x_n$ exists and $z w = x$.   Note also  that 
$$\Vert x_n - x \Vert \leq \sum_{k=1}^n \, \Vert x_k - x_{k-1} \Vert \leq 2 \epsilon,$$ so that $\Vert w - x \Vert   \leq 2 \epsilon$ if
one wishes for that (so that
$\Vert w \Vert \leq  \Vert x \Vert +   \epsilon$).
\end{proof}

{\bf Remark.}   
 In the case of operator algebras, or in the commutative case considered in Section 7,
one can choose  the 
$z$ in the last result in ${\rm conv}(K)$, if $K$ is for example a  finite set 
in $J \cap {\mathfrak F}_A$.  
 If $A$ is noncommutative this fails as we saw in  Example \ref{Ex7}.   

\begin{corollary} \label{ctrid}  Let $A$ be an approximately unital
Banach algebra.   The  closed right ideals with a countable 
left bai in   ${\mathfrak r}_A$ are precisely the `principal right ideals'
$\overline{zA}$ for some $z \in {\mathfrak F}_A$.   Every separable right ${\mathfrak F}$-ideal 
is of this form.
\end{corollary}

\begin{proof}  The one direction is easy since $(z^{\frac{1}{n}})$ is a left bai for $\overline{zA}$ (see the 
proof of Corollary  \ref{lba}).    Conversely, if $(e_n)$ is a  countable 
left bai in   ${\mathfrak r}_A$ for right ideal $J$, set $K = \{ \frac{1}{n} \, e_n \}$ and apply 
Theorem \ref{oleftc}.   

 For the last assertion, if $\{ d_n \}$ is a countable dense set in a right ${\mathfrak F}$-ideal  $J$,  apply 
Theorem \ref{oleftc}, with $K = \{ \frac{d_n}{n \Vert d_n \Vert} \}$.    There exists
 $z \in J \cap {\mathfrak F}_A$ with $K \subset \overline{zA}$.  Hence 
$J \subset \overline{zA} \subset J$.
 \end{proof}

\begin{corollary}  \label{pred}   The right ${\mathfrak F}$-ideals  in an approximately unital  Banach algebra $A$, are precisely the 
closures of  increasing unions of closed right ${\mathfrak F}$-ideals of the form $\overline{zA}$ for some $z \in {\mathfrak F}_A$. \end{corollary}

\begin{proof}  Suppose that  $J$ is an arbitrary right ${\mathfrak F}$-ideal in $A$.   Let $\epsilon > 0$
be given (this is not needed for the proof but will be useful elsewhere).   Let $E$ be the
 left bai  in ${\mathfrak F}_A$ considered as a set, and let $\Lambda$ be the set of  finite subsets of $E$ ordered
by inclusion.   Define $z_G = x$ if 
$G = \{ x \}$ for $x \in E$.  For any two element set $G = \{ x_1, x_2 \}$ in  $\Lambda$,
 one can apply  Theorem  \ref{oleftc} to obtain an element
$z_G \in {\mathfrak F}_A$ with $GA \subset z_G A$, and moreover such that $x_k = z_G w_k$ with $\Vert w_k - x_k \Vert < \epsilon$,
for each $k$,  if
one wishes for that. 
  For any three element set $G = \{ x_1, x_2, x_3 \}$ in $\Lambda$ we can similarly 
choose  $z_G \in {\mathfrak F}_A$ with $z_H A \subset z_G A$ for all proper
subsets $H$ of $G$ (and with the `moreover' above
too).
  Proceeding in this way, we can inductively choose for any $n$ element set $G$ in $\Lambda$ 
an element  $z_G \in {\mathfrak F}_A$ with $z_H A \subset z_G A$ for all proper 
subsets $H$ of $G$ (and moreover
such that each such $z_H$ can be written
as $z_G w$ for some $w$ with  $\Vert w - z_H \Vert < \epsilon$ if
one wishes for that).  
Thus $(\overline{z_G A})$ is increasing (as sets) with $G \in \Lambda$, and $\overline{\cup_{G \in \Lambda} \, z_G A} = J$.  

Conversely, suppose that $\Lambda$ is a directed set and that $J = \overline{\cup_t \, J_t}$,
where $(J_t )_{t \in \Lambda}$ is
an increasing net  of subspaces of $A$, and $J_t = \overline{z_t A}$ for $z_t \in {\mathfrak F}_A$.
Thus if $t_1 \leq t_2$ then $J_{t_1} \subset J_{t_2}$, so that $s(z_{t_2}) z_{t_1} = z_{t_1}$.  Hence 
$s(z_{t}) x \to x$ with $t$ for all $x \in J$.  Thus a weak* limit point $p$ of $(s(z_t))_{t \in \Lambda}$ acts as a left identity for $J$, and hence 
is a left identity  for $J^{\perp \perp}$.   Thus $J^{\perp \perp} = p A^{**}$.
Since this left identity $p$ is in the weak* closure of the convex set
${\mathfrak F}_A \cap J$, the usual argument (see e.g.\ p.\ 81 of \cite{BLM}) shows that 
 $J$ has a left bai in ${\mathfrak F}_A \cap J$.  So $J$ is a right ${\mathfrak F}$-ideal  in $A$. 
\end{proof}
 
{\bf Remarks.}   1) \  
Note that $(z_G^{\frac{1}{n}})$ in the last proof is a left bai for the right ideal $J$ there.
This net is indexed by $n \in \Ndb$ and $G \in \Lambda$.  
To see this, suppose  $x \in J$ is given, and that $\Vert z_{G_1} a - x \Vert < \epsilon$, where
$a \in A$.  If $G_1 \subset G$ then $z_{G_1} \in z_G A$.
By the proof of Corollary \ref{pred} we can choose $w$ with $z_{G_1} = z_{G} w$ and $\Vert w \Vert \leq 3$.
Choose $N$ such that $c_n < \epsilon/3$ for $n \geq N$,
where $c_n$ is as in Lemma \ref{sin}.  Then by that result, $\Vert z_{G}^{\frac{1}{n}}   z_{G_1} - z_{G_1} \Vert
= \Vert z_{G}^{\frac{1}{n}}   z_{G}w - z_{G} w \Vert   \leq 3 c_n < \epsilon$.  
Thus
$$\Vert z_{G}^{\frac{1}{n}}  x - x \Vert \leq \Vert z_{G}^{\frac{1}{n}}  x -  z_{G}^{\frac{1}{n}}  z_{G_1} a  \Vert + \Vert z_{G}^{\frac{1}{n}}  
z_{G_1} a - z_{G_1} a
 \Vert + \Vert z_{G_1} a - x \Vert < (3 + \Vert a \Vert) \epsilon,$$
for all $G$ containing $G_1$, and $n \geq N$.    So $(z_G^{\frac{1}{n}})$  is a left bai for $J$.

\smallskip

2) \ If $(z_G)_{G \in \Lambda}$ is as above, it is tempting to define 
$D = \overline{\cup_{G \in \Lambda} \, z_G A z_G}$.     However
we do not see that this can be adjusted to make it a HSA.    

\bigskip

In the operator algebra 
case, most of the following result and its proof was first in 
the preprint \cite{BRIII} (which as we said on the first page, has now morphed into several papers).  
We thank
 Charles Read for discussions on that result in May 2013,
and thank 
Garth Dales and Tomek Kania for conversations in the same period  on
algebraically finitely generated ideals in
Banach algebras, and in particular for 
drawing our attention to the results in \cite{ST} (these will not 
be used in the present proof below,
but were used in an earlier version).  
We say that a right module $Z$ over $A$ is
{\em algebraically countably generated} (resp.\ {\em algebraically finitely
generated})  
over $A$ if there exists a
countable (resp.\  finite set) $\{ x_k \}$ in $Z$ such that
every $z \in Z$ may be written as a finite sum $\sum_{k=1}^n \, x_k a_k$
for some $a_k \in A$.

\begin{corollary} \label{Aha4}   Let $A$ be an approximately unital  Banach algebra.   A right ${\mathfrak F}$-ideal $J$ in $A$
is algebraically countably generated as a right module over $A$   iff   $J  = q A$ for an idempotent $q \in {\mathfrak F}_A$.  This
is also equivalent to $J$ being algebraically countably generated as a right module over $A^1$.
\end{corollary}   \begin{proof}
Let $J$ be a right ${\mathfrak F}$-ideal which is algebraically countably generated 
over $A$ by elements $x_1, x_2, \cdots$ in $A$.   We can assume that
$\Vert x_k \Vert \to 0$, and so $\{ x_k : k \in \Ndb \}$ is compact. 
  By Theorem \ref{oleftc} there exists $z \in J$
 such that $\{ x_k \} \subset zA$.  Thus $x_k A \subset zA^2=zA$
 for all $k$, and so $J \subset zA \subset J$, and $J = zA$.
By Corollary \ref{Aha3}, $J = qA$ for an idempotent $q \in {\mathfrak F}_A$.

If  $J$ is algebraically countably generated  over $A^1$  then by the above
$J  = q A^1$.  Clearly $q \in A$, and so $J = \{ x \in A : q x = x \} = qA$.  \end{proof}

\begin{lemma}  \label{preot}  Let $A$ be an
approximately unital Banach algebra, with a
closed subalgebra $D$.   If $D$ has a bai from ${\mathfrak F}_A$, then for every compact subset $K\subset D$,
 there is 
$x\in D \cap {\mathfrak F}_A$ such that $K\subset xDx \subset   xAx$. \end{lemma}  

\begin{proof}     This can be done by adapting the proof of Theorem \ref{oleftc} as follows.
We can inductively  choose a
subsequence $(f_n)$ of the bai $(e_n)$ with 
$$\max_{x \in K} \, [ \Vert (1 - f_{n+1}) z_n^{-1} x \Vert  \, + \, 
\Vert x z_n^{-1} (1 - f_{n+1})  \Vert ]
\leq 2^{-2n} \epsilon$$ for each $n$.  Choose $z$ as before.
  If $x \in K$ set 
$x_n = z_n^{-1} x z_n^{-1} \in D$.   Then 
$$\Vert x_{n+1}- x_n \Vert \leq \Vert (z_{n+1}^{-1} x - z_n^{-1}   x )   z_{n+1}^{-1}  \Vert + 
\Vert  z_n^{-1} (x   z_{n+1}^{-1} - x  z_n^{-1}) \Vert,$$
which is dominated by $2^{n+1} \, \Vert z_{n+1}^{-1} x - z_n^{-1}   x   \Vert + 2^n \Vert  x   z_{n+1}^{-1} - x  z_n^{-1} \Vert.$
Again we have  $\Vert z_{n+1}^{-1} x - z_n^{-1}   x   \Vert  \leq 2^{-2n} \epsilon$, and similarly
$\Vert x   z_{n+1}^{-1} - x  z_n^{-1} \Vert \leq 2^{-2n} \epsilon$.  So $\Vert x_{n+1}- x_n \Vert
\leq (2^{1-n} + 2^{-n}) \epsilon < \frac{\epsilon}{2^{n-2}}$.    Thus $w = \lim_n \, x_n$ exists in $D$,
 and $zwz = \lim_n \, z_n x_n  z_n = x$ as desired.
We also have $\Vert w - x \Vert \leq 2 \epsilon$ as before, if we wish for this.  
\end{proof}

{\bf Remark.}  The above, and the next couple of results, are closely related to  
the results of Sinclair, Esterle, and others on the Cohen factorization method (see e.g.\ \cite{Sinc}), which also shows there is a commuting cai or bai
under certain hypotheses.  However the result above 
does not follow from Sinclair's results, and the latter do not directly connect to
`positivity' in our sense.   

\medskip

 Applying Lemma \ref{preot} to a suitable scaling of a countable bai in ${\mathfrak  F}_A$
as in the proof of Corollary \ref{ctrid}, we obtain:

\begin{theorem} \label{otter}  Let $A$ be
 an approximately unital Banach algebra, and let $D$ be an inner ideal in $A$.
Then  $D$  has a
 countable bai from ${\mathfrak F}_A$ 
(or equivalently, from ${\mathfrak r}_A$)
 iff there exists an element 
$z \in D \cap {\mathfrak F}_A$ with $D =  \overline{z A z}$.  Thus 
such $D$ has a countable commuting 
bai from ${\mathfrak F}_A$.  Any separable inner ideal in $A$ with a bai from ${\mathfrak r}_A$
is of this form.  
 \end{theorem}  

The following is an Aarnes-Kadison type theorem
for Banach algebras.  For another result of this type see \cite{Sinc}.

\begin{corollary} \label{ottertoo}  If $A$ is a subalgebra of a unital Banach algebra $B$, 
and we set ${\mathfrak r}_A = A \cap {\mathfrak r}_B$, then the following are equivalent: 
\begin{itemize} 
\item [(i)]    $A$ has a sequential    (commuting) bai from ${\mathfrak r}_A$.
\item [(ii)]    There  exists
an $x \in {\mathfrak r}_A$ with $A = \overline{x A x}$.
\item [(iii)]   There  exists
an $x \in {\mathfrak r}_A$ with $A = \overline{xA} =
\overline{Ax}$.
\item [(iv)]    There  exists
an $x \in {\mathfrak r}_A$ with $s(x)$ a mixed identity for $A^{**}$.
\end{itemize}   Any  separable Banach algebra with a bai from ${\mathfrak r}_A$ 
satisfies all of the above, as does any $M$-approximately unital Banach algebra
which is separable or has a countable bai.
\end{corollary}   

This is clear from earlier results.  Indeed the last theorem gives the  equivalence of (i) and (ii) above and the separability assertion, and that (ii) implies (iii) follows from e.g.\ Lemma \ref{zH}.   Also (iii) implies (i)
by considering $(x^{\frac{1}{n}})$; and (iii) is equivalent to (iv) by Corollary  \ref{supp3}.   Again, ${\mathfrak r}_A$
can be replaced by ${\mathfrak F}_A =  A \cap {\mathfrak F}_B$ throughout this result, or in any of the items (i) to (iv).

As a consequence of the last results, if $D$ is a HSA in an approximately unital  Banach algebra
$A$, and if $D$ has a countable bai from ${\mathfrak F}_A$,
then $D$ is of the form in Lemma  \ref{zH}.      We leave it to the 
reader to check that doing a `HSA variant' of the 
proof of Corollary \ref{pred}, using Lemma \ref{preot} and mixed identities rather than left
identities,  yields:

\begin{corollary}  \label{pred2}  The HSA's in an approximately unital  Banach algebra $A$ are exactly the 
closures of  increasing unions of HSA's of the form $\overline{zAz}$ for  $z \in {\mathfrak F}_A$. \end{corollary}

 \begin{proof}    We just sketch the more  difficult direction of
this since this is so close to  the proof of Corollary \ref{pred}.
Indeed we proceed as in the proof of Corollary \ref{pred}, taking $E$ to be the  bai $(e_t)$.  Define $\Lambda$ and $z_G \in D  \cap {\mathfrak F}_A$ for
$G \in \Lambda$ as before, but using Lemma \ref{preot}.  Note that each $e_t$ is in some $z_G A z_G$, which in turn is contained 
in the closed inner ideal $D'= \overline{\cup_{G \in \Lambda} \, z_G A z_G}$.  Since for $x \in D$ we have 
$x = \lim_t e_t x e_t  \in D'  \subset D$, the result is now clear.   \end{proof}  

{\bf Remark.}  As in the remark after Corollary \ref{pred}, if one takes care with the choice of the $z$ in the last Corollary, the $n$th roots of these
$z$'s  can be  a bai  for  the HSA.

\section{Better cai for $M$-approximately unital algebras}

In this section we consider the better behaved class of $M$-approximately unital Banach algebras.
We will use the fact that $M$-ideals
in Banach spaces are {\em  strongly proximinal}. (Actually the only 
`proximinality-type' condition we use here is `the strongly proximinal at 1 property' mentioned in the introduction.)

\begin{lemma} \label{prox}  Let $X$ be a Banach space, and suppose that $J$ is an $M$-ideal in $X$, and $x \in X, y \in J$, and
$\epsilon > 0$,  with
$\Vert x - y \Vert < d(x,J) + \epsilon$.  Then there exists 
a $z \in J$ with $\Vert y - z \Vert < 3 \epsilon$ and 
$\Vert x - z \Vert = d(x,J)$.
    \end{lemma}

\begin{proof}   This follows from the proof of \cite[Proposition II.1.1]{HWW}.
\end{proof}

\begin{theorem}  \label{lbai}  Let $A$ be an  $M$-approximately unital  
Banach algebra.  Then ${\mathfrak F}_A$ is weak* dense in  ${\mathfrak F}_{A^{**}}$, and ${\mathfrak r}_A$ is weak* dense in  ${\mathfrak r}_{A^{**}}$. 
Thus 
$A$ has a cai in $\frac{1}{2} {\mathfrak F}_A$. \end{theorem} 

\begin{proof}    This is easy if $A$ is unital, so we will focus on the 
nonunital case.  Suppose that $\eta \in A^{**}$
with  $\Vert 1 -  \eta \Vert \leq 1$  .
Suppose that $(x_t)$ is a bounded net in $A$
with weak* limit $\eta$ in $A^{**}$, so that $1 - x_t \to 1 - \eta$ weak* in $(A^1)^{**}$.  By Lemma \ref{ban},
for any $n \in \Ndb$ there exists  a $t_n$ such that
for every $t \geq t_n$, $$\inf \{ \Vert 1 -  y \Vert :
y \in {\rm conv} \{x_j : j \geq t \} \} < 1 + \frac{1}{2n}.$$
For every $t \geq t_n$, choose such a $y^n_t \in {\rm conv} \{x_j : j \geq t \}$ with 
$\Vert 1 -  y^n_t \Vert < 1 + \frac{1}{n}$.   If $t$ does not 
dominate $t_n$ define $y^n_t = y^n_{t_n}$.  So for all $t$
we have $\Vert 1 -  y^n_t \Vert < 1 + \frac{1}{n}$.    
Writing $(n,t)$ as $i$, we may view $(y^n_t)$ as a 
net indexed by $i$, with 
$\Vert 1 - y^n_t \Vert  \to 1$.  Given $\epsilon > 0$ and $\varphi \in A^*$,
there exists a $t_1$
such that $|\varphi(x_t) - \eta(\varphi)| < \epsilon$ for all 
$t \geq t_1$.  Hence $|\varphi(y^n_t) - \eta(\varphi)| \leq
 \epsilon$ for all $t \geq t_1$, and all $n$.   Thus $y^n_t \to \eta$ weak* with $t$.
By Lemma \ref{prox},
since $d(1,A) = 1$, we can choose  $w^n_t \in A$ with $\Vert w^n_t - y^n_t \Vert < \frac{3}{n}$ and $\Vert 1 - w^n_t
\Vert  = 1$.  Clearly $w^n_t \to \eta$ weak*.

That ${\mathfrak r}_A$ is weak* dense in  ${\mathfrak r}_{A^{**}}$ follows from this, and the idea in Proposition \ref{whau}.    We omit the details, since this also follows from Propositions \ref{pgold} and \ref{preq}.

Next, let $e$ be the identity of
$A^{**}$.   By Lemma \ref{hfin} we have that $e \in \frac{1}{2} {\mathfrak F}_{A^{**}}$.  Suppose that $(z_t)$ is a net in $\frac{1}{2} {\mathfrak F}_A$ with weak* limit $e$ in $A^{**}$.  
Standard arguments (see e.g.\ \cite[Proposition 2.9.16]{Dal})  
show that convex combinations $w_t$  of the $z_t$ have the property that $aw_t$ and $w_t a$ converge weakly 
to $a$ for all $a \in A$.  The usual argument (see e.g.\ the proof of \cite[Theorem 6.1]{BHN}) shows that further convex combinations are a  cai in $\frac{1}{2} {\mathfrak F}_A$.   
\end{proof}

{\bf Remark.}   For the first statements of the Theorem we do not need 
the full strength of the `$M$-approximately unital' condition, just strong proximinality at $1$.
For the existence of a cai in $\frac{1}{2} {\mathfrak F}_{A}$ the argument only 
uses strong proximinality at $1$ and  $\Vert 1 - 2 e \Vert \leq 1$.   Similarly,
the existence of a  bai in ${\mathfrak F}_{A}$ will follow from
strong proximinality at $1$ and  $\Vert 1 - e \Vert \leq 1$.

\bigskip

Applied to operator algebras, the latter gives  short proofs of a recent theorem of Read \cite{Read} (see also \cite{Bnpi}), as well as \cite[Lemma 8.1]{BRI} and \cite[Theorem 3.3]{BRII}.    (We remark though that the proof of Read's theorem in \cite{Bnpi} does contain useful extra information that does
not seem to follow from the methods of the present paper, as is pointed out in e.g. Remark 2 after Theorem 2.1 in \cite{Bord}.)
Several other results from \cite{BRI} now follow from the
last result, and with otherwise unchanged proofs,  for 
$M$-approximately unital Banach algebras.  For example: 

\begin{corollary} \label{rd5} {\rm (Cf.\ \cite[Corollary 1.5]{BRI},
\cite[Theorem 2.8]{SW2})}
 \  If $J$ is a closed two-sided ideal in a unital Arens regular 
Banach algebra $A$, and if $J$ is $M$-approximately unital, 
and if the  support projection of $J$ in $A^{**}$ is central there,
then $J$ has a cai $(e_t)$ with $\Vert 1 - 2 e_t \Vert \leq
1$ for all $t$, which is also quasicentral (that is,
$e_t a - a e_t \to 0$ for all $a \in A$).
\end{corollary}

\begin{corollary} \label{r6}  {\rm (Cf.\ \cite[Corollary 1.6]{BRI})}
 \ Let $A$ be an $M$-approximately unital Banach algebra.  Then $A$ has  a
 countable bai $(f_n)$ iff $A$ has a countable cai in $\frac{1}{2} {\mathfrak F}_A$.
This is also equivalent (by Theorem {\rm \ref{otter}}) to $A = \overline{xAx}$ for some $x \in {\mathfrak F}_A$.
\end{corollary}

{\bf Remark.}
We can also use the results in this section  to develop a slightly different approach to 
 hereditary subalgebras than the one taken in Section 4.   For example, the following is a generalization of 
the phenomenon in the first example in \cite[Section 2]{BHN}, which can be interpreted as
saying that for any contractive projection $p$ in the multiplier
algebra $M(A)$, $pAp$ is a HSA in the sense of that paper.     Suppose that
$A$ is an $M$-approximately unital Banach algebra,
and that $p$ is an
idempotent  in $M(A)$ with  
$\Vert 1-2p \Vert \leq 1$. 
For simplicity 
suppose that $A$ is Arens regular.  Define $D = pAp$.
Note that $D$ is an inner ideal in $A$.
We claim that $D$ has a bai in 
$\frac{1}{2} {\mathfrak F}_D$.  
To see this, note that by the 
usual arguments $D^{\perp \perp} = pA^{**} p$.  By  
Theorem  \ref{lbai}
there is a net $w_\lambda$  in
$\frac{1}{2} {\mathfrak F}_A$
with $w_\lambda \to p$ weak*.
Set $d_\lambda = p w_\lambda p$,
then $d_\lambda
\in \frac{1}{2} {\mathfrak F}_D$, and $d_\lambda  \to p$ weak*. 
By the usual arguments, convex combinations
of the $d_\lambda$ give a cai for $D$ 
 in $\frac{1}{2} {\mathfrak F}_D$.  It is easy to see that 
$\overline{DA}= pA$ and $\overline{AD} = Ap$ are the induced one-sided ideals,
and $(d_\lambda)$ is a one-sided cai for these.

\section{Banach algebras and order theory}

As we said earlier, ${\mathfrak r}_A$ and  ${\mathfrak r}^{\mathfrak e}_A$
are closed cones in $A$, but are not
proper in general (hence are what are sometimes called {\em wedges}).    By the argument at the start of Section 2 in 
\cite{Bord},  ${\mathfrak c}_A = \Rdb^+ {\mathfrak F}_A$  is a proper cone.  
These cones naturally induce orderings: we write $a \preceq b$ (resp.\ $a \preceq_{{\mathfrak e}} b$) if $b-a \in 
{\mathfrak r}_A$ (resp.\ $b-a \in  {\mathfrak r}^{{\mathfrak e}}_A$).
These 
are pre-orderings, but are  not in general antisymmetric.    
Because of this some aspects of the classical theory of ordered linear spaces
will not generalize.   Certainly many books on ordered linear spaces assume that their cones are proper.
However other books (such as \cite{AE} or \cite{Jameson}) do not make this assumption in large segments
of the text, and it turns out that the ensuing theory  interacts in
a remarkable way with our recent notion of positivity, as we  point out in this section and in \cite{Bord,BRII}.   For example, in the ordered space theory,
the cone  ${\mathfrak d} = \{ x \in X : x \geq 0 \}$  in an ordered space $X$ is said to be {\em
 generating} if $X = {\mathfrak d} - {\mathfrak d}$.  This is sometimes
called {\em positively generating} or {\em directed} or {\em co-normal}.   
If it is not generating one often looks at the subspace ${\mathfrak d} - {\mathfrak d}$.  
In this language, we shall see next  that ${\mathfrak r}_A$
and ${\mathfrak c}_A = \Rdb^+ {\mathfrak F}_A$  are  generating  cones if $A$ is
$M$-approximately unital, or has a sequential cai and satisfies some further conditions
of the type met in Section 2.   We first discuss the order theory of $M$-approximately unital algebras.

\begin{theorem}  \label{lbdiff}  Let $A$ be 
an  $M$-approximately unital  
Banach algebra.  Any 
$x \in A$ with $\Vert x \Vert < 1$ may be written as $x = a-b$ with $a, b \in {\mathfrak r}_A$
and $\Vert a \Vert < 1$ and $\Vert b \Vert < 1$.
In fact 
one may choose such $a, b$ to also be in $\frac{1}{2} {\mathfrak F}_{A}$.   
\end{theorem}

\begin{proof}   Assume that $\Vert x \Vert = 1$.   Since ${\mathfrak F}_{A^{**}} = e + {\rm Ball}(A^{**})$
by Lemma \ref{hfin}, 
$x = \eta - \xi$ for $\eta, \xi \in \frac{1}{2} {\mathfrak F}_{A^{**}}$.    We may assume that
$A$ is nonunital (the unital case follows from the last line with $A^{**}$ replaced by $A$).
By \cite[Lemma 8.1]{BRI}
we deduce that $x$ is in the weak closure of the convex set
$\frac{1}{2} {\mathfrak F}_{A} - \frac{1}{2} {\mathfrak F}_{A}$.
Therefore it is in the norm closure, so given $\epsilon > 0$ there exists 
$a_0, b_0 \in \frac{1}{2} {\mathfrak F}_{A}$ with 
$\Vert x - (a_0 - b_0) \Vert < \frac{\epsilon}{2}$. 
Similarly, there exists
$a_1, b_1 \in \frac{1}{2} {\mathfrak F}_{A}$ with
 $\Vert x - (a_0 - b_0) - \frac{\epsilon}{2} (a_1 - b_1) \Vert < \frac{\epsilon}{2^2}$.
Continuing in this manner, one produces sequences $(a_k), (b_k)$ in $\frac{1}{2} {\mathfrak F}_{A}$. 
Setting $a' = \sum_{k=1}^\infty \, \frac{1}{2^k} \, a_k$ and 
$b' = \sum_{k=1}^\infty \, \frac{1}{2^k} \, b_k$, which are in $\frac{1}{2} {\mathfrak F}_{A}$ since the 
latter is a closed convex set,
we have $x = (a_0 - b_0) + \epsilon (a' - b')$.  Let $a = a_0 + \epsilon a'$ and 
$b = b_0 + \epsilon b'$.    By convexity $\frac{1}{1 + \epsilon} a \in \frac{1}{2} {\mathfrak F}_{A}$
and $\frac{1}{1 + \epsilon} b \in \frac{1}{2} {\mathfrak F}_{A}$.

If $\Vert x \Vert < 1$ choose  $\epsilon > 0$  with $\Vert x \Vert (1 + \epsilon) < 1$.
Then $x/\Vert x \Vert = a - b$ as above, so that 
$x = \Vert x \Vert \, a - \Vert x \Vert \, b$.    We have $$\Vert x \Vert \, a
= (\Vert x \Vert (1 + \epsilon)) \cdot (\frac{1}{1 + \epsilon} a ) \in [0,1) \cdot \frac{1}{2} {\mathfrak F}_{A}
\subset \frac{1}{2} {\mathfrak F}_{A} ,$$
and similarly $\Vert x \Vert \, b \in  \frac{1}{2} {\mathfrak F}_{A}$.  \end{proof}

{\bf Remarks.}   1) \ If $A$ is $M$-approximately unital then can every
$x \in {\rm Ball}(A)$   be written as $x = a-b$ with $a,b
 \in {\mathfrak r}_A \cap {\rm Ball}(A)$?
  As we said above, this is true if $A$ is unital.  We are particularly interested in this question when
$A$ is an operator algebra (or uniform algebra).  
We can show that in general $x \in {\rm Ball}(A)$ cannot be written  as $x = a-b$ with
$a, b \in \frac{1}{2} {\mathfrak F}_{A}$.   To see this let $A$ be the set
of functions in the  disk algebra vanishing at $-1$, an approximately unital function
algebra.   Let $W$ be the  closed connected set obtained from the unit disk by removing the
`slice' consisting of all complex numbers with negative real part and argument
in a small open interval containing $\pi$.  By the Riemann mapping theorem
 it is easy to see that there is a conformal map $h$ of the disk onto
$W$ taking $-1$ to $0$, so that $h \in {\rm Ball}(A)$.   By way of contradiction
suppose that
$h = a-b$ with $a,b \in \frac{1}{2} {\mathfrak F}_{A}$.  
We use
the geometry of circles
in the plane: if $z, w \in \overline{B(\frac{1}{2}, \frac{1}{2})}$ 
with $|z-w| = 1$ then $z + w = 1$.
It follows that $a + b = 1$
on a nontrivial arc of the unit circle, and hence everywhere
(by
e.g.\ \cite[p.\ 52]{Hof}).  However $a(-1) + b(-1) = 0$, which is the desired contradiction.

\smallskip

2) \ Applying Theorem \ref{lbdiff} to $ix$ for $x  \in A$, one gets a similar decomposition
$x = a-b$ with the `imaginary parts' of $a$ and $b$ positive.   One might ask if, as is suggested by
the $C^*$-algebra case, one may write for each $\epsilon$, any $x \in A$ with $\Vert x \Vert < 1$
as $a_1 - a_2 + i(a_3 - a_4)$ for $a_k$ with numerical range in a thin horizontal `cigar' of height $<
\epsilon$ centered on the line segment $[0,1]$ in the $x$-axis.  In fact this is false, as one can see
in the case that $A$ is  the set of upper triangular $2 \times 2$
matrices with constant diagonal entries.

\bigskip

 A bounded $\Rdb$-linear $\varphi : A \to \Rdb$ (resp.\ $\Cdb$-linear $\varphi : A \to \Cdb$) is called real positive if $\varphi({\mathfrak r}_A) \subset [0,\infty)$ (resp.\ Re$\, \varphi({\mathfrak r}_A) \geq 0$).  
The set  of real positive functionals on $A$ is the  {\em real dual cone}, and we write it as
${\mathfrak c}^{\Rdb}_{A^*} $.  Similarly, the `real version' of ${\mathfrak c}^{\mathfrak e}_{A^*}$ will be written
as ${\mathfrak c}^{{\mathfrak e},\Rdb}_{A^*}$. 
 By the usual trick, for any $\Rdb$-linear $\varphi : A \to \Rdb$, there
is a unique $\Cdb$-linear $\tilde{\varphi} : A \to \Cdb$  with Re $\, \tilde{\varphi} = \varphi$,
and clearly $\varphi$ is real positive 
iff $\tilde{\varphi}$ is real positive.

\begin{proposition} \label{preq}  Let $A$ be an 
M-approximately unital  
Banach algebra.
 An $\Rdb$-linear  $f : A \to \Rdb$ (resp.\ $\Cdb$-linear $f : A \to \Cdb$)
is real positive iff $f$ is a nonnegative multiple of the real part of a state (resp.\  nonnegative multiple
of a state).
Thus $M$-approximately unital algebras are scaled 
 Banach algebras.  
\end{proposition} \begin{proof}    The one direction is obvious.  For the 
other, by the observation above
the Proposition, we can assume that  $f : A \to \Cdb$ is $\Cdb$-linear
and real positive.    If $A$ is unital then the result follows from the proof of \cite[Theorem 2.2]{Mag}.
Otherwise by 
Proposition \ref{auf2} (4)  applied to the inclusion $A \subset A^1$ we  see that the condition
in Corollary \ref{snr} (iii) holds.  
So $A$ is scaled by Corollary \ref{snr}.
(We remark that we had a different proof in an earlier draft.)  
   \end{proof}

We now turn  to other
 classes of algebras (although we will obtain another couple of results for $M$-approximately unital algebras
 later in this section in 
parts (2) of Corollaries \ref{dirset} and \ref{preqi}).

The following is a variant and simplification of 
\cite[Lemma 2.7 and 
Corollary 2.9]{BRIII} and \cite[Corollary 3.6]{BRII}.   

\begin{proposition}  \label{pr3}   Let $A$ be an 
scaled approximately unital  
 Banach algebra.  Then the real dual cone ${\mathfrak c}^{\Rdb}_{A^*} $  equals 
$\{ t \, {\rm Re}(\psi) : \psi \in S(A) , \, t \in [0,\infty) \}$.   The prepolar of ${\mathfrak c}^{\Rdb}_{A^*}$, which equals its
real predual cone,  is  ${\mathfrak r}_{A}$;  
and  the polar of ${\mathfrak c}^{\Rdb}_{A^*}$,  which equals its
real dual cone,  is  ${\mathfrak r}_{A^{**}}$.     
\end{proposition} \begin{proof}   It follows as in Proposition \ref{preq}   that  
  $${\mathfrak c}^{\Rdb}_{A^*} = \{ t \, {\rm Re}(\psi) : \psi \in S(A) , \, t \in [0,\infty) \}.$$
The prepolar of ${\mathfrak c}^{\Rdb}_{A^*}$, which equals its
real predual cone,  is  ${\mathfrak r}_{A}$ by the bipolar theorem.  
We proved in Proposition \ref{pgold}  that ${\mathfrak r}_{A}$ is weak* dense 
in ${\mathfrak r}_{A^{**}}$. 
This  together with  the bipolar theorem gives the last assertion.  \end{proof}

The following is a  `Kaplansky density' result for ${\mathfrak r}_{A^{**}}$:   

\begin{proposition}  \label{Kap} Let $A$ be an approximately unital Banach algebra
such that ${\mathfrak r}_A$ is weak* dense in ${\mathfrak r}_{A^{**}}$
(as we saw in Proposition {\rm \ref{pgold}} was the case for scaled approximately unital  
algebras).  Then the set of contractions
in ${\mathfrak r}_A$ is weak* dense in the set of contractions
in ${\mathfrak r}_{A^{**}}$.    If in addition  there exists a mixed identity of norm 1
in ${\mathfrak r}_{A^{**}}$, then $A$ has a cai in ${\mathfrak r}_{A}$.
\end{proposition}
 \begin{proof}  We use a standard kind of bipolar argument from the theory of ordered spaces.   
If $E$ and $F$ are closed sets in a TVS with $E$ compact, then $E + F$ is closed.   By this principle, 
and by Alaoglu's theorem, 
${\rm Ball}(A^*) + {\mathfrak c}_{A^*}$ is weak* closed.  Its prepolar (resp.\ polar) 
certainly is contained in ${\rm Ball}(A) \cap {\mathfrak r}_A$ (resp.\  ${\rm Ball}(A^{**}) \cap {\mathfrak r}_{A^{**}}$).  This uses the fact that $$({\mathfrak c}_{A^*})^\circ = {\mathfrak r}_A^{\circ \circ} = \overline{{\mathfrak r}_{A}}^{w*}
= {\mathfrak r}_{A^{**}}$$ 
 by the bipolar theorem.     However if $a \in {\rm Ball}(A) \cap {\mathfrak r}_A$ and 
$f \in {\rm Ball}(A^*)$ and $g \in {\mathfrak c}_{A^*}$   then Re$(f(a) + g(a)) \geq -1 + 0 = -1$.
So the prepolar  of ${\rm Ball}(A^*) + {\mathfrak c}_{A^*}$  
is  ${\rm Ball}(A) \cap {\mathfrak r}_A$,
and similarly its polar is ${\rm Ball}(A^{**}) \cap {\mathfrak r}_{A^{**}}$.
Thus ${\rm Ball}(A) \cap {\mathfrak r}_A$ is weak* dense in 
${\rm Ball}(A^{**}) \cap {\mathfrak r}_{A^{**}}$ by the  bipolar  theorem.
The last assertion clearly follows from this and Lemma \ref{netwee}.
 \end{proof}

   The condition in the next result
that $A^{**}$ is unital is a bit restrictive (it holds for example if $A$ is Arens regular and approximately 
unital), but the result illustrates some of what one might like to be true 
in more general situations:

\begin{theorem} \label{soaddf}   Let $A$ be a
Banach algebra such that $A^{**}$ is unital, and suppose that
 ${\mathfrak e}$ is a cai for $A$.  Then ${\mathfrak r}^{\mathfrak e}_{A} \subset {\mathfrak r}_{A^{**}}$ iff ${\mathfrak r}^{\mathfrak e}_{A}  = {\mathfrak r}_{A}$.
Suppose that the latter is true, and that 
$Q_{\mathfrak e}(A)$ is weak* closed.   Then  $A$ is scaled, 
 $S(A) = S_{\mathfrak e}(A)$, and $A$ has a  cai in ${\mathfrak r}_A$.   Also in this case,  
 $A =  {\mathfrak r}_A - {\mathfrak r}_A$.  Indeed any $x \in A$ with $\Vert x \Vert < 1$ may be written as $x = a-b$
for $a, b \in  {\mathfrak r}_A \cap {\rm Ball}(A)$.  
\end{theorem}

\begin{proof}    If $f \in S(A)$ then by viewing $A^1 = A + \Cdb e$ we may extend $f$ to a state $\hat{f}$ of $A^{**}$.  If $x \in {\mathfrak r}^{\mathfrak e}_{A} \subset {\mathfrak r}_{A^{**}}$ then Re$ \, f(x) = {\rm Re} \, \hat{f} (x) \geq 0$.  Thus ${\mathfrak r}^{\mathfrak e}_{A} \subset {\mathfrak r}_{A}$, and so
these sets are equal.  We also see that ${\mathfrak c}_{A^*} = {\mathfrak c}^{\mathfrak e}_{A^*}$.    If $Q_{\mathfrak e}(A)$ is weak* closed then  $A$ is ${\mathfrak e}$-scaled by
Lemma \ref{escaled}, so that $f = tg$ for some $g \in S_{\mathfrak e}(A)$ and for 
some $t$ which must equal 1.   It follows that $S(A) = S_{\mathfrak e}(A)$.
Hence $A$ is scaled, so that the weak* closure of ${\mathfrak r}_A \cap {\rm Ball}(A)$ is ${\mathfrak r}_{A^{**}} \cap {\rm Ball}(A^{**})$ by Proposition  \ref{Kap}.   Since the latter contains an identity,  
$A$ has a cai in ${\mathfrak r}_{A}$ by the observation after that result.   The assertion concerning  $\Vert x \Vert < 1$ 
follows by a slight variant of the 
proof of Theorem \ref{lbdiff}.  
 \end{proof}  

In
fact it is not too hard to see, as we shall show in another paper, that
if $A^{**}$ is unital (or if it
has a unique mixed identity), and $A$ has a cai in ${\mathfrak r}_A$
then $A$ has a cai in ${\mathfrak F}_A$
(and the latter cai can be chosen to be sequential if the first cai
is sequential).

We now attempt to prove parts of the last theorem, and some other order theoretic results,
 in the case that $A^{**}$ is not unital.  We will  mostly 
be using  the class of states $S_{\mathfrak e}(A)$ with respect to a fixed cai ${\mathfrak e}$, and the matching
cones ${\mathfrak r}^{\mathfrak e}_{A}$ and ${\mathfrak c}^{\mathfrak e}_{A^*}$, as opposed to $S(A)$ and its
matching cones.   The reason for this is that we will want norm additivity $$\Vert c_1 \varphi_1 + \cdots + c_n \varphi_n \Vert = c_1 +  \cdots + c_n
, \qquad \varphi_k \in S(A), c_k \geq 0.$$  In many interesting examples $S(A)$ 
will
 satisfies this additivity property  (for example if $A$ is Hahn-Banach smooth, by
Lemma \ref{hbs}), and in this case almost 
all the rest of the results in this section will be true for the $S(A)$ variants, and with all the subscript and superscript  and hyphenated
${\mathfrak e}$'s dropped.    

\begin{lemma}  \label{dualc2}   Suppose that ${\mathfrak e} = (e_t)$ is a fixed cai for a Banach algebra $A$, and suppose that
$Q_{{\mathfrak e}}(A)$ is weak* closed in $A^*$.
\begin{itemize} \item [(1)] 
  The cones  
${\mathfrak c}^{\mathfrak e}_{A^*}$ and  ${\mathfrak c}^{{\mathfrak e},\Rdb}_{A^*}$ are
 {\em additive} (that is, the norm on the dual space of $A$ 
is additive on these cones).   
\item [(2)] 
 If $(\varphi_t)$ is an increasing net in 
${\mathfrak c}^{{\mathfrak e},\Rdb}_{A^*}$ which is bounded in norm,
then the net converges in norm, and its limit is the least upper 
bound of the net.  \end{itemize} 
\end{lemma}  \begin{proof}   (1) \ 
If $\psi = c \varphi$ for  $\varphi \in S_{\mathfrak e}(A)$
and $c \geq 0$,  
then  
$$\Vert \psi \Vert = c \Vert \varphi \Vert =
\lim_t \, \psi(e_t) .$$
Indeed for an appropriate mixed identity $e$ of $A^{**}$  of norm $1$ we have
  $\Vert \varphi \Vert = \langle e , \varphi \rangle$ for all
$\varphi  \in {\mathfrak c}^{{\mathfrak e},\Rdb}_{A^*}$.  It follows
that the norm on $B(A,\Rdb)$ is additive on ${\mathfrak c}^{{\mathfrak e},\Rdb}_{A^*}$.   The complex scalar case is similar. 

(2) \ Follows from (1) and \cite[Proposition 3.2,  Chapter 2]{AE}.   
\end{proof}  

We recall that the positive part of the  open unit ball of a $C^*$-algebra
is a directed set.  The following is a  Banach
 algebra version of this:

\begin{corollary} \label{dirset}  \begin{itemize} \item [(1)]  Let ${\mathfrak e}$ be a cai for a Banach algebra $A$, and suppose that
$Q_{{\mathfrak e}}(A)$ is weak* closed in $A^*$.
Then the open unit ball of $A$
is a directed set with respect to the 
$\preceq_{\mathfrak e}$ ordering.    That is,
if $x, y \in A$ with $\Vert x \Vert , \Vert y  \Vert < 1$,
then there
exists $z \in A$ with $\Vert z\Vert < 1$ and $z \in {\mathfrak r}^{\mathfrak e}_A$ ,
and also
$x \preceq_{\mathfrak e} z$ and $y  \preceq_{\mathfrak e} z$.   
 \item [(2)]    If $A$ is an $M$-approximately unital Banach algebra,   then given $x, y
 \in A$ with $\Vert x \Vert , \Vert y  \Vert < 1$, a majorant 
$z$ can be chosen as in {\rm (1)}, but also with
 $z \in \frac{1}{2} {\mathfrak F}_A$.  \end{itemize}  
 \end{corollary}

\begin{proof}  (1) \ By Lemma \ref{dualc2} (1),  for any $x, y \in A$ with $\Vert x \Vert  <1$ and $\Vert y \Vert  <1$, 
there exists a $w \in A$ with $\Vert w \Vert < 1$ and  $w-x, w-y \in {\mathfrak r}^{\mathfrak e}_A$.   In the 
`countable case', by 
the last assertion of
Theorem  \ref{con} (setting the $a$ there to be $-t w$ for some appropriate $t > 1$), we have $w \preceq_{\mathfrak e} z$ for some $z \in {\mathfrak r}^{\mathfrak e}_A$ with $\Vert z \Vert < 1$.
So $$-z \preceq_{\mathfrak e} - w  \, \preceq_{\mathfrak e} x \, \preceq_{\mathfrak e}  w \, \preceq_{\mathfrak e} z.$$ 
Similarly, $y$ `lies between' $z$ and $-z$.   In the general case the easy trick is given in 
\cite{B2015}.

(2) \ This is similar to (1), but uses the fact that $S(A) = S_{\mathfrak e}(A)$ by Lemma \ref{hbs}, so all ${\mathfrak e}$'s can be dropped.
We also use the following principle twice in place of 
the cited results in the proof above:  if
 $\Vert z \Vert < 1$ then by Corollary  \ref{lbdiff} we may write $z = a- b$ for $a,b \in \frac{1}{2}{\mathfrak F}_A$, and
 then $-b \preceq z \preceq a$.
\end{proof}

For a $C^*$-algebra $B$, a natural ordering on the positive
part of the open unit ball of $B$
turns the latter into a net which is a positive cai for $B$ (see e.g.\ \cite{Ped}).  
A similar result holds for operator algebras \cite[Proposition 2.6]{Bord}.  We are not sure if there is an analogue of this
for the classes of algebras in the last result.

\begin{corollary} \label{preqi}   \begin{itemize} \item [(1)]   Let ${\mathfrak e}$ be a cai for a Banach algebra $A$, and suppose that
$Q_{{\mathfrak e}}(A)$ is weak* closed in $A^*$. 
For all $x  \in A$ there exists an element $z \in {\mathfrak r}^{\mathfrak e}_A$ with $-z \preceq_{\mathfrak e}  x \, \preceq_{\mathfrak e} z$.
 Thus $x = a - b$ where 
$a,b \in {\mathfrak r}^{\mathfrak e}_A$.  Moreover if $\Vert x \Vert < 1$
then $z, a, b$ can all be chosen in ${\rm Ball}(A)$.    
 \item [(2)]   If $A$ is an $M$-approximately unital Banach algebra,   then given $x \in A$  
with  $\Vert x \Vert < 1$,  an element  
$z$ can be chosen satisfying the inequalities in {\rm (1)}, but also with 
$z \in \frac{1}{2} {\mathfrak F}_A$.
\end{itemize}
 \end{corollary} \begin{proof}    Apply
Corollary \ref{dirset} to $x$ and $-x$.   Of course $a = \frac{z+x}{2}$ and $b =  \frac{z-x}{2}$.
\end{proof}

In the language of \cite{Mes}, item (1) implies that the associated preorder on $A$ there is {\em  approximately} $1$-{\em absolutely conormal}, and from the theory of ordered Banach spaces in that reference this is equivalent to $B(A,\Rdb)$ being `absolutely monotone'.   
That is, with respect to the natural induced ordering on
$B(A,\Rdb)$, if $-\psi \leq \varphi \leq \psi$ then $\Vert \varphi \Vert \leq \Vert \psi \Vert$.   

\begin{corollary}  \label{dualc2a}  Let ${\mathfrak e}$ be a cai for a Banach algebra $A$, and suppose that
$Q_{{\mathfrak e}}(A)$ is weak* closed in $A^*$.  If $f \leq g \leq h$   in $B(A,\Rdb)$ in 
the natural 
${\mathfrak c}^{\mathfrak e}_{A^*}$-ordering,
then $\Vert g \Vert \leq \Vert f \Vert + \Vert h \Vert$.  \end{corollary} 

\begin{proof}  This follows 
from   Corollary \ref{preqi} by \cite[Theorem 1.1.4]{BR}.  \end{proof} 

\begin{corollary} \label{soadd}  If $A$ is an approximately unital Banach algebra then the last four results are true with all
the subscript and superscript  and hyphenated ${\mathfrak e}$'s dropped, if  also 
$S(A) = S_{\mathfrak e}(A)$ for the cai ${\mathfrak e}$ appearing in those results (which holds
for example if $A$ is Hahn-Banach smooth in $A^1$).
\end{corollary}  

\begin{proof}  
Indeed in the Hahn-Banach smooth case $S(A) = S_{\mathfrak e}(A)$  by
Lemma \ref{hbs}, and if the latter holds then all ${\mathfrak e}$'s may be dropped. 
\end{proof}  

In the part of Corollary  \ref{soadd} dealing with Corollary {\rm  \ref{dirset} (2)},
and in  Corollary {\rm  \ref{preqi} (2)} in the $\Vert x \Vert < 1$ case, one may often get the majorants
 $z$ appearing in those Corollaries to also  be in ${\mathfrak F}_A$  (and even get a sequential cai for A in ${\mathfrak F}_A$ consisting
of such majorants $z$).  We will discuss this in another paper, but
briefly this follows from the ideas in Corollary \ref{scase} and the paragraphs
after that, and the idea in the paragraph after Theorem \ref{soaddf}. 

\bigskip

{\bf Remarks.}  1) \   Above we saw that under various hypotheses, 
a Banach algebra $A$ had a cai in ${\mathfrak r}_A$, and the latter was a generating cone, that is $A = {\mathfrak r}_A - {\mathfrak r}_A$.
Conversely we shall
see in Corollary \ref{gennew} that if $A$ is commutative, approximately unital, and $A = {\mathfrak r}_A - {\mathfrak r}_A$, 
then $A$ has a bai in ${\mathfrak F}_A$.  

\smallskip
 
2) \ It is probably never  true for 
an approximately unital operator algebra $A$ 
that 
$B(A,\Rdb) = {\mathfrak c}^{\Rdb}_{A^*} - {\mathfrak c}^{\Rdb}_{A^*}$.
Indeed, in the case $A = \Cdb$ the latter space has  real dimension $1$.
However the complex span of the (usual) states of an
 approximately unital operator algebra $A$ is $A^*$ (the complex dual space).  
Indeed by a result of Moore \cite{Moore,AE2}, the 
complex span of the states of any unital Banach algebra $A$ is $A^*$.
In the approximately unital Banach algebra case, 
at least if $A$ is scaled the same fact follows by using a Hahn-Banach extension and
Corollary \ref{snr} (iii).  

\smallskip

3) \  Every element $x \in \frac{1}{2} 
{\mathfrak F}_A$ need not achieve its norm at a state, even in $M_2$ (consider 
$x = (I + E_{12})/2$ for example).

\smallskip

4) \     
We thank Miek Messerschmidt for calling our attention to the result in \cite{BR} used in  Corollary
\ref{dualc2a}.  Previously we 
had a cruder inequality in that result.

\smallskip

5) \ Note that $A$ is not `order-cofinal'  in $A^1$ usually, in the sense of the ordered space literature,
even for $A$ any $C^*$-algebra with no countable cai (and hence no  strictly real positive element).

\section{Ideals in commutative Banach algebras}

Throughout this section $A$ will
be a commutative approximately unital Banach algebra.    
  We will use ideas from \cite{BHN,BRI, BRII} (see \cite{Est,KLU} for some other Banach algebra variants of some of these ideas).    In the following statement, the `respectively's are placed correctly, despite first impressions.

\begin{theorem} \label{nowha}  Let $A$ be a commutative approximately unital Banach algebra.  The   closed  ideals
in $A$ with a bai in ${\mathfrak r}_A$ (resp.\ ${\mathfrak F}_A$)  are precisely the ideals of the form $\overline{EA}$ for 
some subset  $E \subset {\mathfrak F}_A$ (resp.\ $E \subset {\mathfrak r}_A$).    
They are also the closures of increasing unions of  ideals of the form $\overline{xA}$ for 
$x \in {\mathfrak F}_A$ (resp.\ $x \in {\mathfrak r}_A$).      \end{theorem} 

\begin{proof} 
 Suppose that $E \subset {\mathfrak r}_A$, and we will prove that $\overline{EA}$ has a bai in ${\mathfrak F}_A$.
We may assume that $E \subset {\mathfrak F}_A$ since $\overline{EA} = \overline{{\mathfrak F}(E) \, A}$
as may be seen using  Proposition \ref{whba}.  
We will first suppose that $E$ has two elements, and here we will 
include a separate  argument if $A$ is Arens regular since the computations are interesting.   Then we 
will discuss the case  where $E$ has
 $n$ elements, and then the general case.

If $x, y \in {\mathfrak r}_A$
then $\overline{xA}$ and $\overline{yA}$ are ideals with bai in ${\mathfrak F}_A$ by Corollary  \ref{lba}.  Their support idempotents $s(x)$ and $s(y)$ are in ${\mathfrak F}_{A^{**}}$.    
Indeed if $J = \overline{xA}$ then by Corollary  \ref{lba}  we have $J^{\perp \perp} = s(x) A^{**}$, and $J = s(x) A^{**} \cap A$.   (In the non-Arens regular case we are using the   
`second Arens product' here.)  
In the rest of this paragraph we assume that $A$ is Arens regular.
 Set $$s(x,y) = s(x) + s(y) - s(x) s(y) = 1 - (1- s(x)) (1 - s(y)),$$
an idempotent dominating both $s(x)$ and $s(y)$ in the sense that $s(x,y) s(x)= s(x)$ and $s(x,y) s(y)= s(y)$.  
If $f$ is another idempotent dominating both $s(x)$ and $s(y)$ then 
$f s(x,y) = s(x,y)$, so that  $s(x,y)$ is the `supremum' of $s(x)$ and $s(y)$  in this ordering.
Then notice that 
$\Vert (1-x^{\frac{1}{n}}) (1-y^{\frac{1}{m}}) \Vert  \leq 1$, and 
also $$\Vert (1 - s(x)) (1 - s(y)) \Vert = \Vert 1 -  s(x,y) \Vert \leq 1.$$   Notice too
that $\overline{xA + yA}$
has a  bai in ${\mathfrak F}_A$ with terms of form
$$x^{\frac{1}{n}} + y^{\frac{1}{m}} - x^{\frac{1}{n}} y^{\frac{1}{m}}
= 1 - (1 - x^{\frac{1}{n}})(1 - y^{\frac{1}{m}})$$
which have bound $2$.  A double weak* limit point of this bai from ${\mathfrak F}_A \cap \overline{EA}$ is $s(x,y)$.   So as
usual $\overline{xA + yA} = \{ a \in A : s(x,y) a = a \}$.    

In the non-Arens regular case we use the
`second Arens product' below.   We 
show that $\overline{xA + yA} = \overline{(\frac{x+ y}{2}) A} = \overline{aA}$ where $a = \frac{x+ y}{2} \in {\mathfrak F}_A$.    By the proof of \cite[Lemma 2.1]{BRI} we know that 
$(1 - \frac{1}{n}\sum_{k=1}^n \, (1-a)^k)  \in {\mathfrak F}_A$ is a bai for ba$(a)$, and for $\overline{aA}$.  
Write $x = 1-z, y = 1-w$ for contractions $z,w \in A^1$, and let $b  = \frac{z + w}{2}$.
Then $a = 1 - b$. Let $r$ be a weak* limit  point
of the bai above, which is a mixed identity for ba$(a)^{**}$.   Then $r a = a$, so that 
$(1-r) b = (1-r)$.  Note that $s = 1-r$ is a contractive idempotent, and is an identity for 
$s (A^1)^{**} s$.  Since the identity in a Banach algebra is an extreme point, and since
$\frac{sz + sw}{2} = s$ we deduce that $sz = zs = s$.  
 Similarly $sw = ws = s$.
Thus $r x = x$, so that $x \in r A^{**} \cap A = \overline{a A}$ (as in Corollary \ref{lba}).   Similarly for $y$, and thus $\overline{xA + yA} = \overline{(\frac{x+ y}{2}) A}$.   
Thus if $x, y  \in {\mathfrak F}_A$ then the support idempotent
$s(\frac{x + y}{2})$ for $a$ can be taken to be
a `support idempotent' for $\overline{xA + yA}$.    

A very similar argument works for three elements $x, y, z  \in {\mathfrak F}_A$, using for 
example  
the fact that $\Vert (1-x^{\frac{1}{n}}) (1-y^{\frac{1}{n}}) (1-z^{\frac{1}{n}}) \Vert  \leq 1$.
Indeed a similar argument works for any finite collection  $G = \{ x_1, \cdots , x_m \} \in {\mathfrak F}_A$.  We have $\overline{GA} = \overline{x_G A}$, where  $$x_G
= \frac{1}{m} \, (x_1 +  \cdots + x_m)
\in {\mathfrak F}_A \cap \overline{EA}.$$
Let us write $s(G)$ for $s(\frac{1}{m} \, (x_1 +  \cdots + x_m))$,
then $s(G)$ is the support idempotent of $\overline{GA}$, and $s(G) A^{**} = 
(G A)^{\perp \perp}$, and thus $\overline{G A} = s(G) A^{**} \cap A$.    This has a
 bai in ${\mathfrak F}_A \cap \overline{EA}$, namely $(1 - [(1-x_1^{\frac{1}{n}}) \cdots
 (1-x_m^{\frac{1}{n}})])$, or $(1 - [(1-x_1^{\frac{1}{n_1}}) \cdots
 (1-x_m^{\frac{1}{n_m}})])$.  

If $E$ is a subset of ${\mathfrak F}_A$,
let $J  = \overline{EA}$, and let $\Lambda$ be the collection of finite subsets $G$  of
$E$ ordered by inclusion.  Writing $\Lambda$ as a net $(G_i)_{i \in \Lambda}$, we have
$$J = \overline{EA} = \overline{\cup_{i \in \Lambda} \, G_i A} = \overline{\cup_{i \in \Lambda} \, x_{G_i} A},$$
where $x_{G_i}
\in {\mathfrak F}_A \cap \overline{EA}$.  To see that
$J$ has a bai in ${\mathfrak F}_A$, as in e.g.\
 \cite[Theorem 5.1.2 (a)]{Pal}  it is enough 
to show that given $G \in \Lambda$ and $\epsilon > 0$
there exists $a \in {\mathfrak F}_A \cap J$ with $\Vert a x - x \Vert < \epsilon$ for 
all $x \in G$.  However this is clear since, as we saw above, 
$\overline{G A}$ has a bai in ${\mathfrak F}_A$.

Conversely, suppose that $J$ is an  ideal
in $A$ with a bai $(x_t)$ in ${\mathfrak r}_A$.
Then $J =   \overline{\sum_t \, x_t A} = \overline{EA}$ where $E = \{ {\mathfrak F}(x_t)  \}
\subset {\mathfrak F}_A$
by Proposition \ref{whba}.    The remaining results are clear from what we have proved.
 \end{proof}

{\bf Remarks.}   1) \ See \cite{LU} for a recent characterization
of ideals with bai.

\smallskip

2) \   We saw in Example \ref{Ex7} that several of the methods used
 in the last proof fail for noncommutative algebras.
First, it is not true there that if $x, y  \in {\mathfrak F}_A$ then 
$\overline{xA + yA} = \overline{(\frac{x + y}{2})A}$.
 Also $\overline{xA + yA}$  may have no  left cai.  
Also, it need not be the case that $EAE$ has a bai if $E \subset {\mathfrak F}_A$. 

\bigskip

If $E$ is any subset of ${\mathfrak F}_A$
and  $J = \overline{EA}$, and if $s = s_E$ is a weak* limit point of any bai 
in ${\mathfrak F}_A$ for $J$, then we call $s$ a {\em support idempotent}
 for $J$.   Note that $s A^{**} = J^{\perp \perp}$ as usual,
and so $J = s A^{**} \cap A$.  

\bigskip 

{\bf Remark.}   Suppose that $I$ is a directed set,
and that  $\{ E_i : i \in I \}$
is a family of subsets of ${\mathfrak F}_A$ with $E_i \subset E_j$ 
if $i \leq j$.
Then $\overline{\sum_i \, E_i A} = \overline{EA}$, where
$E = \cup_i \, E_i$.  Moreover,
if $s_i$ is a support idempotent for $\overline{E_i A}$,
and if $s_i$ has weak* limit point $s'$ in $A^{**}$ then we
claim that $s'$  is a support idempotent for $J = \overline{EA}$.  
Indeed clearly $s' \in (J \cap {\mathfrak F}_{A})^{\perp \perp}$,
since each $s_i$ resides here.  Conversely, 
if $x \in E_i$ then $s_j x = x$
if $j \geq i$, so that $s' x = x$.
Thus $s_i  x \to x$ in norm for all $x \in J$, so that 
$s' x = x$ for all $x \in J$.  Hence $s' x = x$  for all $x \in J^{\perp \perp}$.
Therefore $s'$ is idempotent,
and $J^{\perp \perp} \subset s' A^{**}$, and so
$J^{\perp \perp} = s' A^{**}$.   As usual, $J = s' A^{**} \cap A$.      This 
concludes the proof of the claim.  
If $(x_t)$ is a net in $J \cap {\mathfrak F}_{A}$ with
weak* limit $s'$ then we leave it as an exercise that one can choose a 
net of  convex combinations of the 
$x_t$, which is a bai for $J$ in ${\mathfrak F}_{A}$ with weak* limit $s'$. 
In particular, if $(G_i)_{i \in \Lambda}$ is as in the proof of Theorem \ref{nowha}, then
the net  $s_i = s(G_i)$ has 
a weak* limit point which is a support projection for 
$J = \overline{EA}$.  

\bigskip

Let us define an ${\mathfrak F}$-{\em ideal} to be an ideal of the kind characterized in  Theorem \ref{nowha},
namely a closed  ideal
in $A$ with a bai in ${\mathfrak r}_A$.

\begin{theorem} \label{nowha2}  Let $A$ be a commutative approximately unital Banach algebra.  Any separable ${\mathfrak F}$-ideal in $A$ is of the form $\overline{xA}$ for 
$x \in {\mathfrak F}_A$.    Also, the closure of the sum of a countable set of 
ideals $\overline{x_k \, A}$ for $x_k \in {\mathfrak F}_A$, 
equals $\overline{z A}$ where $z = \sum_{k = 1}^\infty
\, \frac{1}{2^k} \, x_k$.  \end{theorem} 
\begin{proof}   The first assertion follows from the matching result in Section 4 (Corollary \ref{ctrid}),
or from the second assertion as in \cite[Theorem 2.16]{BRI}.
For the second assertion,
let $x_k, z$ be as in the statement.
Inductively one can prove that $x_k \in \overline{zA}$, which is what is needed.
One begins by setting $x  = x_1$ and $y = \sum_{k = 2}^\infty
\, \frac{1}{2^{k-1}} \, x_k \in {\mathfrak F}_A$.  Then $z =  \frac{x + y}{2}$, and the third paragraph of the
 proof of Theorem \ref{nowha}
shows that $x = x_1 \in \overline{zA}$, and $y \in \overline{zA}$.   One then repeats the argument
to show all $x_k  \in \overline{zA}$.  \end{proof}

As in Section 4, we obtain again that for example: 

 \begin{corollary} \label{geM}   Let $A$ be a
commutative $M$-approximately unital Banach algebra.  Then $A$ has a countable cai iff there exists $x \in  {\mathfrak F}_{A}$ with 
$A = \overline{xA}$  (or equivalently, iff $s(x)$ is the unique mixed identity of $A^{**}$  of norm $1$).    
\end{corollary}

With this in hand, one can generalize some part of the theory of left ideals and cai's in \cite{BHN,BRI,BRII} to the class of ideals
in the last theorem, in the commutative case.   This class  is not closed 
under finite intersections.  In fact this fails rather badly (see Example \ref{Ex1}).  
One may define an ${\mathfrak F}$-{\em open 
idempotent} in $A^{**}$ to be an idempotent
$p \in A^{**}$ for which there exists a net $(x_t)$ in ${\mathfrak F}_A$ 
(or equivalently, as we shall see, in  ${\mathfrak r}_A$) with $x_t = p x_t \to p$ weak*.   Thus 
a left identity for the second Arens product in $A^{**}$ 
is ${\mathfrak F}$-open  iff it is in the weak* closure
of ${\mathfrak F}_A$.    See e.g.\ \cite{Ake2, Ped} for the notion of open projection in a $C^*$-algebra.

 \begin{lemma} \label{suppo}  If $A$ is a
commutative approximately unital Banach algebra
then the  ${\mathfrak F}$-open
idempotents in $A^{**}$ are precisely 
the support
idempotents for ${\mathfrak F}$-ideals.
\end{lemma} \begin{proof}   If $p$ is
an ${\mathfrak F}$-open
idempotent then it follows
that $p \in {\mathfrak F}_{A^{**}}$, and 
that $J = \overline{EA}$ is
an ${\mathfrak F}$-ideal,
where $E = \{ x_t \}$ (using Theorem \ref{nowha}).    Also $p x = x$ if $x \in J$,
and $p \in J^{\perp \perp}$.  So 
$pA^{**} = J^{\perp \perp}$, from which it is easy to see that $p$ is a support
idempotent of $J$.     

The converse is obvious by the definition of support idempotent above,
and the fact that $\overline{EA} = s_E A^{**} \cap A$.  
 \end{proof}

\begin{corollary} \label{gen}   If $A$ is a
commutative approximately unital Banach algebra,
and $E \subset  {\mathfrak r}_{A}$,
then the closed subalgebra generated by $E$ 
has a bai in ${\mathfrak F}_{A}$.
\end{corollary} \begin{proof} 
In Theorem \ref{nowha}
we constructed a bai in ${\mathfrak F}_{A}$
for $\overline{EA}$,
and this bai is clearly in the closed subalgebra generated by $E$, and is a bai for that subalgebra.
\end{proof}

If $A$ is any approximately unital commutative Banach algebra,
define $A_H = \overline{{\mathfrak F}_A A}$.
This is an ideal of the type in Theorem {\rm \ref{nowha}},
and is the largest such (by that result). 

If $A$ is an operator algebra it is proved in \cite{BRII}
that $A = {\mathfrak r}_A - {\mathfrak r}_A$ iff $A$ has 
a cai.   In our setting we at least  have:

\begin{corollary} \label{gennew}    If $A$ is a commutative approximately unital  Banach algebra
which is generated by ${\mathfrak r}_A$ as a Banach algebra
 (and certainly  if $A = {\mathfrak r}_A - {\mathfrak r}_A$), 
then $A$ has a bai in ${\mathfrak F}_A$.   
\end{corollary} \begin{proof} 
This follows from Corollary \ref{gen}  because $A$ is generated by ${\mathfrak r}_A$ in this case,
and hence is generated by  ${\mathfrak F}_A$ since 
${\mathfrak r}_A = \overline{\Rdb^+ {\mathfrak F}_A}$. 
\end{proof} 

Conversely, if $A$ is $M$-approximately unital or has a sequential cai satisfying certain conditions discussed in Section 6, 
then we saw in Section 6 that $A = {\mathfrak r}_A - {\mathfrak r}_A$.
Indeed we saw in the $M$-approximately unital case in Theorem  \ref{lbdiff} that 
$$A = \Rdb^+ ({\mathfrak F}_A - {\mathfrak F}_A)
\subset  {\mathfrak r}_A - {\mathfrak r}_A
\subset A .$$

We do not know if it is always true if,
as in the operator algebra case, for any 
approximately unital commutative Banach algebra
we have $A_H = {\mathfrak r}_A - {\mathfrak r}_A
= \Rdb^+ ({\mathfrak F}_A - {\mathfrak F}_A)$.

\section{$M$-ideals which are ideals}  \label{Mididsect}

We now turn to an interesting class of closed approximately unital ideals in a general  approximately unital Banach algebra
that generalizes the class of approximately unital  closed two-sided ideals in operator algebras.   (Unfortunately,
we see no way yet to apply e.g.\ the theory in \cite{BZ} to 
generalize the results in this section to one-sided ideals.)  
The
study of this class was initiated by Roger Smith and J. Ward \cite{SW1,SW2,SW3}.  We will use basic ideas from these papers 
(see also Werner's theory of inner ideals in the sense of \cite[Section V.3]{HWW}).

First, let $A$ be a unital Banach algebra.  We define an {\em $M$-ideal ideal} in $A$ to be a 
subspace $J$ of $A$ which is  
an $M$-ideal in $A$, such that if $P$ is the $M$-projection then
$z = P1$ is central in $A^{**}$ (the latter  is automatic 
for example if $A$ is commutative and Arens regular).   
Actually it suffices in all the arguments below that simply $z a = a z$ for $a \in A$, 
but for convenience we will stick to the `central' hypothesis.   
By \cite[Proposition 3.1]{SW1}, $z$ is a 
hermitian projection of norm $1$ (or $0$).    It is then a consequence of Sinclair's 
theorem on hermitians \cite{Si} that $z$ is accretive, 
indeed $W(z) \subset [0,1]$.   The proof of  \cite[Proposition 3.4]{SW1} shows that 
$(1-z)J^{\perp \perp} = (0)$ (it is shown there that $z J^{\perp \perp} z \subset J^{\perp \perp} = J_1$ 
in the notation there, and 
that $(1-z)J \subset J_2$, but clearly $z J \subset J_1$ so that $(1-z)J \subset  (J - J_1) \cap J_2  \subset  J_1 \cap J_2 = (0)$).  
 It also shows that $z (I-P) A^{**} = 0$, so that $P$ is simply 
left multiplication by $z$, and  $J^{\perp \perp} = zA^{**}$.
Since the latter is an ideal, so is $J = J^{\perp \perp} \cap A$ an ideal in $A$.  Moreover, $J$ is 
approximately unital since $z$ is a mixed identity for $J^{\perp \perp}$  of norm $1$.   We call $z$ the 
support projection of $J$, and write it as $s_J$.    The correspondence $J \mapsto s_J$ 
is bijective on the class of $M$-ideal ideals.  

\begin{proposition} \label{isfidi}  An $M$-ideal ideal $J$ in a unital
Banach algebra $A$ is $M$-approximately unital, indeed $J$ has a  
 cai in $\frac{1}{2}  {\mathfrak F}_A$.   
Also $J$ is a two-sided ${\mathfrak F}$-ideal in $A$, and $J = \overline{EA} = \overline{AE}$ for 
some subset $E \in J \cap {\mathfrak F}_A$.
\end{proposition}  \begin{proof}   By Proposition
\ref{auf2}, $J$ is $M$-approximately unital, so by 
Theorem \ref{lbai} it has a cai in $\frac{1}{2}  {\mathfrak F}_J
= J \cap \frac{1}{2}  {\mathfrak F}_A$.  (The latter 
equality following from Proposition \ref{auf2} applied in $A^1$.)   Thus $J$ is a two-sided ${\mathfrak F}$-ideal.   We also 
deduce from Proposition \ref{auf2} that $J^1 \cong J + \Cdb 1_A$.
Hence  
$J = \overline{EA} = \overline{AE}$ for some $E \subset J \cap {\mathfrak F}_A$, for example
take $E$ to be the cai above.  
\end{proof} 
     
The converse of the last result fails.  Indeed even in a commutative algebra,
not every ideal $\overline{EA}$ for a subset $E \in {\mathfrak F}_A$, is an $M$-ideal ideal, nor need have a cai in $\frac{1}{2}  {\mathfrak F}_A$
(see Example \ref{Ex2}).

Suppose that $J_1$ and $J_2$ are $M$-ideal ideals in $A$, and that 
$P_1, P_2$ are the corresponding  $M$-projections on $A^{**}$ with $z_k = P_k 1$ central in $A^{**}$.  As in Corollary \ref{supp3},
$J_1 \subset J_2$ iff $z_2 z_1  = z_1$, and the latter equals $z_1 z_2$.
So the correspondence $J \mapsto s_J$ is an order embedding with respect to 
the usual ordering of projections in $A^{**}$.
Then by facts above, $P_1 P_2 (1) = P_1 (z_2) = z_1 z_2$, and this is central in $A^{**}$.  
  Similarly, $(P_1 + P_2 - P_1 P_2) 1 = z_1 + z_2 - z_1 z_2$, and this is central in $A^{**}$. 
Hence $J_1 \cap J_2$ and $J_1 + J_2$  are $M$-ideal ideals in $A$.  

To describe the matching fact about `joins' of an infinite family of ideals 
we introduce some notation.
Set  $N$ to be $A^{**}$.  
We will use the fact that $N$ contains a  commutative
von Neumann algebra.   We recall
that the  {\em centralizer} $Z(X)$ of a dual Banach space $X$ is a  weak* closed
subalgebra of $B(X)$, and it is densely spanned in the norm
topology by its contractive projections, which are the $M$-projections (see e.g.\ \cite{HWW}
and \cite[Section 7.1]{BZ}).
It is also a commutative  $W^*$-algebra in the weak* topology from $B(X)$.  
By \cite[Theorem V.2.1]{HWW}), 
the map $\theta : Z(N) \to N$ taking $T \in Z(N)$ to $T(1)$  is 
an isometric  homomorphism,
and it is weak* continuous by definition of the
weak* topology on $B(N)$ and hence on $Z(N)$.
Therefore by the Krein-Smulian theorem
the range of $\theta$ is weak* closed, and
$\theta$ is a  weak*  homeomorphism onto its range.   
  Thus $Z(N)$ is identifiable with a weak* closed subalgebra $\Delta$ of $N$, 
which is a commutative  $W^*$-algebra,
via the map $T \mapsto T(1)$.
 All computations can be done inside this
commutative  von Neumann algebra.   Indeed the ordering  of support projections $z_1,  z_2$, and their
`meet' and `join', which we met a couple of paragraphs above, are simply the standard operations $z_1 \leq z_2, z_1 \vee z_2,
z_1 \wedge z_2$ with projections, computed in the $W^*$-algebra $\Delta$.  
Of course we are specifically interested in
the weak* closed subalgebra consisting of elements
in $\Delta$ that commute with $A$.  The projections in this
subalgebra densely span a commutative  von Neumann algebra inside
$\Delta$.    

\begin{lemma} \label{unmid}  The closure of the span of  a family  $\{ J_i : i \in I \}$ of $M$-ideal ideals in 
a unital Banach algebra $A$, 
is an $M$-ideal ideal in $A$.
\end{lemma} \begin{proof} Let 
$\{  P_i : i \in I \}$ be the corresponding family of $M$-projections on $A^{**}$ with $z_i = P_i 1$ central in $A^{**}$.
Let $\Lambda$ be the collection of finite subsets of
$I$ ordered by inclusion.   For $F \in \Lambda$ let $J_F = \sum_{i \in F} \, J_i$, by the above 
this will be an $M$-ideal ideal  in $A$ whose  
support projection  $s_{J_F}$ corresponds to $P_F(1)$, where $P_F$ is the $M$-projection 
for $J_F$.     Next suppose that $(P_F)$ has weak* limit $P$ in $Z(N)$; by the 
theory of $M$ projections $P$ is the $M$-projection corresponding to the $M$-ideal 
$J = \overline{\sum_i \, J_i} = \overline{\sum_{F \in \Lambda} \, J_F}$.    We have $P(1) = z$ is the
weak* limit of the $(z_i)$, this is a contractive hermitian projection in the 
ideal $J^{\perp \perp}$.   For $\eta \in N$ we have
$z \eta 
\in J^{\perp \perp}$ so that $$z \eta = P (z \eta ) = \lim_i \, P_i (z \eta ) = \lim_i \, z_i z \eta
=  \lim_i \, z_i  \eta =  \lim_i \, \eta  z_i = \eta z.$$
Thus $z$ is central in $N$, and so $J$ is an  $M$-ideal ideal with support projection $z$,
and $z$ is the supremum $\vee_i \, z_i$ in $\Delta$.    \end{proof}  

Next assume that $A$ is an approximately unital Banach algebra.  We define an {\em $M$-ideal ideal} in $A$ to be a 
subspace $J$ of $A$ which is  
an $M$-ideal in $A^1$, such that $z = P1$ is central in $A^{**}$ 
 (or, as we said above, simply that  $z a = a z$ for $a \in A$,  which will then allow  $M$-approximately unital $A$ to 
always be an $M$-ideal ideal in itself).
We may then apply the theory in the last several paragraphs   to 
$A^1$; thus $N = (A^1)^{**}$ there.  Set $\Delta'$ to be the weak* closure in $\Delta$ of the
span of those projections that happen to be in $A^{**}$.  This 
is also a commutative  $W^*$-algebra.

\begin{theorem} \label{Mis2}  If $A$ is an approximately unital Banach algebra then    
the class of 
$M$-ideal ideals in $A$ forms a lattice, indeed the 
intersection of a finite number, or the closure of the 
sum of any collection, of $M$-ideal ideals is again 
an $M$-ideal ideal. 
 The correspondence between $M$-ideal ideals $J$ in $A$ and their support 
projections $s_J$ in $\Delta' \subset  A^{**}$,
is bijective and preserves order, and preserves
finite `meets' and arbitrary `joins'.   That is,
$s_{J_1 \cap J_2} = s_{J_1} s_{J_2}$ for $M$-ideal 
ideals $J_1, J_2$ in $A$; and if $\{ J_i : i \in I \}$ is any collection of $M$-ideal ideals in $A$ and $J$ is the closure of their span, then  
$s_{J}$  is the supremum in $\Delta' \subset  A^{**}$  of
$\{ s_{J_i} : i \in I \}$.  
 \end{theorem}
\begin{proof}   
This result is essentially a summary of some 
facts above, these facts applied to $A^1$ instead of $A$, and with  $N = (A^1)^{**}$.   \end{proof} 

Clearly any $M$-ideal ideal in $A$ is Hahn-Banach smooth in $A^1$ \cite{HWW}, hence in $A$.

If $J$ is an $M$-ideal ideal then we call $s_J$  above a {\em central open projection} in $A^{**}$.    
Clearly such open projections $p$ are  
weak* limits of nets $x_t \in \frac{1}{2}  {\mathfrak F}_{A}$
with $p x_t = x_t p = x_t$.   However  not every
projection in $A^{**}$ which is such a weak* limit is
the support idempotent of an $M$-ideal ideal (again
see Example \ref{Ex2}).   Nonetheless we expect to generalize 
more  of the theory in \cite{BHN,BRI, BRII} of open projections
and r-ideals to this setting.  For a start, it is now clear that  sups of any collection, and inf's of finite collections,
of central open projections, are central open projections. 
 If $A$ is an $M$-approximately unital Banach algebra
then the mixed identity $e$ for $A^{**}$  of norm $1$ is a 
central open projection.

\begin{proposition} \label{lsc}  If $A$ is an approximately unital Banach algebra
then any central open projection is lower semicontinuous on $Q(A)$.
\end{proposition} \begin{proof}  If $A$ is unital
then this result is in \cite{SW2}, and we use this below.   
Let $\varphi_t \to \varphi$ weak* in $Q(A)$,
and suppose that  $\varphi_t(p)  \leq r$ for all $t$.  Write $\varphi_t = c_t \, \psi_t$ 
for $\psi_t \in S(A)$, and let $\hat{\psi_t} \in S(A^1)$
be a state extending $\psi_t$.
By replacing by a subnet we can assume that  $c_t \to s \in [0,1]$.
  A further subnet
$\widehat{\psi_{t_\nu}} \to \rho \in S(A^1)$ 
weak*.   Thus $\varphi = s \, \rho_{|A}$, since
$$\varphi_{t_\nu}(a) = 
c_{t_\nu}  \, \psi_{t_\nu}(a) = 
c_{t_\nu} \, \widehat{\psi_{t_\nu}}(a) \to s \, \rho(a) , \qquad a \in A.$$ 
  By the result from \cite{SW2} mentioned
above, $\rho(p) \leq \liminf_\nu \,
\widehat{\psi_{t_\nu}}(p) = \liminf_\nu \, \psi_{t_\nu}(p).$
Hence $$\varphi(p) = s \rho(p) \leq \liminf_\nu \, s \,
\psi_{t_\nu}(p)
= \liminf_\nu \, c_{t_\nu} \, \psi_{t_\nu}(p)  \leq r,$$
as desired.   \end{proof}

Given a central open projection $p \in A^{**}$ we set
$F_p = \{ \varphi \in Q(A) : \varphi(p) = 0 \}$.

\begin{theorem} \label{isfa}  Suppose
that  $A$ is a scaled 
approximately unital 
Banach algebra,
 and $p$ is a central open projection in $A^{**}$,
and $J = p A^{**} \cap A$ is the corresponding 
ideal.   Then $F_p = Q(A)
\cap J^\perp$, and this is a weak* closed face
of $Q(A)$.  Moreover, the assignment $\Theta$ taking
$p \mapsto F_p$ (resp.\ 
 $J \mapsto F_p$), from the set of
central open projections (resp.\ $M$-ideal ideals of $A$)
into the set of weak* closed faces of $Q(A)$,
 is one-to-one and is a (reverse) order embedding.
Moreover, `sups' (that is, `joins' of arbitrary families)
are taken by $\Theta$ to
intersections of the corresponding faces.
 \end{theorem}  \begin{proof} 
If $J = p A^{**} \cap A$ and $\varphi \in Q(A) 
\cap J^\perp$ then
$\varphi \in F_p$ since $p \in J^{\perp \perp}$.
Conversely, if $\varphi \in F_p$ has norm 1 then we 
have $$1 = \Vert \varphi \Vert = \Vert \varphi \cdot p \Vert + 
\Vert \varphi \cdot (1-p) \Vert \geq |\varphi (1-p)| = 1.$$
Thus $\varphi \cdot p = 0$, and so $\varphi \in Q(A) 
\cap J^\perp$.

If $\varphi \in F_p$ and $\varphi = t \psi_1 + (1-t) \psi_2$
for $\psi_1, \psi_2 \in Q(A)$ and $t \in [0,1]$, then 
it is clear that $\psi_1, \psi_2 \in F_p$.  So $F_p$ is a face 
of $Q(A)$.    Since $F_p = Q(A)
\cap J^\perp$ it is weak* closed.

Write $F^1_p = \{ \varphi \in S(A^1) : \varphi(p) = 0 \}$.  Suppose that $\varphi_t \to \varphi \in Q(A)$ weak*,
with $\varphi_t \in F_p$ and $\varphi \neq 0$.  Suppose that $\varphi_t = c_t \psi_t$ with $\psi_t \in S(A)$.   
We may assume that $\psi_t \in S(A^1)$, and then $\psi_t \in F^1_p$.
By \cite{SW1,SW2}, $F^1_p$ is weak* closed, so we have a weak* convergent subnet $\varphi_{t_\mu}
 \to \psi \in F^1_p$.   A further subnet of the $c_{t_\mu}$ converges to $c \in [0,1]$ say.  In fact $c \neq 0$ 
or else $\varphi_{t_\mu}$ has a norm null subnet, so that $\varphi = 0$.
Now it is clear that  $c \psi_{|A} = \varphi \in F_p$.    So $F_p$ is weak* closed.    

If we have
two central open projections $p_1 \leq p_2$ then
$w = p_2 - p_1$ is a hermitian projection in $(A^1)^{**}$,
so that as we said above $W(z) \subset [0,1]$.  Thus
it is clear that $\varphi(p_1) \leq \varphi(p_2)$ for states
$\varphi \in S(A)$.   Hence $F_{p_2} \subset F_{p_1}$.

Conversely, suppose that $F_{p_2} \subset F_{p_1}$.
If $\varphi \in F^1_{p_2}$ and $\varphi$ is nonzero
on $A$ then since it is real positive on $A$ it 
will be a positive multiple of a state $\psi$ on $A$. 
We have $\psi \in F_{p_2} \subset F_{p_1}$, so that
$\varphi \in F^1_{p_1}$.    That is, $F^1_{p_2} \subset F^1_{p_1}$.
We are now in the setting of \cite{SW1,SW2}, from where we see 
that these are split faces of $S(A^1)$, and are weak* closed.    Let $N_1 \subset N_2$ be the 
complementary split faces.   
  We may view $p_1, p_2$ as affine lower semicontinuous functions $f_1, f_2$
 on $S(A^1)$.   As in those references, we  have $f_k = 0$ on $F^1_{p_k}$, and $f_k = 1$ on $N_k$.
 From this
and the theory of split faces \cite[Section II.6]{Alf}
it is easy to see that $f_1 \leq f_2$.
 That is, $\varphi(p_2 - p_1) \geq 0$
for all $\varphi \in S(A^1)$.   By \cite{Mag} this 
is also true if $\varphi \in S((A^1)^{**})$,
and hence if $\varphi \in S(\Delta)$.  Therefore $p_1 \leq p_2$
in $\Delta$, so that indeed $p_1 \leq p_2$ in the usual ordering
of projections in $A^{**}$.       

The last assertion follows from the identity
$Q(A) \cap (\sum_i \, J_i)^\perp = \cap_i \, (Q(A) \cap J_i^\perp)$.
\end{proof}

Note that 
the support projection $s(x) \notin \Delta$ in general
if $x \in {\mathfrak F}_A$.  This can be overcome by
restricting to the class  where this  is 
true--but unfortunately this class seems
often  only 
to be interesting if $A$ is commutative.
Thus if 
 $A$ is an
approximately unital Banach algebra,
write ${\mathfrak F}'_A$ for the set of 
$x \in {\mathfrak F}_A$ such that 
multiplying on the left by $s(x)$ in the second Arens product is an $M$-projection  
on $N = (A^1)^{**}$, and $s(x)$ is commutes with $A^1$ (again the latter is
automatic if $A$ is commutative and Arens regular).   (Note that 
 if $A$ is $M$-approximately unital then multiplying on the left by $s(x)$ 
is an $M$-projection  
on $A^{**}$ iff it is an $M$-projection on $(A^1)^{**}$.) 
Define an $m$-ideal in $A$ to be an  ideal of form $\overline{EA}$ for a subset $E \subset {\mathfrak F}'_A$. 
If $A$ is also a commutative  operator algebra then the $m$-ideals in $A$ are exactly the closed ideals with a cai, by the characterization of r-ideals in \cite{BRI} (see also \cite{ER}), since in this case ${\mathfrak F}'_A = {\mathfrak F}_A$. 

\begin{proposition} \label{Mis0}  If $A$ is an
approximately  unital Banach algebra then 
any $m$-ideal in $A$ is an $M$-ideal ideal in $A$.    \end{proposition}
\begin{proof}    Suppose that $x \in {\mathfrak F}'_A$.  Setting  $J_x = \overline{xA} \subset s(x) A^{**} \cap A$,
we have $J_x^{\perp \perp} = s(x) A^{**} = s(x) N$, as in the 
proof of Corollary \ref{lba}.
So $J_x =  s(x) A^{**} \cap A$ is an $M$-ideal ideal.
Then $\overline{EA} = \overline{\sum_{x \in E} \, xA}$
is also an $M$-ideal ideal by Theorem \ref{Mis2}.  \end{proof} 

The above class is perhaps also a context to which there is a natural 
generalization of some of the  results in \cite{BHN,BRI,BRII,Hay} related to
noncommutative peak interpolation, and noncommutative peak and $p$-sets (see \cite{Bnpi} for a short survey of
this topic).
However one should not expect the ensuing  theory to be particularly useful for noncommutative 
algebras since the projections in this section are all `central'.  
  
Indeed it is unlikely that one could generalize to general Banach algebras the main noncommutative peak 
interpolation results surveyed in \cite{Bnpi}, or see e.g.\ \cite{Hay,BHN,BRII,Bord}.
However we end with one nice noncommutative peak interpolation result concerning $M$-ideal ideals in general Banach algebras, which can 
also be viewed as  a `noncommutative Tietze theorem'.  In particular it also solves a problem that arose at the time
of   \cite{BRII}, and was mentioned in \cite{BRIII}, namely whether 
 ${\mathfrak r}_{A/J} = q_J({\mathfrak r}_{A})$
when $J$ is an approximately unital ideal in an 
operator algebra $A$, and $q_J : A \to A/J$ is the quotient map.  
In \cite{BRI} it was shown that 
${\mathfrak F}_{A/J} = q_J({\mathfrak F}_{A})$, and it is easy to see  that
$q_J({\mathfrak r}_{A}) \subset {\mathfrak r}_{A/J}$.  
In fact a much more general fact is true.   The main new ingredient needed  
is \cite[Theorem 3.1]{CSSW}.   Their proof of this result, while remarkable and deep,
clearly contains misstatements.   However we were able to confirm that (a small modification of)  their proof works
at least in the case of unital Banach algebras.    For the readers interest we will give a rather
different, and more direct, proof  of their full result.

Let $(X,e)$ be a pair consisting of a Banach space $X$ and an element $e\in X$ such that $\|e\|\le1$. 
Let 
\[
S_e(X) = \{ \varphi\in X^* : \|\varphi\|=1=\varphi(e) \}
\ \mbox{ and }\ 
W(x) =W_X^e(x)=\{ \varphi(x) : \varphi\in S_e(X) \}
\]
denote respectively the state space and the numerical range of $x\in X$, relative to $e$. 
Of course, these are empty if $\| e\|<1$.   Below  we write $B(\lambda,r)$ for the closed disk centered at $\lambda$ of radius $r$.  The following
formula in the Banach algebra case is attributed to Williams in \cite{BDNRII}, and it may be proved by a tiny modification of the proof at the 
end of page 1 there.

\begin{lemma} \label{LSW}  {\rm (Williams formula)} \
For every $x\in X$, one has 
\[
W(x) = \bigcap_{\lambda\in\Cdb} B(\lambda,\| x- \lambda e \|).
\]
In particular, $W_X^e(x)=W_{X^{**}}^e(x)$ for every $x\in X$. 
\end{lemma}

\begin{theorem} \label{CSSW}  {\rm (Chui--Smith--Smith--Ward)}  \
Let $(X,e)$ be as above. 
Suppose that $J$ is an $M$-ideal in $X$ and $x\in X$ is such that 
$W_{X/J}^{Q(e)}(Q(x)))$ has non-empty interior, 
where $Q\colon X\to X/J$ is the quotient map. 
Then there exists $y\in J$ such that $\|x-y\|_X=\| Q(x)\|_{X/J}$ 
and $W^e_X(x-y) = W_{X/J}^{Q(e)}(Q(x))$.
\end{theorem}

\begin{proof}  For a bounded convex subset $C\subset \Cdb$, $\alpha\in C$, and $\varepsilon>0$, 
we define 
\[
N(C,\alpha,\varepsilon)=\{ \alpha + (1+\varepsilon)(\gamma - \alpha) : \gamma \in C\}.
\]
It is an exercise that the $N(C,\alpha,\varepsilon)$ are open convex neighborhoods of $C$ if 
$\alpha\in {\rm int}(C)$, and they shrink as $\epsilon$ decreases.

Let $x\in X$ be given, and fix $\alpha\in  {\rm int}( W_{X/J}^{Q(e)}(Q(x)) )$.   Then $|\alpha| < \Vert Q(x) \Vert$.
Now $N(W_{X/J}^{Q(e)}(Q(x)),\alpha,1)$ is an open neighborhood of the 
compact subset $W_{X/J}^{Q(e)}(Q(x))$.  The latter equals $\bigcap_{\lambda\in\Cdb}B(\lambda, \|Q(x-\lambda e)\|_{X/J})$ by the lemma, and so 
we can find 
$0=\lambda_0,\lambda_1,\ldots,\lambda_n\in\Cdb$, and $\delta>0$, such that 
\[
\bigcap_{i}B(\lambda_i, \|Q(x-\lambda_i e)\|_{X/J}+\delta)\subset N(W_{X/J}^{Q(e)}(Q(x)),\alpha,1).
\]
Let $z_0 = P(x-\alpha e) \in J^{\perp\perp}$ and  $\lambda\in\Cdb$.   Since $P$ is an $M$-projection, 
$$\| x - z_0 - \lambda e \| = \max\{\| P((\alpha - \lambda) e) \| , \| (I-P)(x -  \lambda e + y) \| \}
, \qquad y \in J ,$$  which is dominated by 
$$ \max\{ |\lambda - \alpha |, \| Q( x -\lambda e) \|_{X/J}\} = \| Q( x -\lambda e) \|_{X/J}$$
since 
$\alpha \in \bigcap_{\lambda\in\Cdb}B(\lambda, \|Q(x-\lambda e)\|_{X/J})$.
Thus $\| x - z_0 - \lambda_i e \| < r_i$ for each $i$, where $r_i = \| Q( x -\lambda_i e) \|_{X/J} + \delta$.   
 Hence by Lemma \ref{ban} there exists $y_0\in J$ such that 
$\| x - y_0 -\lambda_i e\|< r_i$ for all $i$. 
Indeed using that lemma similarly to 
some other proofs in our paper,  if $x'\in X$ and $z\in J^{\perp\perp}$ are such that $\| z + x'\|_{X^{**}}<r$, and if
$\{ y_i\}$ is a net in $J$ which converges to $z$ weak*, 
one can find a net $\{ y_j'\}$ of convex combinations of the $y_j$ such that $y_j'\to z$ and 
$\|y_j' + x'\|_X<r$. One can iterate this procedure and obtain the same conclusion for 
any finite sequence  $x_1', \cdots , x_m' \in X$ such that $\| z + x_i'\|_{X^{**}}<r_i$ for all $i = 1, \cdots, m$. 

It follows that $x_0 = x-y_0$ satisfies 
$\| x_0 \| <\|Q(x)\|_{X/J}+\delta$, and
\[
|\varphi(x_0) - \lambda_i |= |\varphi( x- y_0-\lambda_i e )| \le\|Q(x-\lambda_i e)\|_{X/J}+\delta, \qquad
\varphi\in S_e(X).
\]
This implies 
$W_X(x_0)\subset \bigcap_{i}B(\lambda_i,\|Q(x-\lambda_i e)\|_{X/J}+\delta)
 \subset N(W_{X/J}^{Q(e)}(Q(x)),\alpha,1)$.

Now we iterate the above process, controlling the increments. If $\epsilon > 0$ let
$N(\varepsilon)$ denote the set of those $x' \in x+J \subset X$ such that 
$$\|x'\|_X\le\|Q(x)\|_{X/J}+\frac{\varepsilon}{1-\varepsilon}(\|Q(x)\|_{X/J} - |\alpha|),$$ 
and such that $W_{X}(x')\subset N(W_{X/J}^{Q(e)}(Q(x)),\alpha,\varepsilon)$.
 Note that $x_0\in N(1)$ (the first condition in the definition of $N(1)$ we treat as being vacuous). 

Claim:
For any $n = 0, 1, 2, \ldots$ and $x_n \in N(2^{-n})$, there is $x_{n+1} \in N(2^{-{(n+1)}})$ 
such that $\| x_{n+1} - x_n \| \le 3 \cdot 2^{-n} \| Q(x) \|$ when $n \geq 1$. 

Before we prove the Claim, we finish the proof of the theorem. 
Note that  if $n\geq1$ then $\|x_n\|\le 2\|Q(x)\|_{X/J}$ by the first clause in the definition 
of $N(\epsilon)$.   It follows from this and the inequality in the Claim that the norm-limit $v =\lim x_n$ 
exists in $x+J$. It satisfies $\|v\|\le\|Q(x)\|_{X/J}$ by  the first clause in the definition 
of $N(2^{-n})$, and 
$W_X(v)\subset W_{X/J}(Q(x))$ since  by  the second clause in that definition, 
$$\varphi(v)=\lim\varphi(x_n)\in \bigcap_n N(W_{X/J}^{Q(e)}(Q(x)),\alpha,2^{-n}) = W_{X/J}(Q(x)), \qquad
\varphi\in S_e(X).$$   That $W_{X/J}(Q(x)) \subset W_X(v)$ is an easy exercise.  
This completes the proof of the theorem.

To prove the  Claim, let $z = 2^{-n} P(x_n-\alpha e) \in J^{\perp\perp}$.   Using the first clause in the definition 
of $x_n \in N(2^{-n})$ we have 
$$\|z\| \leq 2^{-n}(\| x_n \| + |\alpha|) < 3 \cdot 2^{-n} \| Q(x) \|.$$   Also, $P(x_n-z) = (1-2^{-n}) x_n + 2^{-n} \alpha$, so by an argument similar to the
$M$-projection argument in the second paragraph of the proof, we have 
\[
\| x_n-z \|\le \max\{  (1-2^{-n})\|x_n\|+2^{-n}|\alpha|, \| Q(x) \|_{X/J}\} .
\]
The latter equals $\| Q(x) \|_{X/J}$, using  the first clause in the definition 
of $x_n \in N(2^{-n})$.

Suppose that $\varphi_1\in S_e(X^{**})$ with $\varphi_1\circ P=\varphi_1$. 
There exists  $\gamma\in W_{X/J}^{Q(e)}(Q(x))$ such that 
$\varphi_1(x_n)=\alpha + (1+2^{-n})(\gamma - \alpha)$, by  the second clause in the definition 
of  $x_n \in N(2^{-n})$.
Hence, one has 
\[
\varphi_1(x_n-z) = \alpha + (1-2^{-n})(\varphi_1(x_n)-\alpha)
= \alpha + (1-2^{-2n})(\gamma - \alpha),
\] and the latter is in  $W_{X/J}^{Q(e)}(Q(x))$ since it is a convex combination of $\alpha$ and
$\gamma$.
Next, suppose that $\varphi_2\in S_e(X^{**})$ with $\varphi_2\circ P=0$. 
Then $\varphi_2$ induces a `state' on $(X/J)^{**} \cong X^{**}/J^{\perp \perp}$, so that  
\[
\varphi_2(x_n-z)=\varphi_2(x_n) \in W_{(X/J)^{**}}^{Q(e)}(Q(x))=W_{X/J}^{Q(e)}(Q(x)).
\]
Thus $W_{X^{**}}^e(x_n-z)\subset W_{X/J}^{Q(e)}(Q(x))$, 
since any $\varphi\in S_e(X^{**})$ is a convex combination of $\varphi_1 = \varphi \circ P$ and $\varphi_2 = \varphi \circ (I-P)$ 
as above.   Here we are using the $L$-projection argument we have seen several times, relying on 
$$1 =  \varphi(e) = \varphi_1(e) + \varphi_2(e) \leq \Vert \varphi_1 \Vert  + \Vert \varphi_2 \Vert = 1.$$  
By the Williams formula (Lemma \ref{LSW}), 
\[
\bigcap_{\lambda\in\Cdb} B(\lambda, \| x_n-z-\lambda e\|_{X^{**}}) = W_{X^{**}}^e(x_n-z) 
\subset W_{X/J}^{Q(e)}(Q(x)).
\]
Let $\delta=2^{-(n+1)}$.  By the argument at the start of the proof
one can choose a finite sequence $\lambda_1,\ldots,\lambda_m \in\Cdb$ such that 
\[
\bigcap_i B(\lambda_i, \| x_n-z-\lambda_i e\|) 
\subset N(W_{X/J}^{Q(e)}(Q(x)),\alpha,\delta) .
\]
Choose $r_i >  \| x_n-z-\lambda_i e\|$ with $\bigcap_i B(\lambda_i, r_i) 
\subset N(W_{X/J}^{Q(e)}(Q(x)),\alpha,\delta)$.   By the argument using
 Lemma \ref{ban} in the second paragraph of the proof, we can replace $z$ in these
inequalities by an element in $J$.  Thus there exists $y \in J$ such that 
$\| y \| < 3 \cdot 2^{-n}  \| Q(x) \|$, 
$\| x_n - y \|\le \|Q(x)\|_{X/J}+\frac{\delta}{1-\delta}(\|Q(x)\|_{X/J} - |\alpha|)$,
and $$W(x_n-y) \subset \bigcap_i B(\lambda_i, \| x_n-y-\lambda_i e\|) \subset \bigcap_i B(\lambda_i, r_i)
\subset N(W_{X/J}^{Q(e)}(Q(x)),\alpha,\delta).$$ 
Hence $x_{n+1} =x_n-y\in N(\delta)$, which completes the proof of the Claim. 
\end{proof}

We next deal with the exceptional case when $W_{X/J}^{Q(e)}(Q(x)))$ has empty interior, which by convexity hapens
 exactly when it is a line segment or point.

\begin{corollary} \label{staTi}     Suppose that  $J$ is an $M$-ideal ideal (or simply an ideal 
which is an $M$-ideal) in a 
unital Banach algebra $A$.    Let $x \in A/J$ with $K = W_{A/J}(x)$.
 Then \begin{itemize} \item [(1)]  
 If 
$K$  is a point, 
 then 
there exists $a \in A$  with $\Vert a \Vert = \Vert x \Vert$ and with $W_A(a) = W_{A/J}(x)$.  
\item [(2)]    If $K = W_{A/J}(x)$ is a nontrivial line segment then {\rm (1)}
 is true `within epsilon'.  More precisely, in this case  
let $\hat{K}$ be any thin triangle with $K$ as one of the sides (so contained in a thin rectangle with side $K$).
Then there exists $a \in A$  with $\Vert a \Vert = \Vert x \Vert$ and with 
  $K \subset W_A(a) \subset \hat{K}$. \end{itemize} \end{corollary}

\begin{proof}   If $K$  is a point, then $x$ is a scalar multiple of $1$, so this case is
obvious.  
For (2), if  $K$ is a nontrivial line segment, choose 
$\lambda$ within a small distance $\epsilon$ of the midpoint of the line.   Then replace $A$ by $B = A \oplus^\infty \Cdb$, replace $J$ by  $I = J \oplus (0)$, and consider $(x,\lambda) \in B/I$.
It is easy to see that $W_{B/I}((x,\lambda))$ is the  convex hull $\hat{K}$ of $K$ and $\lambda$.  By Theorem \ref{CSSW} there exists
$(a,\lambda) \in B$ with $W_B((a,\lambda)) = \hat{K}$.  If $\epsilon$ is small enough, we also have   $\Vert a \Vert = \Vert x \Vert$ (since then $|\lambda|$ is dominated by 
the maximum of the moduli of two numbers in the numerical range, which is dominated by 
$\Vert x \Vert \leq \Vert a \Vert$).  
However similarly  $W_B((a,\lambda))$ is the
convex hull  of $W_A(a)$ and $\lambda$, which makes the rest of the proof of (2) an easy exercise in the geometry of triangles. \end{proof}

We remark that in a previous version of our paper
the last result (and Theorem \ref{CSSW} in the unital 
Banach algebra case) was stated as a `Claim', not as a theorem.   Thus it is
referred to in \cite{Bord} as `the Claim at the end of' the present paper. 

We can now answer the open question referred to above Theorem \ref{CSSW}.

\begin{corollary}      If $A$ is an approximately unital Banach algebra, and if $J$ is an  $M$-ideal ideal in $A$,
then  ${\mathfrak r}_{A/J} = q_J({\mathfrak r}_{A})$.    In particular ${\mathfrak r}_{A/J} = q_J({\mathfrak r}_{A})$ for approximately unital 
closed two sided ideals $J$ in any (not necessarily  approximately unital)
 operator algebra $A$.
\end{corollary}  

\begin{proof}    
First suppose that $A$ is unital.   We leave it as an exercise that $q_J({\mathfrak r}_{A}) \subset {\mathfrak r}_{A/J}$.
The converse inclusion  follows from 
Theorem \ref{CSSW} and Corollary \ref{staTi}  (in the line situation
take the triangle above and/or  to the right of $K$).  
Next suppose that  $A$ is a nonunital approximately unital Banach algebra, and  that $A/J$ is also nonunital.
Then by the last paragraph of 
A.4.3 in \cite{BLM}, the inclusion  $A/J \subset A^1/J$  induces an isometric isomorphism $A^1/J \cong (A/J)^1$.
The result then follows by applying the unital case to the canonical map from    
$A^1$ onto $(A/J)^1$.    If  $A/J$ was unital then one can reduce to the previous case where it is not, by considering 
the ideal $J \oplus^\infty K$ in $A \oplus^\infty B$, where $K$ is an approximately unital  ideal in (e.g.\ a commutative $C^*$-algebra)
$B$ such that $B/J$ is not unital.  For this latter trick one needs to know that ${\mathfrak r}_{A \oplus^\infty B} = \{ (x,y) \in 
A \oplus^\infty B : x \in {\mathfrak r}_{A}, y \in {\mathfrak r}_{B} \}$ for approximately unital Banach algebras, but this is an easy exercise (and a similar relation holds 
for ${\mathfrak F}_{A \oplus^\infty B}$).   

Finally, suppose that $A$ is any nonunital operator algebra and $J$ is an approximately unital 
closed ideal in $A$.   Then $J$ is an $M$-ideal in $A^1$ by \cite{ER}.   Also, by the uniqueness of the unitization 
of an operator algebra mentioned in the introduction, we have $A^1/J \cong (A/J)^1$ completely 
isometrically if $A/J$ is nonunital (see also \cite[Lemma 4.11]{Bord}).
Then the result follows again by applying the unital case to the canonical map from    
$A^1$ onto $(A/J)^1$.    If $A/J$ is unital we  can reduce to the case where it is not by the trick in the last paragraph.
\end{proof}

By the assertion about the norms in Theorem \ref{CSSW} and Corollary \ref{staTi}, we can lift elements in ${\mathfrak r}_{A/J}$ to elements in ${\mathfrak r}_A$ keeping the same norm, in the situations
considered in the corollary.

As we said,  these results may be viewed as  noncommutative peak interpolation or  noncommutative Tietze theorems.  For in the case that $A$ is a uniform algebra on a compact Hausdorff set $\Omega$,
the $M$-ideals $J$ are well known to be the closed ideals with a cai, and are exactly the functions in $A$ vanishing on some $p$-set $E \subset \Omega$
(see  \cite{SW3} and \cite[Theorem V.4.2]{HWW}).
Then $q_J$ is identifiable with the restriction map $f \mapsto f_{|E}$, and $A/J \cong \{  f_{|E} : f \in A \} \subset C(E)$.      
The lifting result in Theorems \ref{CSSW} and \ref{staTi} in this case 
 say that if $f \in A$ with $f(E) \subset C$ for a compact convex set $C$ in the plane, then there 
exists a function $g \in A$ which agrees with $f$ on $E$, which has norm $\Vert g \Vert_{\Omega} = \Vert f_{|E} \Vert_E$, and 
which has range $g(\Omega) \subset C$ (or $g(\Omega) \subset \hat{K}$ if conv$(f(E))$ is a line segment $K$, where $\hat{K}$ is a thin triangle 
given in advance,  one of whose sides is $K$).

\section{Banach algebras without cai}

If $A$ is a Banach algebra without a cai, or without any kind of bai,
 we briefly indicate here how to obtain nearly all the results from Sections
3, 4, and 7.   We give more details in a forthcoming conference proceedings survey article
\cite{B2015}, however 
the interested reader will have no trouble reconstructing this independently from the discussion below.
Namely, if $B$ is any unital Banach algebra
containing $A$, for example any unitization of $A$,
one can define  ${\mathfrak F}^B_A = \{ a \in A : \Vert 1_B - a \Vert
\leq 1 \}$, and define ${\mathfrak r}^B_A$ to be the set of $a \in A$ whose 
numerical range in $B$ is contained in the right half plane.    Also one can define
 ${\mathfrak F}_A$ (resp.\ ${\mathfrak r}_A$) to be the union of the ${\mathfrak F}^B_A$
(resp.\ ${\mathfrak r}^B_A$) over all $B$ as above.    Unfortunately it is not clear to us that
 ${\mathfrak F}_A$ and ${\mathfrak r}_A$ are always convex, which is needed in Sections 4 and 7 (indeed we often need them closed too there).   Of course  ${\mathfrak F}_A$ and ${\mathfrak r}_A$ are 
convex and closed if there is an `extremal' unitization $B$ of $A$
such that ${\mathfrak F}^B_A = {\mathfrak F}_A$ (resp.\ ${\mathfrak r}^B_A = {\mathfrak r}_A$).
This is the case with $B$ equal to the multiplier unitization if $A$ is  approximately unital, or more generally if the   left regular representation embeds $A$ isometrically in $B(A)$.

Most of the results in Sections 3, 4, and 7 of our paper then work without the approximately unital
hypothesis, if ${\mathfrak F}^B_A$ and ${\mathfrak r}^B_A$ 
are used.    In particular 
we mention the results \ref{clnth}--\ref{Bal}, \ref{hasbi}--\ref{whba},  \ref{wh}--\ref{supp3},
\ref{ws}, \ref{Aha3}--\ref{Aha}, and all  lemmas, theorems, and corollaries in Sections 4 and 7 not concerning $M$-approximately unital 
algebras. 
Every one of the statements of these results is still correct if one  drops the  approximately unital
hypothesis, but uses  ${\mathfrak F}^B_A$ and ${\mathfrak r}^B_A$ in place of
${\mathfrak F}_A$ and ${\mathfrak r}_A$.   Indeed the results just mentioned in Section 3 (and also
the first lemma in Section 4) are also correct for general Banach algebras 
if one uses ${\mathfrak F}_A$ or ${\mathfrak r}_A$ as defined in the last paragraph (the other results in Sections
4 and 7 would seem to  need  ${\mathfrak F}_A$ and ${\mathfrak r}_A$ (as defined in the last paragraph)
being closed and convex).   

Some of the results asserted in the last paragraph are obvious from the unital case of the result, and some follow by the obvious modification of
the given proof of the result.   See for example Corollary \ref{ottertoo} as an example of this.
However in some of these results  one also needs to know
that $\overline{EA} = \overline{EB}$ where $B$ is a unitization of $A$ and $E$ is
a subset of ${\mathfrak F}^B_A$ or ${\mathfrak r}^B_A$.  This follows from the following 
fact: if $x \in {\mathfrak r}_A$ as defined in the last paragraph then 
$$x \in \overline{xA} = \overline{{\rm ba}(x) \, A} =  \overline{xB},$$
for any unitization $B$ of $A$.   Indeed this is clear since by Cohen factorization 
$x \in {\rm ba}(x) = {\rm ba}(x)^2 \subset \overline{xA}$.   We also need to know that the ${\mathfrak F}$-transform, and $n$th roots, are  independent of 
the particular unitization used, but this is easy to see 
using the fact that all unitization norms are equivalent.

\bigskip

{\bf Acknowledgments.}    
We thank Charles Read for useful discussions, and for allowing us to 
take out some of the material in \cite{BRIII} for inclusion here.    We thank the referee for his careful reading of the manuscript, which was plagued by innumerable typos, and for his suggestions.
We were also  supported as participants
in the Thematic Program on Abstract Harmonic Analysis, Banach and Operator Algebras 2014 at the Fields Institute,
for which we thank the  Institute and the organizers of that program.  As we said earlier, the survey article
\cite{B2015} contains a few additional details on some of the material in the present
paper, as well as some small improvements found while this paper was in press.

\end{document}